\pdfoutput=1
\documentclass[12pt]{article}
\usepackage[T1]{fontenc}
\usepackage[utf8]{inputenc}
\usepackage[nomath]{lmodern}
\usepackage{amsmath}
\usepackage{amsthm}
\usepackage{amssymb}
\usepackage{geometry}
\usepackage{enumitem}
\usepackage[nottoc]{tocbibind} 
\usepackage[unicode, backref]{hyperref}

\usepackage{tikz}
	\usetikzlibrary{arrows.meta} 
	\usetikzlibrary{graphs} 
	\usetikzlibrary{shapes} 
	\usetikzlibrary{calc} 
\let \H \undefined 
\let \L \undefined 
\let \O \undefined 
\let \P \undefined 
\let \S \undefined 
\let \SS \undefined 
\let \int \indefined 
\let \C \undefined 
\let \G \undefined 
\let \U \undefined 

\let \Re \undefined 
\let \Im \undefined

\newcommand*\cdef{\newcommand*}

\cdef \A {\mathcal{A}}
\cdef \B {\mathcal{B}}
\cdef \C {\mathcal{C}}
\cdef \D {\mathcal{D}}
\cdef \E {\mathcal{E}}
\cdef \F {\mathcal{F}}
\cdef \G {\mathcal{G}}
\cdef \H {\mathcal{H}}
\cdef \I {\mathcal{I}}
\cdef \J {\mathcal{J}}
\cdef \K {\mathcal{K}}
\cdef \L {\mathcal{L}}
\cdef \M {\mathcal{M}}
\cdef \N {\mathcal{N}}
\cdef \O {\mathcal{O}}
\cdef \P {\mathcal{P}}
\cdef \R {\mathcal{R}}
\cdef \S {\mathcal{S}}
\cdef \T {\mathcal{T}}
\cdef \U {\mathcal{U}}
\cdef \V {\mathcal{V}}
\cdef \W {\mathcal{W}}
\cdef \Y {\mathcal{Y}}

\cdef \CC {\mathbb{C}}
\cdef \II {\mathbb{I}}
\cdef \NN {\mathbb{N}}
\cdef \PP {\mathbb{P}}
\cdef \QQ {\mathbb{Q}}
\cdef \RR {\mathbb{R}}
\cdef \SS {\mathbb{S}}
\cdef \ZZ {\mathbb{Z}}

\cdef \set [1]{%
	\{#1\}%
}

\cdef \tuple [1]{%
	\langle #1\rangle
}

\cdef \card [1]{%
	\lvert #1\rvert
}

\cdef \powset [1]{%
	\mathcal{P}(#1)%
}

\cdef \subsets [2]{%
	[#2]^{#1}%
}

\cdef \disunion {%
	\sqcup
}

\cdef \DisUnion {%
	\bigsqcup
}

\cdef \continuum {%
	\mathfrak{c}%
}

\cdef \maps {%
	\colon
}

\cdef \into {%
	\hookrightarrow
}

\cdef \onto {%
	\twoheadrightarrow
}

\cdef \from {%
	\leftarrow
}

\cdef \id {%
	\operatorname{id}%
}

\cdef \dom {%
	\operatorname{dom}%
}

\cdef \rng {%
	\operatorname{rng}%
}

\cdef \cod {%
	\operatorname{cod}%
}

\cdef \im [1]{%
	[#1]%
}

\cdef \coim [1]{%
	^\forall\im{#1}%
}

\cdef \inv {%
	^{-1}%
}

\cdef \preim [1]{%
	\inv\im{#1}%
}

\cdef \fiber [1]{%
	\inv(#1)%
}

\cdef \restr [1]{%
	\mathord{\upharpoonright}_{#1}%
}

\cdef \diag {%
	\mathbin\vartriangle
}

\cdef \Diag {%
	\bigtriangleup
}

\cdef \codiag {%
	\mathbin\triangledown
}

\cdef \CoDiag {%
	\bigtriangledown
}

\cdef \homeo {%
	\cong
}

\cdef \nothomeo {%
	\ncong
}

\cdef \clo [2][]{%
	\overline{#2}%
}

\cdef \cl {%
	\operatorname{cl}%
}

\cdef \int {%
	\operatorname{int}%
}

\cdef \topsum {%
	\oplus
}

\cdef \TopSum {%
	\sum
}

\cdef \dist {%
	\operatorname{d}%
}

\cdef \diam {%
	\operatorname{diam}%
}

\cdef \abs [1]{%
	\lvert #1\rvert
}

\cdef \Seq {%
	\operatorname{Seq}%
}

\cdef \Seqp {%
	\operatorname{Seq_p}%
}

\cdef \AllCompacta {%
    \mathbf{K}%
}

\cdef \BM {%
	\operatorname{\mathsf{BM}}%
}

\cdef \lcm {%
	\operatorname{lcm}%
}

\cdef \sdeg {%
	\operatorname{σ_0 deg}%
}

\cdef \Ob {%
	\operatorname{Ob}%
}

\cdef \geom [1]{%
	\lvert #1\rvert
}

\cdef \crn {%
	\operatorname{cr}%
}

\cdef \concat {%
	\frown
}

\cdef \supp {%
	\operatorname{supp}%
}

\cdef \clx {\operatorname{cl}^*}
\cdef \clcat {\operatorname{cl}^\text{\upshape cat}_\L}
\cdef \clloc {\operatorname{cl}^\text{\upshape loc}_\L}
\cdef \cllim [1][\K] {\operatorname{cl}^\text{\upshape lim}_{#1, \L}}
\cdef \clabs [1][\K] {\operatorname{cl}^\text{\upshape refl}_{#1, \L}}

\cdef \CatFont {%
	\mathsf
}

\cdef \Set {%
	\CatFont{Set}%
}

\cdef \Top {%
	\CatFont{Top}%
}

\cdef \MetU {%
	\CatFont{Met_u}%
}

\cdef \Met {%
	\CatFont{Met}%
}

\cdef \MetUS {%
	\CatFont{Met_{us}}%
}

\cdef \MCpt {%
	\CatFont{MCpt}%
}

\cdef \MCptS {%
	\CatFont{MCpt_s}%
}

\cdef \CMetU {%
	\CatFont{CMet_u}%
}

\cdef \CMetUS {%
	\CatFont{CMet_{us}}%
}

\cdef \CMet {%
	\CatFont{CMet}%
}

\cdef \MCont {%
	\CatFont{MCont}%
}

\cdef \MContS {%
	\CatFont{MCont_s}%
}

\cdef \PeanoS {%
	\CatFont{Peano_s}%
}

\cdef \CPolS {%
	\CatFont{CPol_s}%
}

\cdef \Ab {%
	\CatFont{Ab}%
}

\cdef \FinS {%
	\CatFont{Fin_s}%
}

\cdef \Vsup {%
	\bigvee
}

\cdef \Vinf {%
	\bigwedge
}

\cdef \iso {%
	\cong
}

\cdef \Aut {%
	\operatorname{Aut}%
}

\cdef \Age {%
	\operatorname{Age}%
}

\cdef \sAge {%
	\operatorname{σAge}%
}

\cdef \Re {%
	\operatorname{Re}%
}

\cdef \Im {%
	\operatorname{Im}%
}

\cdef \DefUnicode [2]{%
	\expandafter\cdef
		\csname u8:\detokenize{#1}\endcsname
		{#2}%
}

\cdef \UndefUnicode [1]{%
	\expandafter\let
		\csname u8:\detokenize{#1}\endcsname
		\undefined
}

\cdef \ShowUnicode [1]{%
	\expandafter\show
		\csname u8:\detokenize{#1}\endcsname
}

\UndefUnicode{ } 

\UndefUnicode{…} 
\DefUnicode{…}{\ldots} 

\UndefUnicode{•} 
\DefUnicode{•}{\ifmmode \bullet \else \IeC{\textbullet}\fi}

\UndefUnicode{¬} 
\DefUnicode{¬}{\lnot}

\UndefUnicode{×} 
\DefUnicode{×}{\times}

\UndefUnicode{±} 
\DefUnicode{±}{\pm}

\DefUnicode{α}{\alpha}
\DefUnicode{β}{\beta}
\DefUnicode{γ}{\gamma}
\DefUnicode{δ}{\delta}
\DefUnicode{ε}{\varepsilon}
\DefUnicode{ζ}{\zeta}
\DefUnicode{η}{\eta}
\DefUnicode{θ}{\theta}
\DefUnicode{ι}{\iota}
\DefUnicode{κ}{\kappa}
\DefUnicode{λ}{\lambda}
\DefUnicode{μ}{\mu}
\DefUnicode{ν}{\nu}
\DefUnicode{ξ}{\xi}
\DefUnicode{π}{\pi}
\DefUnicode{ρ}{\rho}
\DefUnicode{σ}{\sigma}
\DefUnicode{ς}{\varsigma}
\DefUnicode{τ}{\tau}
\DefUnicode{υ}{\upsilon}
\DefUnicode{φ}{\varphi}
\DefUnicode{χ}{\chi}
\DefUnicode{ψ}{\psi}
\DefUnicode{ω}{\omega}

\DefUnicode{Γ}{\Gamma}
\DefUnicode{Δ}{\Delta}
\DefUnicode{Θ}{\Theta}
\DefUnicode{Λ}{\Lambda}
\DefUnicode{Ξ}{\Xi}
\DefUnicode{Π}{\Pi}
\DefUnicode{Σ}{\Sigma}
\DefUnicode{Φ}{\Phi}
\DefUnicode{Ψ}{\Psi}
\DefUnicode{Ω}{\Omega}

\DefUnicode{ℂ}{\CC}
\DefUnicode{ℍ}{\HH}
\DefUnicode{ℕ}{\NN}
\DefUnicode{ℙ}{\PP}
\DefUnicode{ℚ}{\QQ}
\DefUnicode{ℝ}{\RR}
\DefUnicode{ℤ}{\ZZ}

\UndefUnicode{←} 
\DefUnicode{←}{\leftarrow}

\UndefUnicode{→} 
\DefUnicode{→}{\rightarrow}

\DefUnicode{↛}{\nrightarrow}

\DefUnicode{∀}{\forall}
\DefUnicode{∃}{\exists}
\DefUnicode{∄}{\nexists}
\DefUnicode{∅}{\emptyset}
\DefUnicode{∈}{\in}
\DefUnicode{∉}{\notin}
\DefUnicode{∋}{\owns}
\DefUnicode{∌}{\notowns}
\DefUnicode{∏}{\prod}
\DefUnicode{∐}{\coprod}
\DefUnicode{∑}{\sum}
\DefUnicode{∘}{\circ}
\DefUnicode{∞}{\infty}
\DefUnicode{∧}{\wedge}
\DefUnicode{∨}{\vee}
\DefUnicode{∩}{\cap}
\DefUnicode{∪}{\cup}
\DefUnicode{≠}{\neq}
\DefUnicode{≤}{\leq}
\DefUnicode{≥}{\geq}
\DefUnicode{⊆}{\subseteq}
\DefUnicode{⊇}{\supseteq}
\DefUnicode{⊈}{\nsubseteq}
\DefUnicode{⊉}{\nsupseteq}
\DefUnicode{⊊}{\subsetneq}
\DefUnicode{⊋}{\supsetneq}
\DefUnicode{⋃}{\bigcup}
\DefUnicode{⋂}{\bigcap}

\DefUnicode{∂}{\partial}

\DefUnicode{⋅}{\cdot}

\UndefUnicode{⟨} 
\DefUnicode{⟨}{\langle}

\UndefUnicode{⟩} 
\DefUnicode{⟩}{\rangle}

\DefUnicode{≈}{\approx}

\cdef \note [1]{%
	\textcolor{red}{[#1]}%
}

\cdef \link [1]{%
	\href{#1}{\nolinkurl{[#1]}}%
}

\cdef \arxivlink [1]{%
	\href{https://arxiv.org/abs/#1}{\nolinkurl{[arXiv:#1]}}%
}

\theoremstyle{plain}
\newtheorem{theorem}{Theorem}[section]
\newtheorem{proposition}[theorem]{Proposition}
\newtheorem{corollary}[theorem]{Corollary}
\newtheorem{lemma}[theorem]{Lemma}
\newtheorem{deftheorem}[theorem]{Definition and Theorem}

\theoremstyle{definition}
\newtheorem{definition}[theorem]{Definition}
\newtheorem{observation}[theorem]{Observation}
\newtheorem{remark}[theorem]{Remark}
\newtheorem{notation}[theorem]{Notation}
\newtheorem{example}[theorem]{Example}

\newtheorem{construction}[theorem]{Construction}

\newenvironment{talign*}
	{ \csname align*\endcsname}
	{\endalign}
\newenvironment{taligned}
	{ \aligned}
	{\endaligned}

\setlist{itemsep = 0pt}
\setlist[enumerate]{leftmargin=*} 
\setlist[enumerate, 1]{label=\upshape (\roman*), ref=(\roman*)}

\DeclareEmphSequence{\itshape,\bfseries}

\hypersetup{
	pdftitle = {Hereditarily indecomposable continua as generic mathematical structures},
	pdfauthor = {Adam Bartoš, Wiesław Kubiś},
	colorlinks,
	linkcolor = [rgb]{0, 0, 0.7},
	citecolor = [rgb]{0, 0.7, 0},
	urlcolor = [rgb]{0.4, 0.2, 0},
	linktoc = section,
}

\title{Hereditarily indecomposable continua \\ as generic mathematical structures}
\author{Adam Bartoš\footnote{Research of both authors was supported by GA ČR (Czech Science Foundation) grant EXPRO 20-31529X and by the Czech Academy of Sciences (RVO 67985840).} \\
	\small \href{mailto:bartos@math.cas.cz}{\nolinkurl{bartos@math.cas.cz}}
\and Wiesław Kubiś\footnotemark[\value{footnote}] \\
	\small \href{mailto:kubis@math.cas.cz}{\nolinkurl{kubis@math.cas.cz}}
}
\date{
	{\small Institute of Mathematics, Czech Academy of Sciences, \\
	Žitná 25, 115 67 Prague, Czech Republic} \\[3ex]
	January 2026
}

\begin{document}
	
	\maketitle
	
	\begin{abstract}
		We characterize the pseudo-arc as well as $P$-adic pseudo-solenoids (for a set of primes $P$) as generic structures, arising from a natural game in which two players alternate in building an inverse sequence of surjections. The second player \emph{wins} if the limit of this sequence is homeomorphic to a concrete (fixed in advance) space, called \emph{generic} whenever the second player has a winning strategy.
		
		For this purpose, we develop a new robust approximate Fraïssé theory in the context of \emph{MU-categories}, a generalization of metric-enriched categories, suitable for working directly with continuous maps between metrizable compacta.
		Our framework extends both the classical and projective Fraïssé theories.
		
		We reprove the Fraïssé-theoretic characterization of the pseudo-arc and we realize every $P$-adic pseudo-solenoid as a Fraïssé limit of a suitable category of continuous surjections on the circle.
		Moreover, we show that, when playing the game with continuous surjections between non-degenerate Peano continua, the pseudo-arc is always generic, while
		the universal pseudo-solenoid is generic over all surjections between circle-like continua.
		This gives a complete classification of generic continua over full non-trivial subcategories of connected polyhedra with continuous surjections.
		
		\vspace{1em}
		


	\end{abstract}
	
	\newpage
	\tableofcontents

\section{Introduction}

Classically, Fraïssé theory, coming back to Fraïssé~\cite{Fraisse54}, is concerned with ultrahomogeneous structures in model theory~\cite[Chapter~7]{Hodges93}.
There the homogeneity is injective, i.e. with respect to (finitely generated) substructures.
In 2006, Irwin and Solecki~\cite{IS06} introduced projective Fraïssé theory to study metrizable compact structures and their projective homogeneity.
They realized a pre-space of the pseudo-arc as a projective Fraïssé limit, and gave an approximate Fraïssé-theoretic characterization of the pseudo-arc.
Here the term “pre-space” means the Cantor space together with a special closed equivalence relation.
Since then, many other metrizable compacta were realized as quotients of pre-spaces that are projective Fraïssé limits, for example the Lelek fan by Bartošová and Kwiatkowska~\cite{BK15}, the Menger curve by Panagiotopoulos and Solecki~\cite{PS22}, the so called Fraïssé fence by Basso and Camerlo~\cite{BC21}, the generalized Ważewski dendrite $D_3$, the Mohler–Nikiel universal dendroid, and a new Kelley dendroid with a dense set of endpoints by Charatonik, Kwiatkowska, Roe, and Yang~\cite{ChKRY23}, and a new one-dimensional Kelley continuum containing the pseudo-arc as well as the universal pseudo-solenoid by Charatonik, Kwiatkowska, and Roe~\cite{ChKR25}.

On the other hand, the second author~\cite{Kubis13} introduced a framework for approximate Fraïssé theory based on metric-enriched categories and realized the pseudo-arc as a Fraïssé limit directly, i.e. without taking a quotient of a pre-space.
In the same spirit, the second author and Kwiatkowska~\cite{KK17} realized the Lelek fan and the Poulsen simplex as Fraïssé limits.
We should also mention that already Mioduszewski~\cite{Mioduszewski62} constructed the pseudo-arc as a Fraïssé limit (without using that terminology) of special countable categories of copies of the unit interval and piecewise linear maps, and proved the surjective universality among arc-like continua.
Later, Rogers~\cite{Rogers70} used Mioduszewski's formalism to special countable categories of copies of the unit circle and piecewise linear maps of limited degree, proved the amalgamation property for these categories, and constructed pseudo-solenoids (called pseudo-circles by Rogers) universal for the corresponding categories of circle-like continua.

This approximate Fraïssé theory, which is further developed in the present paper, can be viewed two ways: (1) as a setup for Fraïssé theory of metric structures, and (2) as an alternative approach to projective Fraïssé theory.
From the viewpoint~(1), the setup complements approaches based on continuous model theory, in particular the line of research starting in the thesis of Schoretsanitis~\cite{Schoretsanitis07}, standardized by Ben Yaacov~\cite{BenYaacov15}, and refined for the purpose of applications in C*-algebras by Masumoto~\cite{Masumoto20}.
The enriched categorical approach abstracts certain aspects of the model-theoretic setup (namely, approximate commutativity of diagrams and approximate morphisms), so that theory can be directly applied in a broader context, including projective setting, embedding-projection pairs, and comma categories.
In the present paper we do not work with approximate morphisms, while we provide a robust setup for approximate commutativity, suited not only for non-expansive surjections but also for uniformly continuous surjections.

From the viewpoint~(2), the key feature of the setup is that the Fraïssé limit is the space itself, bypassing the quotient construction.
Very recently, another alternative approach to projective Fraïssé theory was developed by the first author, Bice, and Vignati~\cite{BBV25} and \cite{BBVtwo}.
In this approach, we keep finite graphs as our building blocks, but we allow relations/multivalued maps as morphisms, while the standard inverse limit construction is replaced by taking the so-called spectrum of the $ω$-poset corresponding to a sequence of graphs.
We believe it is possible to combine the strengths of the different approaches to projective Fraïssé theory within a unified framework.

\bigskip

In this work we are primarily interested in hereditarily indecomposable continua, namely the pseudo-arc and the pseudo-solenoids.
At the very end of the paper~\cite{IS06}, Irwin and Solecki suggest to realize the universal pseudo-solenoid as a projective Fraïssé limit similarly to the pseudo-arc, extending the work of Rogers~\cite{Rogers70}.
This was done by Irwin in his thesis~\cite{Irwin07}, mentioning a possibility of extending his result to other pseudo-solenoids.
This is what we do in the present paper directly, without using pre-spaces.

We develop a robust approximate framework for Fraïssé theory, extending both the classical and projective Fraïssé theory, that is also suitable for working directly with continuous maps between metrizable compacta.
Then we reprove the Fraïssé-theoretic characterization of the pseudo-arc in our framework to demonstrate it, and we realize every $P$-adic pseudo-solenoid (for any set of primes $P$) as a Fraïssé limit of a suitable category of continuous surjections on the circle.
Moreover, we consider the abstract Banach–Mazur game and the notion of generic object introduced by Krawczyk and the second author~\cite{KK21} for finitely generated structures and later generalized by the second author~\cite{Kubis22} to arbitrary categories.
Since every Fraïssé limit is a generic object and a generic object is unique, this gives yet another characterization of the pseudo-arc and of the pseudo-solenoids.

Our theory culminates in the complete classification of generic continua over full non-trivial subcategories of polyhedra with continuous surjections: the pseudo-arc and the universal pseudo-solenoid (see Theorem~\ref{thm:polyhedra} below). We also show that playing the Banach--Mazur game with arbitrary surjections between a fixed class of non-degenerate Peano continua, the second player has a strategy leading to the pseudo-arc (see Theorem~\ref{thm:generic_Peano} below).

The theory of MU-categories and their Fraïssé limits can certainly be applied to more special topological categories, for example, categories of retractions between certain metric spaces (see e.g.~\cite{KK17}) or categories of homomorphic embeddings between certain topological algebras. In some situations one can work in categories enriched over metric spaces, which are special cases of MU-categories (see e.g.~\cite{Kubis13}).
Recently, applications in C*-algebras using a closely related framework were obtained by Cantier and Vilalta~\cite{CV24}.

\subsection{Summary of the structure and results}

Here we summarize the article and describe our main results.
Section~\ref{sec:continua} is a mini-survey on crookedness and the pseudo-arc.
In Section~\ref{sec:generic} and \ref{sec:fraisse} we build our general theory while demonstrating the concepts on the known case of the pseudo-arc, with several technical details postponed to Appendix~\ref{sec:appendix}.
Finally, in Section~\ref{sec:circlelike} we obtain a new application.

\medskip\noindent
\textbf{Continua and crookedness} (Section~\ref{sec:continua}).
We recall the notions of hereditarily indecomposable continua and crookedness.
We also recall the standard construction of the pseudo-arc as an inverse limit of a crooked sequence of interval maps.
\begin{itemize}
	\item We summarize known facts regarding $ε$-crooked maps and generalize them to the context of metrizable compacta (Proposition~\ref{thm:crooked_calculus} and Theorem~\ref{thm:crooked_limit}), and show how the notion of $ε$-crooked map simplifies to a well-known form for Peano continua and compact graphs (Theorem~\ref{thm:graph_crookedness}).
	\item We give a direct proof of the crookedness factorization theorem (Theorem~\ref{thm:crooked_factorization}) that every continuous surjection on the unit interval $ε$-factorizes through every sufficiently crooked continuous surjection.
		This is later used to observe that every crooked sequence is a Fraïssé sequence (in the context of continuous surjections on the unit interval), and so the uniqueness of a hereditarily indecomposable arc-like continuum (Bing's theorem) also follows from the uniqueness of a Fraïssé limit (Remark~\ref{thm:Bing_reproved}).
\end{itemize}

\medskip\noindent
\textbf{Generic objects} (Section~\ref{sec:generic}).
We recall the abstract Banach–Mazur game played in a category and the notion of a generic object.
\begin{itemize}
	\item The pseudo-arc is generic over every dominating subcategory of Peano continua and continuous surjections (Theorem~\ref{thm:generic_Peano}).
		This is proved first in the discrete setting and then in the approximate setting (Remark~\ref{rm:generic_Peano}).
	\item Our approximate setting is realized via the notion of \emph{MU-category} (Definition~\ref{def:MU-category}), a certain generalization of metric-enriched categories, suitable for our applications regarding limits of inverse sequences of metrizable compacta.
	\item We prove abstract versions of Brown's approximation theorem (Corollary~\ref{thm:Brown}) and the back and forth construction (Corollary~\ref{thm:back_and_forth}) for locally complete MU-categories.
	These results form the base for stability of generic objects under dominating subcategories in the approximate setting as well as for uniqueness results in the next section.
\end{itemize}

\medskip\noindent
\textbf{Fraïssé theory} (Section~\ref{sec:fraisse}).
We build a framework for abstract approximate Fraïssé theory in the context of MU-categories.
The usual setup consists of a pair $\K ⊆ \L$ of MU-categories of “small” and “large” objects, respectively.
\begin{itemize}
	\item Our framework generalizes abstract discrete Fraïssé theory, which in turn generalizes both the classical Fraïssé theory for first-order structures and embeddings as well as the projective Fraïssé theory for topological first-order structures and quotients developed by Irwin and Solecki (see Remarks~\ref{rm:discrete_fraisse}, \ref{rm:classical_fraisse}, and \ref{rm:projective_fraisse}).
	\item The core of the theory consists of Theorem~\ref{thm:fraisse_limit} characterizing the unique Fraïssé limit equivalently as a cofinal homogeneous object, cofinal projective object, and an $\L$-limit of a Fraïssé sequence in $\K$, and of Theorem~\ref{thm:fraisse_category} characterizing when a Fraïssé sequence exists in $\K$ (which depends only on $\K$).
	\item The characterization of the Fraïssé limit holds under several assumptions leading to the notion of \emph{free completion} (Definition~\ref{def:free_completion}) – namely, $\tuple{\K, \L}$ is a free completion if $\L$ as an MU-category essentially arose by freely and continuously adding limits of sequences to $\K$.
		We discuss in detail how to prove that a given pair is a free completion – using the construction of the \emph{$σ$-closure} (Definition~\ref{def:sigma_closure}) under the assumption of \emph{$σ$-consistency} (Definition~\ref{def:sigma_consistent}).
	\item The pseudo-arc is a Fraïssé limit in our framework.
		This reproves the characterization by Irwin and Solecki.
		The established characterizations of the pseudo-arc are gathered in Theorem~\ref{thm:pseudo-arc}.
	\item We obtain a complete classification of full categories of connected polyhedra with a Fraïssé limit / generic object / cofinal object (Theorem~\ref{thm:polyhedra}).
\end{itemize}

\medskip\noindent
\textbf{Circle-like continua and pseudo-solenoids} (Section~\ref{sec:circlelike}).
The last section contains the main application of our theory in this paper.
The category of continuous surjections on the circle does not have a Fraïssé limit, but we can restrict to the subcategory $\S_P$ of those maps whose degree uses only primes from a fixed set of primes $P$.
\begin{itemize}
	\item There is a Fraïssé limit of $\S_P$ for every set of primes $P$ (Theorem~\ref{thm:fraisse_circle-like}) and the limit is the $P$-adic pseudo-solenoid (Theorem~\ref{thm:fraisse_pseudo-solenoid}).
	\item We introduce a type functor (Construction~\ref{con:type_functor}) refining a categorization of circle-like continua and of continuous surjections.
		The type functor can be viewed as an extension of the degree map as well as a “skeletal rigidification” of the first Čech cohomology (see Remark~\ref{rm:Cech_cohomology}).
	\item Every pseudo-solenoid (not necessarily $P$-adic) is generic with respect to its type (Proposition~\ref{thm:pseudo-solenoids_generic}).
		To prove this statement we consider a modification of the Banach–Mazur game played with circle-like continua (Definition~\ref{def:modified_Banach--Mazur}).
\end{itemize}

\medskip\noindent
\textbf{More on $σ$-closure and $σ$-consistency} (Appendix~\ref{sec:appendix}).
We give a general definition of the $σ$-closure and relate it to several other definitions present in the literature.
\begin{itemize}
	\item In a $σ$-consistent situation, all alternative simplified definitions of the $σ$-closure agree with the general one, and this to some extent characterizes $σ$-consistency (Proposition~\ref{thm:AE_sigma_consistent} and Corollary~\ref{thm:B_sigma}).
	\item We give a $σ$-inconsistent example of a Fraïssé category such that the limit of a Fraïssé sequence is not homogeneous (Example~\ref{ex:non_sigma_consistent}).
\end{itemize}

\subsection{Preliminaries} \label{sec:preliminaries}

For a subset $A ⊆ X$ of a topological space, $\clo{A}$, $\int(A)$, and $∂A$ denote the closure, the interior, and the boundary, respectively.

An \emph{$∞$-metric space} is a generalization of a metric space allowing infinite distances.
This is a quite innocent generalization since $d(x, y) < ∞$ in an $∞$-metric space $X$ is an equivalence relation, and the equivalence classes are metric subspaces of $X$.
The set of all functions $f\maps X \to Y$ to a metric space $Y$ is an $∞$-metric space when endowed with the sup-metric $d(f, g) := \sup_{x ∈ X} d(f(x), g(x))$.
In any $∞$-metric space for $ε > 0$ we sometimes write $x ≈_ε y$ as a shortcut for $d(x, y) < ε$.
So for functions with the sup-metric, $f ≈_ε g$ means $\sup_{x ∈ X} d(f(x), g(x)) < ε$.
Similarly, we write $x ≈_{≤ε} y$ as a shortcut for $d(x, y) ≤ ε$.
A map $f\maps X \to Y$ between $∞$-metric spaces is called \emph{$\tuple{ε, δ}$-continuous} if $δ > 0$ is a witness for uniform continuity for $ε > 0$, i.e. $x ≈_δ y$ implies $f(x) ≈_ε f(y)$ for every $x, y ∈ X$.

We denote categories by calligraphic letters $\K$, $\L$, …
A category $\K$ is identified with the class of all $\K$-maps, so it makes sense to write $f ∈ \K$.
The class of all $\K$-objects is denoted by $\Ob(\K)$.
For a $\K$-map $f$ the notation $f\maps X \to Y$ means that $f$ is a member of the hom-set $\K(X, Y)$.
When convenient we equivalently write $f\maps Y \from X$.
Recall that a subcategory $\K ⊆ \L$ is called \emph{full} if $\K(X, Y) = \L(X, Y)$ for all $\K$-objects $X, Y$, and that it is called \emph{wide} if $\Ob(\K) = \Ob(\L)$.
We gather a list of names of several standard categories used in the text in Table~\ref{tab:categories}.

\begin{table}[ht]
	\centering
	\begin{tabular}{ll}
		$\MetU$ & metric spaces and uniformly continuous maps \\
		$\Met$ & metric spaces and non-expansive maps \\
		$\CMetU$ & complete metric spaces and uniformly continuous maps \\
		$\CMet$ & complete metric spaces and non-expansive maps \\
		$\MCpt$ & metrizable compact spaces and continuous maps \\
		$\MCptS$ & non-empty metrizable compact spaces and continuous surjections \\
		$\MContS$ & non-empty metrizable continua and continuous surjections \\
		$\PeanoS$ & (non-empty) Peano continua and continuous surjections \\
		$\CPolS$ & non-empty connected (compact) polyhedra and continuous surjections \\
		$\I$ & the unit interval $\II$ and continuous surjections \\
		$σ\I$ & arc-like continua and continuous surjections \\
		$\S$ & the unit circle $\SS$ and continuous surjections \\
		$σ\S$ & circle-like continua and continuous surjections 
	\end{tabular}
	\caption{Several standard categories used in the text.}
	\label{tab:categories}
\end{table}

\newpage 

Recall that an \emph{inverse sequence} $\tuple{X_*, f_*}$ in a category $\K$ consists of a sequence $X_* = \tuple{X_n}_{n ∈ ω}$ of $\K$-objects and of a sequence $f_* = \tuple{f_n: X_n \from X_{n + 1}}_{n ∈ ω}$ of $\K$-maps.
Moreover, for every $n ≤ n'$ we have the composition $f_{n, n'} = f_n ∘ f_{n + 1} ∘ \cdots ∘ f_{n' - 1}\maps X_n \from X_{n'}$.
In particular, $f_{n, n} = \id_{X_n}$ and $f_{n, n + 1} = f_n$ for every $n ∈ ω$.
Of course, we have $f_{n, n'} ∘ f_{n', n''} = f_{n, n''}$ for every $n ≤ n' ≤ n''$.
In other words, $\tuple{X_*, f_*}$ is a functor $\tuple{ω, ≥} \to \K$.
We often say just \emph{sequence in $\K$}, and we often denote the sequence just by $f_*$.
By a \emph{subsequence} of $\tuple{X_*, f_*}$ we mean a sequence of the form $\tuple{X_{n_*}, f_{n_*}} := \tuple{\tuple{X_{n_k}}_{k ∈ ω}, \tuple{f_{n_k, n_{k + 1}}}_{k ∈ ω}}$ for a strictly increasing map $n_*\maps ω \to ω$.

A \emph{cone} $\tuple{Y, γ_*}$ for the sequence $\tuple{X_*, f_*}$ consists of a $\K$-object $Y$ and of a sequence $γ_* = \tuple{γ_n\maps X_n \from Y}_{n ∈ ω}$ of $\K$-maps such that $f_{n, n'} ∘ γ_{n'} = γ_n$ for every $n ≤ n' ∈ ω$.
The sequence $\tuple{X_*, f_*}$ may have a \emph{limit} $\tuple{X_∞, f_{*, ∞}}$ in $\K$, which is a cone for $f_*$ such that for any cone $\tuple{Y, γ_*}$ for $f_*$ there is a unique $\K$-map $γ_∞\maps X_∞ \from Y$ such that $f_{n, ∞} ∘ γ_∞ = γ_n$ for every $n ∈ ω$.
The limit is unique up to a canonical isomorphism.
Recall that the limit in the category $\Top$ of topological spaces and continuous maps consists of the set 
\[\textstyle
	X_∞ := \set{x_* ∈ ∏_{n ∈ ω} X_n: x_n = f_{n, n'}(x_{n'}) \text{ for every } n ≤ n' ∈ ω}
\]
endowed with the subspace topology of the product topology, and of the limit cone maps $f_{n, ∞}\maps X_∞ \to X_n$ that are the restrictions of the projections.
In the category $\MetU$ of metric spaces and uniformly continuous maps, the limit is the limit in $\Top$ endowed with a metric that is a suitable combination of the metrics on the spaces $X_n$, so that the projections are uniformly continuous.

\section{Continua and crookedness}
\label{sec:continua}

In this self-contained section we recall the key notion of a \emph{hereditarily indecomposable continuum}, which is the target for our applications, and revise the closely related notion of \emph{crookedness} at various levels of generality.
Crookedness can be made a quantitative notion that behaves nicely with respect to inverse limits, which is the core of standard constructions of hereditarily indecomposable continua.
This is made precise in Theorem~\ref{thm:crooked_limit}.

By a \emph{continuum} we mean a connected Hausdorff compactum, typically metrizable for our applications.
Recall that a continuum is \emph{indecomposable} if it is not the union of two of its proper subcontinua,
and that a continuum (or more generally a Hausdorff compactum) $X$ is \emph{hereditarily indecomposable} if each of its subcontinua is indecomposable, or in other words, for all subcontinua $C, D ⊆ X$ we have $C ⊆ D$ or $D ⊆ C$ or $C ∩ D = ∅$.
By $\II$ we denote the unit interval $[0, 1]$, and by $\I$ we denote the category of all continuous surjections of $\II$.
Recall that a (necessarily metrizable) continuum is called \emph{arc-like} if it is the limit of an inverse sequence of continuous surjections between copies of the unit interval, see e.g. \cite[II.5]{Nadler92}.
Let $σ\I$ denote the category of all arc-like continua and all continuous surjections.
By the classical theorem of Bing~\cite{Bing51} there exists a unique (up to homeomorphism) hereditarily indecomposable arc-like continuum, called the \emph{pseudo-arc}.
We denote the pseudo-arc by $\PP$.
It follows from Bing's theorem, that every non-degenerate subcontinuum of the pseudo-arc is homeomorphic to the pseudo-arc.

\begin{definition}
	Let $X$ be a Hausdorff compactum.
	\begin{itemize}
		\item A quadruple $\tuple{A, B, U_A, U_B}$ of subsets of $X$ is called \emph{admissible} if $A, B$ are closed and $U_A, U_B$ are their respective open neighborhoods.
		\item The space $X$ is \emph{crooked at} a (not necessarily admissible) quadruple $\tuple{A, B, U_A, U_B}$ if there is a closed cover $X = F_A ∪ H ∪ F_B$ such that $A ⊆ F_A$, $B ⊆ F_B$, $F_A ∩ H ⊆ U_B$, $F_B ∩ H ⊆ U_A$, and $F_A ∩ F_B ⊆ U_A ∩ U_B$, see Figure~\ref{fig:crooked_space}.
		\item The space $X$ is \emph{crooked} if it is crooked at every admissible quadruple.
	\end{itemize}
	Let $f\maps X \to Y$ be a continuous map between Hausdorff compacta.
	\begin{itemize}
		\item The map $f$ is \emph{crooked at} an (admissible) quadruple $\tuple{A, B, U_A, U_B}$ in $Y$ if $X$ is crooked at the (admissible) quadruple $\tuple{A, B, U_A, U_B}\restr{f} := \tuple{f\preim{A}, f\preim{B}, f\preim{U_A},\allowbreak f\preim{U_B}}$.
		Hence, the space $X$ is crooked at a quadruple if and only if $\id_X$ is crooked at that quadruple.
		\item The map $f$ is \emph{crooked} if it is crooked at every admissible quadruple.
	\end{itemize}
	Suppose the space $Y$ is metric and $ε > 0$.
	\begin{itemize}
		\item The map $f$ is \emph{$ε$-crooked at $\tuple{A, B}$}, where $A, B ⊆ Y$, if it is crooked at $\tuple{A, B, N_ε(A),\allowbreak N_ε(B)}$, where $N_ε(A)$ denotes $\set{y ∈ Y: d(y, A) < ε}$.
		\item The map $f$ is \emph{$ε$-crooked} if it is $ε$-crooked at every pair $\tuple{A, B}$ of closed subsets of $Y$.
	\end{itemize}
	We write $\tuple{A', B', U', V'} ≤ \tuple{A, B, U, V}$ for two quadruples in the same space if $A' ⊇ A$, $B' ⊇ B$, $U' ⊆ U$, and $V' ⊆ V$, so if a space or a map is crooked at $\tuple{A', B', U', V'}$, then it is also crooked at $\tuple{A, B, U, V}$.
\end{definition}

\begin{figure}[!ht]
	\centering
	
	\begin{tikzpicture}[
			x = {(0.5, 0)},
			y = {(0, 0.5)},
			semithick,
		]
		
		\node (A) at (-4, 1.5) {$A$};
		\node (B) at (4, -1.5) {$B$};
		
		\begin{scope}[x = {(1.15, 0)}, y = {(0, 1.15)}]
			\draw ($(A)+(-1,-1)$) rectangle ($(A)+(1,1)$);
			\draw ($(B)+(-1,-1)$) rectangle ($(B)+(1,1)$);
		\end{scope}
		
		\draw (-4, 0) ellipse (3 and 4) [dashed];
		\draw (4, 0) ellipse (3 and 4) [dashed];
		
		\draw (-5.5, 3) -- (5, 3) -- (5, 2) -- (-2.5, 2) -- (-2.5, 0) -- (-5.5, 0) -- cycle;
		\draw (5.5, -3) -- (-5, -3) -- (-5, -2) -- (2.5, -2) -- (2.5, 0) -- (5.5, 0) -- cycle;
		\draw (3.5, 2.5) -- (4.5, 2.5) -- (4.5, 0.9) -- (-3.5, -1.6) 
			-- (-3.5, -2.5) -- (-4.5, -2.5) -- (-4.5, -0.9) -- (3.5, 1.6) -- cycle;
		
		\node at (0, 2.5) {$F_A$};
		\node at (0, -2.5) {$F_B$};
		\node at (0, 0) {$H$};
		\node at (-6, -1) {$U_A$};
		\node at (6, 1) {$U_B$};
	\end{tikzpicture}
	
	\caption{Crookedness of a space $X = F_A ∪ H ∪ F_B$ at a quadruple $\tuple{A, B, U_A, U_B}$.}
	\label{fig:crooked_space}
\end{figure}

\begin{observation}
	In the decomposition $X = F_A ∪ H ∪ F_B$ witnessing that $X$ is crooked at $\tuple{A, B, U_A, U_B}$, we may without loss of generality take $H = \clo{X \setminus (F_A ∪ F_B)}$.
	Then $F_A ∩ H ⊆ ∂F_A ⊆ (F_A ∩ H) ∪ (F_A ∩ F_B)$ and similarly for $F_B$, so given $F_A ∩ F_B ⊆ U_A ∩ U_B$, we have $F_A ∩ H ⊆ U_B$ if and only if $∂F_A ⊆ U_B$.
	Hence, $X$ is crooked at $\tuple{A, B, U_A, U_B}$ if and only if there are closed sets $F_A ⊇ A$ and $F_B ⊇ B$ such that $∂F_A ⊆ U_B$, $∂F_B ⊆ U_A$, and $F_A ∩ F_B ⊆ U_A ∩ U_B$.
	We shall use these witnessing pairs whenever convenient.
\end{observation}

\begin{remark} \label{rm:hereditarily_indecomposable}
	This general notion of crookedness for spaces was introduced by Krasinkie\-wicz and Minc~\cite{KM76}, \cite{KM77}.
	It follows from \cite[Theorem~3.4]{KM77} that a continuum (even a Hausdorff compactum) is hereditarily indecomposable if and only if it is crooked.
	The idea to consider crookedness of continuous maps comes from Maćkowiak~\cite{Mackowiak85}.
	
	In fact, Krasinkiewicz and Minc use a slightly different “strict” version of crookedness where the closed sets $A, B$ as well as the witnessing sets $F_A, F_B$ are required to be disjoint.
	The following lemmata show that in fact the definitions are equivalent.
	We have modified the definition since sometimes it is useful to consider non-disjoint pairs, e.g. in order to obtain a compact hyperspace for admissible closed pairs $\tuple{A, B}$.
\end{remark}

\begin{lemma} \label{thm:disjoint_enough}
	A continuous map $f\maps X \to Y$ between Hausdorff compacta is crooked at a (not necessarily admissible) quadruple $\tuple{A, B, U, V}$ if and only if it is crooked at $\tuple{A \setminus V, B \setminus U, U, V}$.
	Hence, when showing that $f$ is crooked (or $ε$-crooked when $Y$ is metric), it is enough to consider disjoint closed pairs $\tuple{A, B}$.
	
	\begin{proof}
		Since $f\preim{A \setminus V} = f\preim{A} \setminus f\preim{V}$ and $f\preim{B \setminus U} = f\preim{B} \setminus f\preim{U}$, to prove the first part we can work in $X$, i.e. suppose that $f = \id_X$.
		Let $A' := A \setminus V$ and $B' := B \setminus U$.
		Clearly, if $f$ is crooked at $\tuple{A, B, U, V}$, then it is crooked also at $\tuple{A', B', U, V}$.
		On the other hand, let $\tuple{F'_A, F'_B}$ be a witness for crookedness at $\tuple{A', B', U, V}$.
		We show that $\tuple{F_A, F_B} := \tuple{F'_A ∪ A, F'_B ∪ B}$ is a witness for $\tuple{A, B, U, V}$.
		Let $H := \clo{X \setminus (F'_A ∪ F'_B)}$.
		Note that $F_A = F'_A ∪ (A ∩ V) ⊆ F'_A ∪ (U ∩ V)$ since $F'_A$ covers $A' = A \setminus V$.
		Hence, $F_A ∩ H ⊆ (F'_A ∩ H) ∪ (U ∩ V) ⊆ V$.
		Similarly, $F_B ∩ H ⊆ U$.
		Finally, $F_A ∩ F_B ⊆ (F'_A ∩ F'_B) ∪ (U ∩ V) ⊆ U ∩ V$.
		
		Regarding $ε$-crookedness, we have shown that $f$ is crooked at $\tuple{A, B, N_ε(A), N_ε(B)}$ if and only if it is crooked at $\tuple{A', B', N_ε(A), N_ε(B)}$, which follows from being crooked at $\tuple{A', B', N_ε(A'), N_ε(B')}$, the latter being $ε$-crookedness at the disjoint pair $\tuple{A', B'}$.
	\end{proof}
\end{lemma}

\begin{lemma} \label{thm:disjoint_witness}
	If a Hausdorff compactum is crooked at an admissible quadruple $\tuple{A, B,\allowbreak U_A, U_B}$, then there is a witness $\tuple{F_A, F_B}$ such that $F_A ∩ F_B = A ∩ B$. So for disjoint pairs $\tuple{A, B}$ we may have disjoint $\tuple{F_A, F_B}$.
	
	\begin{proof}
		Let $\tuple{G_A, G_B}$ be a witness for $\tuple{A, B, U_A, U_B}$.
		Using normality, there is a neighborhood $V_A$ such that $A ∪ ∂G_B ∪ (G_A ∩ G_B) ⊆ V_A ⊆ \clo{V_A} ⊆ U_A$. Similarly there is a suitable neighborhood $V_B$, and we have that $\tuple{G_A, G_B}$ witnesses crookedness also at $\tuple{A, B, V_A, V_B}$.
		The set $G'_A := G_A \setminus V_B$ contains $A \setminus V_B$ , and we have $∂G'_A ⊆ \clo{V_B}$ since $∂G_A ⊆ V_B$.
		Similarly $G'_B := G_B \setminus V_A$ contains $B \setminus V_A$, and we have $∂G'_B ⊆ \clo{V_A}$.
		We put $\tuple{F_A, F_B} := \tuple{G'_A ∪ A, G'_B ∪ B}$.
		We have $G'_A ∩ G'_B = G'_A ∩ B = A ∩ G'_B = ∅$ since $G_A ∩ G_B ⊆ V_A ∩ V_B$ and since $A ⊆ V_A$ and $B ⊆ V_B$.
		It follows that $F_A ∩ F_B = A ∩ B$.
		Since $F_A = G'_A ∪ (A ∩ V_B)$, we have $∂F_A ⊆ ∂G'_A ∪ ∂(A ∩ V_B) ⊆ \clo{V_B} ⊆ U_B$.
		Similarly, $∂F_B ⊆ U_A$.
	\end{proof}
\end{lemma}

Let us now note some basic properties of ($ε$-)crookedness.

\begin{observation}
	Let $f\maps X \to Y$ be a continuous map between Hausdorff compacta.
	If a triple $F_A ∪ H ∪ F_B = Y$ witnesses the crookedness of $Y$ at some admissible quadruple $\tuple{A, B, U, V}$, then $f\preim{F_A} ∪ f\preim{H} ∪ f\preim{F_B} = X$ witnesses the crookedness of $X$ at the admissible quadruple $\tuple{A, B, U, V}\restr{f}$.
\end{observation}
	
\begin{proposition} \label{thm:crooked_calculus}
	Let $f\maps X \to Y$ and $g\maps Y \to Z$ be continuous maps between metric compacta.
	\begin{enumerate}
		\item If $g$ is $ε$-crooked, then $g ∘ f$ is $ε$-crooked.
		\item If $g$ is $\tuple{ε, δ}$-continuous and $f$ is $δ$-crooked, then $g ∘ f$ is $ε$-crooked.
	\end{enumerate}
	Let $f, g\maps X \to Y$ be continuous maps between metric compacta.
	\begin{enumerate}[resume]
		\item If $f$ is $ε$-crooked, and $g ≈_δ f$, then $g$ is $(ε + 2δ)$-crooked.
	\end{enumerate}
	
	\begin{proof}
		(i) follows from the previous observation.
		(ii) The $\tuple{ε, δ}$-continuity of $g$ means that $N_δ(g\preim{A}) ⊆ g\preim{N_ε(A)}$, so if $f$ is $δ$-crooked at $\tuple{g\preim{A}, g\preim{B}}$, then $g ∘ f$ is $ε$-crooked at $\tuple{A, B}$.
		(iii) Let $\tuple{F_A, F_B}$ witness that $f$ is $ε$-crooked at $\tuple{\clo{N_δ(A)}, \clo{N_δ(B)}}$.
		Then it also witnesses that $g$ is $(ε + 2δ)$-crooked at $\tuple{A, B}$, as we show now.
		For every $x ∈ g\preim{A}$ we have $f(x) ≈_δ g(x) ∈ A$, so $x ∈ f\preim{N_δ(A)} ⊆ F_A$.
			Hence, $g\preim{A} ⊆ F_A$ and similarly $g\preim{B} ⊆ F_B$.
		For every $x ∈ ∂F_A$ we have $f(x) ∈ N_ε(\clo{N_δ(B)}) ⊆ N_{ε + δ}(B)$, and so $g(x) ∈ N_{ε + 2δ}(B)$.
			Hence $g\im{∂F_A} ⊆ N_{ε + 2δ}(B)$ and similarly $g\im{∂F_B} ⊆ N_{ε + 2δ}(A)$.
		Finally, $g\im{F_A ∩ F_B} ⊆ N_{ε + 2δ}(A) ∩ N_{ε + 2δ}(B)$ since $f\im{F_A ∩ F_B} ⊆ N_{ε + δ}(A) ∩ N_{ε + δ}(B)$.
	\end{proof}
\end{proposition}

Next, let us observe that the definition of $ε$-crookedness can be simplified for compact graphs or more generally Peano continua.

By a \emph{compact graph}, we mean a topological realization of a finite ${≤}1$-dimensional simplicial complex, or equivalently a finite union of arcs that intersect at most at their end-points.
Recall that a \emph{Peano continuum} is a (necessarily metrizable) continuum that is a continuous image of the unit interval $\II$.
Equivalently, it is a non-empty metrizable continuum that is locally connected, see \cite[VIII]{Nadler92}.

\begin{notation}
	By an \emph{ordered arc} $[x, y]$ in a space $X$ for $x ≠ y$ we mean the image of an embedding $f\maps [0, 1] \to X$ such that $f(0) = x$ and $f(1) = y$, together with the linear order induced from $[0, 1]$.
	Also $[x, x]$ denotes the degenerate ordered arc $\set{x}$.
\end{notation}

\begin{lemma} \label{thm:Peano_crookedness}
	Let $X$ be a Peano continuum and let $\tuple{A, B, U, V}$ be a quadruple in $X$ such that $A, B$ are closed and $U, V$ are their respective closed neighborhoods.
	If for every ordered arc $[x, y]$ in $X$ such that $x ∈ A$ and $y ∈ B$ there are points $x ≤ y' ≤ x' ≤ y$ such that $x' ∈ U$ and $y' ∈ V$, then $X$ is crooked at $\tuple{A, B, U, V}$.
	
	\begin{proof}
		First, we suppose that $A ∩ V = ∅ = B ∩ U$.
		Let \[\textstyle
			G_A := ⋃\set{[x, y]\text{ arc in }X: x ∈ A\text{ and }[x, y] ∩ V = ∅}
		\]
		and $F_A := \clo{G_A}$.
		Since $A ∩ V = ∅$, we have $A ⊆ G_A$.
		We also have $∂F_A ⊆ ∂G_A$.
		Let us show $∂G_A ⊆ \clo{V} = V$.
			Let $x ∈ ∂G_A$ and let $N$ be a basic arcwise connected neighborhood of $x$.
			There is an arc $[a, b] ⊆ N$ such that $a ∈ G_A$ and $b ∉ G_A$.
			But from the definition of $G_A$ there is an arc $[a_0 , a]$ disjoint with $V$ such that $a_0 ∈ A$.
			The union $[a_0, a] ∪ [a, b]$ may not be an arc, but if we let $b_0 := \max\set{t ∈ [a, b]: t ∈ [a_0, a]}$, then $[a_0, b_0] ∪ [b_0, b]$ is an arc.
			Since $b ∉ G_A$, this arc is not in the collection defining $G_A$, and so $N ∩ V ⊇ [a, b] ∩ V ⊇ [b_0, b] ∩ V ≠ ∅$.
		Analogously we define $G_B$ and $F_B := \clo{G_B}$, so we have $B ⊆ F_B$ and $∂F_B ⊆ U$.
		
		By the assumption we have $G_A ∩ G_B = ∅$.
			Otherwise we would have arcs $[x, z] ⊆ G_A$ with $x ∈ A$ and $[z, y] ⊆ G_B$ with $y ∈ B$ that can be concatenated into a single arc $[x, y]$.
			The equalities $[x, z] ∩ V = ∅$ and $[z, y] ∩ U = ∅$ contradict the existence of points $x ≤ y' ≤ x' ≤ y$ with $y' ∈ V$ and $x' ∈ U$.
		Hence, we have $F_A ∩ F_B ⊆ ∂G_A ∩ ∂G_B ⊆ U ∩ V$.
		
		In the general case, we put $A' := A \setminus V$ and $B' := B \setminus V$.
		Our assumption holds for $\tuple{A', B', U, V}$, and so, by the special case, $X$ is crooked at $\tuple{A', B', U, V}$.
		Hence, $X$ is crooked at $\tuple{A, B, U, V}$ by Lemma~\ref{thm:disjoint_enough}.
	\end{proof}
\end{lemma}

\begin{theorem} \label{thm:graph_crookedness}
	Let $f\maps X \to Y$ be a continuous map from a compact graph to a metric compactum, and let $ε > 0$. The following conditions are equivalent.
	\begin{enumerate}
		\item $f$ is $ε$-crooked.
		\item $f ∘ e$ is $ε$-crooked for every continuous map $e\maps \II \to X$.
		\item For every arc $[x, y] ⊆ X$ there are $x ≤ y' ≤ x' ≤ y$ such that $f(x) ≈_ε f(x')$ and $f(y) ≈_ε f(y')$.
	\end{enumerate}
	
	\begin{proof}
		(i)$\implies$(ii) follows from Proposition~\ref{thm:crooked_calculus}.
		
		(ii)$\implies$(iii).
		For degenerate arcs (iii) is trivially true.
		A non-degenerate arc $[x, y] ⊆ X$ is homeomorphic to $\II$, so by (ii), $f\restr{[x, y]}$ is $ε$-crooked.
		Let $\tuple{F_x, F_y}$ be a disjoint pair witnessing that $f\restr{[x, y]}$ is $ε$-crooked at $\tuple{\set{f(x)}, \set{f(y)}}$, which exists by Lemma~\ref{thm:disjoint_witness}.
		Let $y' := \max\set{t ∈ [x, y]: [x, t] ⊆ F_x}$ and $x' := \min\set{t ∈ [x, y]: [t, y] ⊆ F_y}$.
		Then $x ≤ y' ≤ x' ≤ y$ since $F_x ∩ F_y = ∅$, and $f(y') ∈ f\im{∂F_x} ⊆ N_ε(\set{f(y)})$ and $f(x') ∈ f\im{∂F_y} ⊆ N_ε(\set{f(x)})$.
		
		(iii)$\implies$(i).
		Let $A, B ⊆ Y$ be closed subsets.
		We would like to use Lemma~\ref{thm:Peano_crookedness} for $\tuple{A, B, N_ε(A), N_ε(B)}\restr{f}$, however the neighborhoods are not closed.
		Instead we take the sets $\set{x ∈ X: d(f(x), A) ≤ ε'}$ and $\set{x ∈ X: d(f(x), B) ≤ ε'}$ for suitable $ε' < ε$.
		This is possible since if for an arc $[x, y] ⊆ X$ there are $x ≤ y' ≤ x' ≤ y$ such that $f(x) ≈_ε f(x')$ and $f(y) ≈_ε f(y')$, the same is true also for some $ε' < ε$, and by the assumption $X$ is a compact graph, so there are only finitely many arcs from $x$ to $y$.
	\end{proof}
\end{theorem}

\begin{remark}
	Originally, Bing~\cite{Bing51_crooked} defined an $ε$-crooked arc by (iii) applied to $f$ an inclusion of an arc.
	Later, Brown~\cite{Brown60_crooked} defined an $ε$-crooked map by (ii) where $ε$-crookedness of the compositions $f ∘ e\maps \II \to Y$ is defined according to (iii).
	In general, this is a weaker notion of $ε$-crookedness.
	The theorem above shows that it is equivalent for compact graphs.
	Moreover, the notions are essentially equivalent for Peano continua (which Brown considered): if a continuous map $f\maps X \to Y$ from a Peano continuum is $ε$-crooked in the sense of Brown, then it is $ε'$-crooked for every $ε' > ε$.
	
	Condition~(iii) also gives a convenient characterization when $X = \II$, which we shall use later.
\end{remark}

\subsection{Crookedness of inverse limits}

We observe how $ε$-crookedness behaves with respect to inverse limits, namely that by building a sequence of more and more crooked maps we obtain a crooked space as the limit, as stated in Theorem~\ref{thm:crooked_limit}.
The proof is rather straightforward, but builds upon several somewhat technical propositions.

\begin{observation}
	A continuous map $f\maps X \to Y$ between metric compacta is crooked if and only if it is $ε$-crooked for every $ε > 0$.
	This is because for every admissible quadruple $\tuple{A, B, U_A, U_B}$ there is $ε > 0$ such that $N_ε(A) ⊆ U_A$ and $N_ε(B) ⊆ U_B$ since $A$ and $B$ are compact, so $ε$-crookedness at $\tuple{A, B}$ implies crookedness at $\tuple{A, B, U_A, U_B}$.
\end{observation}

Throughout this subsection we fix a sequence $\tuple{X_*, f_*}$ of Hausdorff compacta with limit $\tuple{X_∞, f_{*, ∞}}$.

The following proposition and lemma (proof of which is left to the reader) use the standard fact that if two closed subsets $F, H$ of the limit are disjoint, then so are their images $f_{m, ∞}\im{F}, f_{m, ∞}\im{H}$ for sufficiently large $m$.

\begin{proposition} \label{thm:crooked_limit_maps}
	For every admissible quadruple $\tuple{A, B, U, V}$ in $X_∞$ there is $n ∈ ω$ and an admissible quadruple $\tuple{A', B', U', V'}$ in $X_n$ such that $\tuple{A', B', U', V'}\restr{f_{n, ∞}} ≤ \tuple{A, B, U, V}$.
	Hence, $X_∞$ is crooked if and only if every $f_{n, ∞}$ is crooked.
	
	\begin{proof}
		We take $n ∈ ω$ such that $f_{n, ∞}\im{A} ∩ f_{n, ∞}\im{X_∞ \setminus U} = ∅$ and $f_{n, ∞}\im{B} ∩ f_{n, ∞}\im{X_∞ \setminus V} = ∅$ and put 
		\[
			\tuple{A', B', U', V'} := \tuple{f_{n, ∞}\im{A},\ f_{n, ∞}\im{B},\ X_n \setminus f_{n, ∞}\im{X_∞ \setminus U},\ X_n \setminus f_{n, ∞}\im{X_∞ \setminus V}}.
		\]
		Clearly, $f_{n, ∞}\preim{A'} ⊇ A$.
		We also have $X_∞ \setminus f_{n, ∞}\preim{U'} = f_{n, ∞}\preim{f_{n, ∞}\im{X_∞ \setminus U}} ⊇ X_∞ \setminus U$, so $f_{n, ∞}\preim{U'} ⊆ U$.
		Similarly for $B$ and $V$.
		It follows that if $f_{n, ∞}$ is crooked at $\tuple{A', B', U', V'}$, then $X_∞$ is crooked at $\tuple{A, B, U, V}$.
		Also clearly, if $X_∞$ is crooked, then every $f_{n, ∞}$ is crooked.
	\end{proof}
\end{proposition}

\begin{lemma} \label{thm:disjoint_before_limit} \hfill
	\begin{enumerate}
		\item If for some closed sets $F, H ⊆ X_∞$ and an open set $U ⊆ X_n$ we have $F ∩ H ⊆ f_{n, ∞}\preim{U}$, then $f_{m, ∞}\im{F} ∩ f_{m, ∞}\im{H} ⊆ f_{n, m}\preim{U}$ for every sufficiently large $m ≥ n$.
		\item If for some open set $U ⊆ X_n$ we have $f_{n, ∞}\im{X_∞} ⊆ U$, then $f_{n, m}\im{X_m} ⊆ U$ for every sufficiently large $m ≥ n$.
		\qed
	\end{enumerate}
	
\end{lemma}

\begin{lemma} \label{thm:normal_refinement}
	Let $X$ be a normal space, let $F_i ⊆ X$, $i < 3$, be closed sets, and let $U_{i, j} ⊆ X$, $i < j < 3$, be open sets such that $F_i ∩ F_j ⊆ U_{i, j}$ for every $i < j < 3$.
	Then there are closed sets $F'_i ⊆ X$, $i < 3$, such that $F_i ⊆ \int(F'_i)$ for every $i < 3$, and $F'_i ∩ F'_j ⊆ U_{i, j}$ for every $i < j < 3$.
	
	\begin{proof}
		There is an open set $W$ such that $F_0 ∩ F_1 ⊆ W ⊆ \clo{W} ⊆ U_{0, 1}$.
		The disjoint closed sets $F_0 \setminus W$ and $F_1 \setminus W$ can be separated by neighborhoods $W_0, W_1$ with disjoint closures.
		By putting $V_{0, 1} := W_0 ∪ W$ and $V_{1, 0} := W_1 ∪ W$ we obtain neighborhoods $F_0 ⊆ V_{0, 1}$ and $F_1 ⊆ V_{1, 0}$ such that $\clo{V_{0, 1}} ∩ \clo{V_{1, 0}} ⊆ U_{0, 1}$.
		Similarly, we obtain sets $V_{i, j}, V_{j, i}$ for every $i < j < 3$.
		Then it is enough to put $F'_i := \clo{⋂_{j < 3, j ≠ i} V_{i, j}}$.
	\end{proof}
\end{lemma}

\begin{proposition} \label{thm:crooked_before_limit}
	If $f_{n, ∞}$ is crooked at an admissible quadruple $\tuple{A, B, U, V}$, then there is some $m ≥ n$ such that already $f_{n, m}$ is crooked at $\tuple{A, B, U, V}$.
	
	\begin{proof}
		For every $m ≤ ∞$ such that $m ≥ n$ let $A_m := f_{n, m}\preim{A} ⊆ X_m$.
		Analogously we define $B_m, U_m, V_m$.
		We have that $X_∞$ is crooked at $\tuple{A_∞, B_∞, U_∞, V_∞}$, so let $X_∞ = F_{A, ∞} ∪ H_∞ ∪ F_{B, ∞}$ be a witnessing triple.
		For every $m < ∞$ such that $m ≥ n$ we put 
		\[
			F_{A, m} := f_{m, ∞}\im{F_{A, ∞}} ∪ A_m, \qquad F_{B, m} := f_{m, ∞}\im{F_{B, ∞}} ∪ B_m, \quad\text{and}\quad H_m := f_{m, ∞}\im{H_∞}.
		\]
		Note that $A_m ⊆ F_{A, m}$ and $F_{A, m} ∩ H_m = f_{m, ∞}\im{F_{A, ∞}} ∩ f_{m, ∞}\im{H_∞}$ since $F_{A, m} ∩ f_{m, ∞}\im{X_∞} = f_{m, ∞}\im{A_∞}$.
		Similarly for $B$ instead of $A$, and also $F_{A, m} ∩ F_{B, m} = (f_{m, ∞}\im{A_∞} ∩ f_{m, ∞}\im{B_∞}) ∪ (A_m ∩ B_m)$.
		Hence, by Lemma~\ref{thm:disjoint_before_limit} there is $n_0 ≥ n$ such that for every $m ≥ n_0$ we have
		\[
			F_{A, m} ∩ H_m ⊆ V_m, \qquad F_{B, m} ∩ H_m ⊆ U_m, \quad\text{and}\quad F_{A, m} ∩ F_{B, m} ⊆ U_m ∩ V_m.
		\]
		But unless $f_{m, ∞}$ is surjective, we are missing the property $F_{A, m} ∪ H_m ∪ F_{B, m} = X_m$.
		By Lemma~\ref{thm:normal_refinement} there are closed sets $F'_A, H', F'_B ⊆ X_{n_0}$ such that $F_{A, n_0} ⊆ \int(F'_A)$, $H_{n_0} ⊆ \int(H')$, and $F_{B, n_0} ⊆ \int(F'_B)$ that satisfy the same neighborhood conditions, i.e. $F'_A ∩ H' ⊆ V_{n_0}$, $F'_B ∩ H' ⊆ U_{n_0}$, and $F'_A ∩ F'_B ⊆ U_{n_0} ∩ V_{n_0}$.
		We put
		\[
			W := \int(F'_A ∪ H' ∪ F'_B) ⊇ F_{A, n_0} ∪ H_{n_0} ∪ F_{B, n_0} = f_{n_0, ∞}\im{X_∞}.
		\]
		By Lemma~\ref{thm:disjoint_before_limit} there is $m ≥ n_0$ such that $f_{n_0, m}\im{X_m} ⊆ W$, and so $f_{n_0, m}\preim{F'_A} ∪ f_{n_0, m}\preim{H'} ∪ f_{n_0, m}\preim{F'_B} = X_m$.
		Now it is easy to see that this triple witnesses that $f_{n, m}$ is crooked at $\tuple{A, B, U, V}$.
	\end{proof}
\end{proposition}

Recall that for every metric compactum $X$, the collection $\K(X)$ of all non-empty compact subsets of $X$ endowed with the Hausdorff distance $d_H(A, B) := \inf\set{ε > 0: A ⊆ N_ε(B), B ⊆ N_ε(A)}$ forms a metric compactum, see e.g. Nadler~\cite[IV]{Nadler92}.

\begin{lemma} \label{thm:crookedness_continuity}
	Let $f\maps X \to Y$ be a continuous map between metric compacta.
	If $f$ is $ε$-crooked at $\tuple{A, B}$, there is some $δ > 0$ such that $f$ is $ε$-crooked at every $\tuple{A', B'}$ such that $d_H(A, A') < δ$ and $d_H(B, B') < δ$.
	
	\begin{proof}
		Let $\tuple{F_A, F_B}$ be a witness.
		First observe that $\tuple{F_A, F_B}$ also witnesses that $f$ is $ε'$-crooked at $\tuple{A, B}$ for some $ε' < ε$.
		We have $f\im{∂F_A} ⊆ N_ε(B)$.
		Since the left-hand-side set is compact, we have $f\im{∂F_A} ⊆ N_{ε'}(B)$ for some $ε' < ε$.
		Similarly for the other conditions.
		
		Next, by Lemma~\ref{thm:normal_refinement} there a is pair $\tuple{H_A, H_B}$ witnessing that $f$ is $ε'$-crooked at $\tuple{A, B}$ such that $F_A ⊆ \int(H_A)$ and $F_B ⊆ \int(H_B)$.
		We have $f\preim{A} ⊆ \int(H_A)$.
		Since $f$ is closed and $A$ is compact, there is $δ > 0$ such that $f\preim{N_δ(A)} ⊆ \int(H_A)$ and similarly for $B$.
		Hence, $\tuple{H_A, H_B}$ witnesses that $f$ is crooked at $\tuple{N_δ(A), N_δ(B), N_{ε'}(A), N_{ε'}(B)}$.
		
		Finally, suppose that $δ$ is small enough so that $ε' + δ ≤ ε$.
		Then for every $d_H(A', A) < δ$ and $d_H(B', B) < δ$, we have $A' ⊆ N_δ(A)$ and $N_{ε'}(A) ⊆ N_{ε'}(N_δ(A')) ⊆ N_ε(A')$, and similarly for $B$.
		Hence, $\tuple{N_δ(A), N_δ(B), N_{ε'}(A), N_{ε'}(B)}$ is a stronger condition than $\tuple{A', B',\allowbreak N_ε(A'), N_ε(B')}$, and therefore $\tuple{H_A, H_B}$ also witnesses that $f$ is $ε$-crooked at $\tuple{A', B'}$.
	\end{proof}
\end{lemma}

Together, we obtain the following.

\begin{theorem} \label{thm:crooked_limit}
	Let $\tuple{X_*, f_*}$ be a sequence of metric compacta with limit $\tuple{X_∞, f_{*, ∞}}$.
	The following conditions are equivalent:
	\begin{enumerate}
		\item $X_∞$ is hereditarily indecomposable.
		\item $X_∞$ is crooked.
		\item Every map $f_{n, ∞}$, $n ∈ ω$, is crooked.
		\item For every $n ∈ ω$, $ε > 0$, and a pair of closed sets $A, B ⊆ X_n$ there is $m ≥ n$ such that $f_{n, m}$ is $ε$-crooked at $\tuple{A, B}$.
		\item $f_*$ is a \emph{crooked sequence}, i.e. for every $n ∈ ω$ and $ε > 0$ there is $m ≥ n$ such that $f_{n, m}$ is $ε$-crooked.
	\end{enumerate}
	
	\begin{proof}
		The equivalence (i)$\iff$(ii) follows from the result by Krasinkiewicz and Minc~\cite[Theorem~3.4]{KM77}, see also Remark~\ref{rm:hereditarily_indecomposable}.
		The equivalence (ii)$\iff$(ii) holds by Proposition~\ref{thm:crooked_limit_maps}.
		The implications (v)$\implies$(iv)$\implies$(iii) are trivial.
		The implication (iii)$\implies$(iv) follows from Proposition~\ref{thm:crooked_before_limit}.
		We prove (iv)$\implies$(v).
		
		Let us fix $n ∈ ω$ and $ε > 0$.
		For a pair $\tuple{A, B}$ of closed subsets of $X_n$ there is $m_{A, B} ≥ n$ such that $f_{n, m_{A, B}}$ is $ε$-crooked at $\tuple{A, B}$.
		By Lemma~\ref{thm:crookedness_continuity} there is an open neighborhood $\U_{A, B}$ of $\tuple{A, B}$ in the hyperspace $\K(X_n) × \K(X_n)$ such that $f_{n, m_{A, B}}$ is $ε$-crooked at every $\tuple{A', B'} ∈ \U_{A, B}$.
		Since the hyperspace $\K(X_n) × \K(X_n)$ is compact, it can be covered by finitely many of these open sets $\U_{A, B}$, and so there are numbers $m_i ≥ n$ for $i ∈ I$ finite such that for every closed pair $\tuple{A', B'}$ the map $f_{n, m_i}$ is $ε$-crooked at $\tuple{A', B'}$ for some $i ∈ I$.
		Hence, it is enough to consider $m = \max\set{m_i: i ∈ I}$.
	\end{proof}
\end{theorem}

\subsection{Construction of the pseudo-arc} \label{sec:pseudo}

Let us recall a canonical construction of the pseudo-arc while (re)introducing several useful notions.
By Bing's theorem, the pseudo-arc is the unique hereditarily indecomposable arc-like continuum, i.e. the unique hereditarily indecomposable limit of a sequence in $\I$.
By Theorem~\ref{thm:crooked_limit}, the limit is hereditarily indecomposable if and only if the $\I$-sequence is crooked.
Also, by Theorem~\ref{thm:graph_crookedness}, an $\I$-map $f\maps \II \to \II$ is $ε$-crooked if for every $x ≤ y ∈ \II$ there are $x ≤ y' ≤ x' ≤ y$ such that $f(x) ≈_ε f(x')$ and $f(y) ≈_ε f(y')$.

\begin{construction} \label{thm:crooked_sequence_exists}
	We obtain a crooked sequence in $\I$ under the assumption that for every $ε > 0$ there is an $ε$-crooked $\I$-map by an inductive construction.
	
	Let $ε_0 := 1$.
	Given $ε_n$ we let $f_n$ be an $ε_n$-crooked $\I$-map, and we pick $ε_{n + 1} > 0$ such that for every $k < n + 1$ the map $f_{k, n + 1}$ is $\tuple{ε_k / 2^{n + 1 - k}, ε_n}$-continuous.
	Then for every $n ∈ ω$ and $ε > 0$ there is $n' ≥ n$ such that $ε_n / 2^{n' - n} ≤ ε$.
	Since the map $f_{n'}$ is $ε_{n'}$-crooked and the map $f_{n, n'}$ is $\tuple{ε_n / 2^{n' - n}, ε_{n'}}$-continuous, and so $\tuple{ε, ε_{n'}}$-continuous,
	by Proposition~\ref{thm:crooked_calculus} we have that $f_{n, n' + 1} = f_{n, n'} ∘ f_{n'}$ is $ε$-crooked. 
\end{construction}

To observe that there are arbitrarily crooked $\I$-maps we consider a discrete variant of crookedness.
By a \emph{graph} we mean a set $G$ of vertices endowed with a symmetric reflexive edge relation $≈$ (to be thought of as a discrete nearness relation).
We use reflexive instead of antireflexive edge relation so that homomorphisms may contract edges to points. 
Note that a graph is essentially the same thing as a ${≤}1$-dimensional abstract simplicial complex.
A \emph{graph homomorphism} or a \emph{simplicial map} is a map $s\maps G \to H$ that preserves the edge relation, i.e. if $x ≈ y$, then $s(x) ≈ s(y)$, for every $x, y ∈ G$.
For every $n ∈ ω$ let $\II_n$ denote the finite linear graph with vertices $\set{0, …, n}$ and the edge relation $i ≈ j \iff \abs{i - j} ≤ 1$.
The \emph{geometric realization} $\geom{\II_n}$ of the graph $\II_n$ is the unit interval $\II$ if $n > 0$, and the singleton $\set{0}$ otherwise.
The \emph{geometric realization} of a simplicial map $s\maps \II_m \to \II_n$ is the piecewise linear map $\geom{s}\maps \geom{\II_m} \to \geom{\II_n}$ determined by the points $i/m \mapsto s(i)/n$ for $i ∈ \II_m$.

\begin{definition}
	We call a simplicial map $s\maps \II_m \to \II_n$ \emph{crooked} if for every $i ≤ j ∈ \II_m$ there are integers $i ≤ j' ≤ i' ≤ j$ such that $s(i) ≈ s(i')$ and $s(j) ≈ s(j')$.
\end{definition}

\begin{remark}
	Lewis and Minc~\cite{LM10} considered a closely related notion of an \emph{$n$-crooked} continuous surjection $\II \to \II$.
	It is easy to see that a simplicial map $s\maps \II_m \to \II_n$ is crooked if and only if $\geom{s}$ is $n$-crooked.
	But in general, $n$-crooked maps do not have to be geometric realizations of simplicial maps.
\end{remark}

The following proposition can be viewed as a small improvement of \cite[Proposition~3.4 and 3.5]{LM10} in the special case of simplicial maps.
Note that our bounds $1/n < ε ≤ 3/2n$ form overlapping intervals, as opposed to the original bounds $2/n \nleq 1/n$.
We will later use the better bound $3/2n$ in Lemma~\ref{thm:simplicial_approximation}.

\begin{proposition} \label{thm:n_crooked}
	The geometric realization of a crooked simplicial map $s\maps \II_m \to \II_n$ is $ε$-crooked for every $ε > 1/n$.
	On the other hand, if $\geom{s}$ is $ε$-crooked for some $ε ≤ 3/2n$, then $s$ is crooked.
	
	\begin{proof}
		Let $f$ be the geometric realization of $s$.
		Suppose that $s$ is crooked, and let $x ≤ y ∈ \II$.
		If $\abs{f(x) - f(y)} ≤ 2/n < 2ε$, then for a point $z ∈ [x, y]$ such that $f(z) = (f(x) + f(y))/2$ we have $f(x) ≈_ε f(z) ≈_ε f(y)$, so it is enough to put $x' := y' := z$.
		Otherwise let $i, i' ∈ \II_m$ with $\abs{i - i'} ≤ 1$ be such that $x ∈ [i/m, i'/m]'$ and so $f(x) ∈ [s(i), s(i')]'$.
		Here, $[a, b]'$ is a shortcut for $[a, b] ∪ [b, a]$.
		There are analogous integers $j, j' ∈ \II_m$ for $y$.
		Note that since $\abs{f(x) - f(y)} > 2/n$ and $\abs{f(i/m) - f(i'/m)}, \abs{f(j/m) - f(j'/m)} ≤ 1/n$, we may suppose, maybe after swapping $i, i'$ and/or $j, j'$ that $s(i) ≤ s(i') < s(j') ≤ s(j)$ or $s(i) ≥ s(i') > s(j') ≥ s(j)$, and in both cases we have $i, i' < j, j'$.
		Suppose the former.
		The latter case is analogous.
		Let $i ≤ i'' ≤ j'' ≤ j$ be such that $s(i) = s(i'')$ and $s(j) = s(j'')$ and $\abs{i'' - j''}$ is minimal possible.
		It follows that $s\restr{[i'', j'']}$ has a unique minimum at $i''$ and a unique maximum at $j''$.
		Since the map $s$ is crooked, there are $i'' ≤ j''' ≤ i''' ≤ j''$ such that $s(i'') ≈ s(i''')$ and $s(j'') ≈ s(j''')$.
		Let $x' := i'''/m$ and $y' := j'''/m$.
		Since $\abs{f(x) - f(y)} > 2/n$, we have $s(j') - s(i) ≥ 2$, and hence $s(i) = s(i'') < s(j''')$.
		Then, since $\abs{i - i'} ≤ 1$, we have $i, i' ≤ j'''$, and so $x ≤ y'$.
		Similarly, we have $x' ≤ y$.
		Also, we have $s(i) = s(i'') ≤ s(i''') ≈ s(i)$ and $s(i) ≤ s(i') ≈ s(i)$, hence also $s(i') ≈ s(i''')$, and so $f(x) ≈_ε f(x')$.
		Similarly $f(y) ≈_ε f(y')$.
		
		For the other implication, suppose that $f$ is $ε$-crooked, so for every $i ≤ j ∈ \II_m$ there are $i/m ≤ y ≤ x ≤ j/m$ such that $f(x) ≈_ε s(i)/n$ and $f(y) ≈_ε s(j)/n$.
		Let $i' ≤ i'' ∈ \II_m$ be such that $x ∈ [i'/m, i''/m]$ and $\abs{i' - i''}$ is minimal possible.
		Similarly, let $j' ≤ j''$ be the corresponding pair for $y$.
		There are two cases: either $i ≤ j' ≤ j'' ≤ i' ≤ i'' ≤ j$, or $i ≤ j' = i' < j'' = i'' ≤ j$.
		In the first case, we can use the following argument since both $j', j''$ are below $i', i''$.
		Because of the piecewise linear definition of $f$, we have that $s(i')/n ≈_ε s(i)/n$ or $s(i'')/n ≈_ε s(i)/n$.
		Since $ε ≤ 2/n$, we have $s(i') ≈ s(i)$ or $s(i'') ≈ s(i)$.
		The same argument works for $j, j', j''$.
		
		In the second case, we have $y ≤ x ∈ (i'/m, i''/m) = (j'/m, j''/m)$.
		We may suppose that $\abs{s(i) - s(j)} > 2$, and so $≥ 3$, since otherwise we can pick $i ≤ k ≤ j$ such that $s(i) ≈ s(k) ≈ s(j)$.
		We also suppose that $s(i) < s(j)$ since the other case is symmetric.
		Now, if $s(i') ≥ s(i'')$, we are done since then $s(i)/n ≤ s(i'')/n ≤ f(x) ≈_ε s(i)/n$, and so $s(i) ≈ s(i'')$ since $ε ≤ 2/n$, and similarly $s(j) ≈ s(i')$.
		Otherwise, $s(i') ≤ s(i'')$ and the interval $[s(i)/n, s(j)/n]$ is covered by $[s(i)/n, f(x)] ∪ [f(y), s(j)/n]$.
		Hence, we have $3/n ≤ (f(x) - s(i)/n) + (s(j)/n - f(y)) < 2ε$, and so $ε > 3/2n$, which is a contradiction with the hypothesis.
	\end{proof}
\end{proposition}

\begin{construction} \label{con:crooked}
	For every $n ∈ ω$ we construct a canonical crooked simplicial surjection $c_n\maps \II_{\crn(n)} \to \II_n$ (see Figure~\ref{fig:crooked}).
	These correspond to the $\I$-maps originally considered by Lewis and Minc~\cite{LM10}.
	It follows via the geometric realization that for every $ε > 0$ there is an $ε$-crooked $\I$-map.
	We also define a reversed version of $c_n$ by $c'_n(i) := c_n(\crn(n) - i)$, $i ∈ \II_{\crn(n)}$.
	We put $\crn(0) := 0$, $\crn(1) := 1$, and $\crn(n) := 2 \crn(n - 1) + \crn(n - 2)$ for $n ≥ 2$.
	For the maps we put $c_0 := \id_{\II_0}$, $c_1 := \id_{\II_1}$, and $c_n$, $n ≥ 2$, is defined as the “concatenation” of $c_{n - 1}$, $c'_{n - 2}$, and another copy of $c_{n - 1}$, i.e. 
	\begin{alignat*}{4}
		c_n(i) &:= c_{n - 1}(i), &&\quad i ≤ \crn(n - 1), \\
		c_n(b_1 + i) &:= c'_{n - 2}(i) + 1, &&\quad i ≤ \crn(n - 2), &&\quad b_1 := \crn(n - 1), \\
		c_n(b_2 + i) &:= c_{n - 1}(i) + 1, &&\quad i ≤ \crn(n - 1), &&\quad b_2 := b_1 + \crn(n - 2).
	\end{alignat*}

\begin{figure}[ht]
	\centering
	\begin{tikzpicture}[
		x = {(0.6em, 0)},
		y = {(0, 1.5em)},
	]
		\begin{scope}[gray]
			\foreach \x in {0, 29}
				\draw (\x, 0) -- (\x, 5);
			\foreach \y in {0, ..., 5}
				\draw (0, \y) -- (29, \y);
		\end{scope}
		\draw[thick] (0, 0)
			-- (1, 1) -- (2, 2) -- (3, 1) -- (4, 2) -- (5, 3)
			-- (6, 2) -- (7, 1)
			-- (8, 2) -- (9, 3) -- (10, 2) -- (11, 3) -- (12, 4)
			-- (13, 3) -- (14, 2) -- (15, 3) -- (16, 2) -- (17, 1)
			-- (18, 2) -- (19, 3) -- (20, 2) -- (21, 3) -- (22, 4)
			-- (23, 3) -- (24, 2)
			-- (25, 3) -- (26, 4) -- (27, 3) -- (28, 4) -- (29, 5);
	\end{tikzpicture}
	\caption{The canonical crooked simplicial surjection $c_5$.}
	\label{fig:crooked}
\end{figure}

	\begin{proof}
		Let $i ≤ j ∈ \II_{\crn(n)}$.
		We show that there are $i ≤ j' ≤ i' ≤ j ∈ \II_{\crn(n)}$ such that $c_n(i) ≈ c_n(i')$ and $c_n(j) ≈ c_n(j')$.
		If $n ≤ 1$, this is easy.
		If $n > 1$, $i = 0$ and $j = \crn(n)$, it is enough to put $j' := b_1$ and $i' := b_2$.
		Otherwise, let $i ≤ i'' ≤ j'' ≤ j$ be such that $c_n(i'') = c_n(i)$, $c_n(j'') = c_n(j)$, and $\abs{i'' - j''}$ is minimal possible.
		It follows that $c_n\restr{[i'', j'']}$ has a unique minimum and maximum at the end-points.
		Since $i > 0$ or $j < \crn(n)$, we have $b_1, b_2 ∉ (i'', j'')$, and it is enough to solve the problem in one of the restrictions $c_n\maps [0, b_1] \to [0, n - 1]$, $c_n\maps [b_1, b_2] \to [1, n - 1]$, and $c_n\maps [b_2, \crn(n)] \to [1, n]$, which are copies of $c_{n - 1}$, $c'_{n - 2}$, and $c_{n - 1}$, respectively.
		It follows by induction that there are $i'' ≤ j' ≤ i' ≤ j''$ such that $c_n(i') ≈ c_n(i'') = c_n(i)$ and $c_n(j') ≈ c_n(j'') = c_n(j)$.
	\end{proof}
\end{construction}

\begin{construction}
	Lewis and Minc~\cite[Theorem~3.24]{LM10} gave an algorithmic construction of a crooked $\I$-sequence and so of the pseudo-arc.
	Let $m_0 := 1$, and $m_{n + 1} := \crn(2 m_n)$ for $n ∈ ω$.
	For every $n$ we put $ε_n := 1 / m_n$ and $f_n := \geom{c_{2m_n}}$.
	Every map $f_n$ is $ε$-crooked for every $ε > 1 / (2m_n)$, and so is $ε_n$-crooked.
	Also every $c_n$ is $\crn(n)/n$-Lipschitz, and so $f_n$ is $L$-Lipschitz for $L = \crn(2m_n) / (2m_n) = m_{n + 1} / (2m_n) = ε_n / (2ε_{n + 1})$.
	Hence, every composition $f_{n, n'}$, $n ≤ n'$, is $\tuple{ε_n / 2^{n' - n}, ε_{n'}}$-continuous, and the sequence $f_*$ is crooked as in Construction~\ref{thm:crooked_sequence_exists}.
\end{construction}

\subsection{Crookedness factorization theorem}

The following theorem seems to be known (see \cite[p.~179]{HO16}), however we could not find an explicit proof, and so we prove it here.

\begin{theorem}[Crookedness factorization] \label{thm:crooked_factorization}
	For every $\I$-map $g$ and every $ε > 0$ there is $δ > 0$ such that for every $δ$-crooked $\I$-map $f$ there is an $\I$-map $h$ such that $f ≈_ε g ∘ h$.
\end{theorem}

We will use the theorem later to observe that crookedness of an $\I$-sequence is enough for it to be a Fraïssé sequence.
This yields an alternative proof of Bing's Theorem on uniqueness of the pseudo-arc via uniqueness of a generic object, see Remark~\ref{thm:Bing_reproved}.

\medskip

The proof will follow after several preparatory lemmas.

\begin{notation}
	For two simplicial maps $s\maps \II_m \to \II_n$ and $s'\maps \II_{m'} \to \II_n$ such that $s(m) = s'(0)$ let $s \concat s'$ denote the \emph{concatenation} $\II_{m + m'} \to \II_n$ defined by $(s \concat s')(i) := s(i)$ for $i ≤ m$ and $s'(i - m)$ for $i ≥ m$.
	Note that $t ∘ (s \concat s') = (t ∘ s) \concat (t ∘ s')$ for a simplicial map $t\maps \II_n \to \II_{n'}$.
	For every $m + k ≤ n$ we denote the simplicial inclusion $i \mapsto i + k\maps \II_m \to \II_n$ by $e^n_{m, k}$.
	We have $(s \concat s') ∘ e^{m + m'}_{m, 0} = s$ and $(s \concat s') ∘ e^{m + m'}_{m', m} = s'$.
	For every $m ∈ ω$ let $r_m\maps \II_m \to \II_m$ denote the “inversion” $i \mapsto m - i$.
	So for every simplicial map $s\maps \II_m \to \II_n$, its “reversed version” is $s ∘ r_m$, e.g. for the canonical crooked maps we have $c'_n = c_n ∘ r_{\crn(n)}$.
	Since the lower index is often clear, we write just $r$ instead of $r_m$.
	
	Using this notation, we may write the recursive definition of the canonical crooked map (Construction~\ref{con:crooked}) as $c_n := (e^n_{n - 1, 0} ∘ c_{n - 1}) \concat (e^n_{n - 2, 1} ∘ c_{n - 2} ∘ r) \concat (e^n_{n - 1, 1} ∘ c_{n - 1})$ for $n ≥ 2$.
	
	Additionally, let $B_s$ denote the set of \emph{break-points} of $s$, i.e. points $i ∈ \II_m$ such that $s(i + 1) - s(i) ≠ s(i) - s(i - 1)$, which includes the end-points $0, m$.
\end{notation}

\begin{lemma} \label{thm:simplicial_crooked_factorization}
	Let $s\maps \II_m \to \II_n$ be a simplicial surjection.
	If $s(0) = 0$, $s(m) = n$, and $s\restr{B_s}$ is an injection, then there is a simplicial surjection $s'\maps \II_{\crn(n)} \to \II_m$ such that $s ∘ s'$ is the canonical crooked map $c_n$.
	
	\begin{proof}
		Clearly, if $n ≤ 1$, then $s = \id_{\II_n} = c_n$, or $s = r_n$ and $s ∘ r_n = \id_{\II_n} = c_n$.
		Suppose that $n > 1$, so $1 ≤ n - 1$.
		By the assumptions on $s$, either $s\fiber{n - 1} = \set{b_1, m - 1}$ for a unique break-point $b_1$, or $s\fiber{n - 1} = \set{m - 1}$, in which case we put $b_1 := m - 1$.
		Similarly, either $s\fiber{1} = \set{b_2, 1}$ for a unique break-point $b_2$, or $s\fiber{1} = \set{1}$, in which case we put $b_2 := 1$.
		We put $m_0 := b_1$, $m_1 := \abs{b_2 - b_1}$, and $m_2 := m - b_2$, so $m = m_0 ± m_1 + m_2$ depending on the order of $b_1$ and $b_2$.
		
		Let $s_0\maps \II_{m_0} \to \II_{n - 1}$ be the restriction $s\restr{[0, b_1]}$, i.e. $s ∘ e^m_{m_0, 0} = e^n_{n - 1, 0} ∘ s_0$.
		We see that every break-point of $s_0$ is a break-point of $s$, or it is equal to the point $m_0 = b_1$, where $s_0$ has a unique maximum.
		So $s_0$ satisfies the assumptions of our lemma, and by the induction there is a simplicial surjection $s'_0\maps \II_{\crn(n - 1)} \to \II_{m_0}$ such that $s_0 ∘ s'_0 = c_{n - 1}$.
		Similarly, for the translation $s_2\maps \II_{m_2} \to \II_{n - 1}$ of the restriction $s\restr{[b_2, m]}$ satisfying $s ∘ e^m_{m_2, b_2} = e^n_{n - 1, 1} ∘ s_2$, there is a simplicial surjection $s'_2\maps \II_{\crn(n - 1)} \to \II_{m_2}$ such that $s_2 ∘ s'_2 = c_{n - 1}$.
		
		Suppose that $b_1 ≤ b_2$.
		We have that $s\restr{[b_1, b_2]}$ has a unique maximum $n - 1$ at $b_1$ and a unique minimum $1$ at $b_2$.
		Let $s_1\maps \II_{m_1} \to \II_{n - 2}$ be the unique simplicial map satisfying $s ∘ e^m_{m_1, b_1} = e^n_{n - 2, 1} ∘ s_1$.
		The map $s_1 ∘ r$ satisfies the assumptions of our lemma, and so by the induction there is a simplicial map $s'_1\maps \II_{\crn(n - 2)} \to m_1$ such that $s_1 ∘ r ∘ s'_1 = c_{n - 2}$.
		
		Since the last point $m_0$ is the unique maximum of $s_0$ and since $s_0 ∘ s'_0 = c_{n - 2}$, we have that $s'_0$ maps the last point $\crn(n - 1)$ to $m_0$.
		Hence, $(e^m_{m_0, 0} ∘ s'_0)$ maps the last point to $m_0 = b_1 ∈ \II_m$.
		Similarly, $(e^m_{m_1, b_1} ∘ r ∘ s'_1 ∘ r)$ maps the first point to the unique point in $[b_1, b_2]$ that gets mapped to $n - 1$ by $s$, i.e. to $b_1$, and the last point to $b_2$.
		Finally, $(e^m_{m_2, b_2} ∘ s'_2)$ maps the first point to $b_2$.
		Hence, we may define the concatenation $s'\maps \II_{\crn(n)} \to \II_m$ of the maps satisfying
		\begin{align*}
			s ∘ s' :={}& s ∘ \bigl((e^m_{m_0, 0} ∘ s'_0) \concat (e^m_{m_1, b_1} ∘ r ∘ s'_1 ∘ r) \concat (e^m_{m_2, b_2} ∘ s'_2)\bigr) \\
				={}& (e^n_{n - 1, 0} ∘ s_0 ∘ s'_0) \concat (e^n_{n - 2, 1} ∘ s_1 ∘ r ∘ s'_1 ∘ r) \concat (e^n_{n - 1, 1} ∘ s_2 ∘ s'_2) \\
				={}& (e^n_{n - 1, 0} ∘ c_{n - 1}) \concat (e^n_{n - 2, 1} ∘ c_{n - 2} ∘ r) \concat (e^n_{n - 1, 1} ∘ c_{n - 1}) = c_n.
		\end{align*}
		
		If $b_2 ≤ b_1$, we change the definitions accordingly.
		We consider the restriction $s\restr{[b_2, b_1]}$ and the unique simplicial map $s_1\maps \II_{m_1} \to \II_{n - 2}$ such that $s ∘ e^m_{m_1, b_2} = e^n_{n - 2, 1} ∘ s_1$.
		Again, $s_1$ satisfies the assumptions of our lemma, and so by the induction there is a simplicial surjection $s'_1\maps \II_{\crn(n - 2)} \to \II_{m_1}$ such that $s_1 ∘ s'_1 = c_{n - 2}$.
		We define the second part of $s'$ by $e^m_{m_1, b_2} ∘ s'_1 ∘ r$.
		Then we have $s ∘ (e^m_{m_1, b_2} ∘ s'_1 ∘ r) = e^n_{n - 2, 1} ∘ s_1 ∘ s'_1 ∘ r = e^n_{n - 2, 1} ∘ c_{n - 2} ∘ r$.
	\end{proof}
\end{lemma}

\begin{lemma} \label{thm:simplest_crooked}
	Let $s\maps \II_m \to \II_n$ be a crooked simplicial surjection.
	There is a simplicial surjection $s'\maps \II_m \to \II_{\crn(n)}$ such that $s = c_n ∘ s'$.
	So the canonical crooked surjections are the simplest crooked surjections.
	
	\begin{proof}
		If $n ≤ 1$, then this is easy to see.
		Otherwise, we proceed by induction.
		Note that if $s(0) = c_n(s'(0))$ is an end-point of $\II_n$, then $s'(0)$ is an end-point of $\II_{\crn(n)}$.
		The same is true for $m$ instead of $0$.
		
		First, suppose that $s$ has a unique minimum at $0$ and a unique maximum at $m$.
		Let $b_1 := \min(s\fiber{n - 1})$ and $b_2 := \max(s\fiber{1})$.
		Since $s$ is crooked and $n ≥ 2$, we have $b_1 ≤ b_2$.
		Let $m_0 := b_1$, $m_1 := b_2 - b_1$, and $m_2 := m - b_2$, so $m = m_0 + m_1 + m_2$.
		We consider the crooked restrictions $s\restr{[0, b_1]}$, $s\restr{[b_1, b_2]}$, $s\restr{[b_2, m]}$, the corresponding simplicial surjections $s_0\maps \II_{m_0} \to \II_{n - 1}$, $s_1\maps \II_{m_1} \to \II_{n - 2}$, and $s_2\maps \II_{m_2} \to \II_{n - 1}$, and their factorizations via simplicial surjections $s'_0\maps \II_{m_0} \to \II_{\crn(n - 1)}$, $s'_1\maps \II_{m_1} \to \II_{\crn(n - 2)}$, and $s'_2\maps \II_{m_2} \to \II_{\crn(n - 1)}$ obtained by the induction.
		The maps satisfy the following conditions.
		We also recall some identities coming from the inductive definition of $c_n$.
		\begin{align*}
			s ∘ e^m_{m_0, 0} &= e^n_{n - 1, 0} ∘ s_0,
				& s_0 &= c_{n - 1} ∘ s'_0, 
				& c_n ∘ e^{\crn(n)}_{\crn(n - 1), 0} &= e^n_{n - 1, 0} ∘ c_{n - 1}, \\
			s ∘ e^m_{m_1, b_1} &= e^n_{n - 2, 1} ∘ s_1,
				&\quad s_1 &= c_{n - 2} ∘ s'_1, 
				& c_n ∘ e^{\crn(n)}_{\crn(n - 2), \crn(n - 1)} &= e^n_{n - 2, 1} ∘ c_{n - 2} ∘ r, \\
			s ∘ e^m_{m_2, b_2} &= e^n_{n - 1, 1} ∘ s_2,
				& s_2 &= c_{n - 1} ∘ s'_2, 
				&\quad c_n ∘ e^{\crn(n)}_{\crn(n - 1), \crn(n) - \crn(n - 1)} &= e^n_{n - 1, 1} ∘ c_{n - 1}.
		\end{align*}
		Together, we have
		\begin{align*}
			s &= s ∘ (e^m_{m_0, 0} \concat e^m_{m_1, b_1} \concat e^m_{m_2, b_2}) \\
				&= (e^n_{n - 1, 0} ∘ s_0) \concat (e^n_{n - 2, 1} ∘ s_1) \concat (e^n_{n - 1, 1} ∘ s_2) \\
				&= (e^n_{n - 1, 0} ∘ c_{n - 1} ∘ s'_0) 
					\concat (e^n_{n - 2, 1} ∘ c_{n - 2} ∘ r ∘ r ∘ s'_1) 
					\concat (e^n_{n - 1, 1} ∘ c_{n - 1} ∘ s'_2) \\
				&= c_n ∘ \bigl((e^{\crn(n)}_{\crn(n - 1), 0} ∘ s'_0) 
					\concat (e^{\crn(n)}_{\crn(n - 2), \crn(n - 1)} ∘ r ∘ s'_1) 
					\concat (e^{\crn(n)}_{\crn(n - 1), \crn(n) - \crn(n - 1)} ∘ s'_2) \bigr) \\
				& =: c_n ∘ s'.
		\end{align*}
		The definition of $s'$ is correct since the surjections $s'_0$ and $s'_2$ fix the end-points, $s'_1$ reverse the end-points, and the connecting values are $\crn(n - 1)$ and $\crn(n - 1) + \crn(n - 2) = \crn(n) - \crn(n - 1)$.
		
		Now suppose that the crooked surjection $s\maps \II_m \to \II_n$ is arbitrary.
		Let $[a_i, b_i]$, $i ≤ k$, be the increasing enumeration of intervals in $\II_m$ with disjoint interiors such that $s(a_i) = 0$ and $s(b_i) = n$ or vice versa, and the restriction $s\restr{[a_i, b_i]}$ has unique extrema at the end-points.
		We may apply the special case of our lemma to the map $s_i$ or $s_i ∘ r$ where $s_i := (s ∘ e^m_{b_i - a_i, a_i})\maps \II_{b_i - a_i} \to \II_n$.
		Hence, there is a simplicial surjection $s'_i\maps \II_{\crn(n)} \to \II_{b_i - a_i}$ such that $s_i = c_n ∘ s'_i$.
		
		Let us consider the restriction $s\restr{[0, a_0]}$ and suppose $s(a_0) = n$.
		The other case $s(a_0) = 0$ is analogous.
		Let $v := \min(s\restr{[0, a_0]}) ∈ (0, n)$, let $n' := n - v$, and let $t_0\maps \II_{a_0} \to \II_{n'}$ be the crooked simplicial surjection such that $s ∘ e^m_{a_0, 0} = e^n_{n', v} ∘ t_0$.
		By the induction there is a simplicial surjection $t'_0\maps \II_{a_0} \to \II_{\crn(n')}$ such that $t_0 = c_{n'} ∘ t'_0$.
		It is easy to see that $c_n ∘ e^{\crn(n)}_{\crn(n'), \crn(n) - \crn(n')} = e^n_{n', n - n'} ∘ c_{n'}$.
		Together, we have 
		\[
			s ∘ e^m_{a_0, 0} = e^n_{n', v} ∘ t_0 = e^n_{n', v} ∘ c_{n'} ∘ t'_0 = c_n ∘ e^{\crn(n)}_{\crn(n'), \crn(n) - \crn(n')} ∘ t'_0 =: c_n ∘ t''_0.
		\]
		We have $t''_0(a_0) = \crn(n)$, and so $t''_0 \concat s'_0$ is well-defined, and it satisfies $c_n ∘ (t''_0 \concat s'_0) = (s ∘ e^m_{a_0, 0}) \concat (s ∘ e^m_{b_0 - a_0, a_0}) = s ∘ e^m_{b_0, 0}$.
		Similarly, we handle the restriction $s\restr{[b_k, m]}$.
		
		Let us consider the restriction $s\restr{[b_i, a_{i + 1}]}$ for $i < k$.
		Suppose that $s(b_i) = s(a_{i + 1}) = 0$.
		The other case $s(b_i) = s(a_{i + 1}) = n$ is analogous.
		Let $n' := \max(s\restr{[b_i, a_{i + 1}]})$,
		and let $t_{i + 1}\maps \II_{a_{i + 1} - b_i} \to \II_{n'}$ be the crooked simplicial surjection such that $s ∘ e^m_{a_{i + 1} - b_i, b_i} = e^n_{n', 0} ∘ t_{i + 1}$.
		By the induction there is a simplicial surjection $t'_{i + 1}\maps \II_{a_{i + 1} - b_i} \to \II_{\crn(n')}$ such that $t_{i + 1} = c_{n'} ∘ t'_{i + 1}$.
		We have $c_n ∘ e^{\crn(n)}_{\crn(n'), 0} = e^n_{n', 0} ∘ c_{n'}$.
		Together, we have 
		\[
			s ∘ e^m_{a_{i + 1} - b_i, b_i} = e^n_{n', 0} ∘ t_{i + 1} = e^n_{n', 0} ∘ c_{n'} ∘ t'_{i + 1} = c_n ∘ e^{\crn(n)}_{\crn(n'), 0} ∘ t'_{i + 1} =: c_n ∘ t''_{i + 1}.
		\]
		We have $t''_{i + 1}(0) = t''_{i + 1}(a_{i + 1} - b_i) = 0$, and so $s'_i \concat t'_{i + 1} \concat s'_{i + 1}$ is well-defined, and it satisfies $c_n ∘ (s'_i \concat t'_{i + 1} \concat s'_{i + 1}) = s ∘ e^m_{b_{i + 1} - a_i, a_i}$.
		
		Altogether, we have partitioned $\II_m$ to subintervals $[d_i, d_i + m_i]$, $i < 2(k + 1) + 1$, such that $s ∘ e^m_{m_i, d_i} = c_n ∘ q_i$ for suitable simplicial maps $q_i\maps \II_{m_i} \to \II_{\crn(n)}$.
		The maps $q_i$ can be concatenated to a single map $s'\maps \II_m \to \II_{\crn(n)}$, and $c_n ∘ s'$ is the concatenation of the maps $c_n ∘ q_i$, and so equal to $s$.
	\end{proof}
\end{lemma}

\begin{lemma} \label{thm:crooked_approximation}
	For every $\I$-map $g$ and every $ε > 0$ there is $n_0 ∈ ω$ such that for every $n ≥ n_0$ there is an $\I$-map $g'$ such that $g ∘ g' ≈_ε \geom{c_n}$.
	
	\begin{proof}
		First, let $g'_0$ be an $\I$-map such that $g ∘ g'_0$ fixes both $0$ and $1$.
		Next, it is well-known that every continuous map $\II \to \II$ can be approximated by a piecewise linear map (and it also follows from Lemma~\ref{thm:simplicial_approximation}), so there are points $0 = x_0 < x_1 < \cdots < x_k = 1$ and the corresponding values $y_i$, $i ≤ k$, such that the piecewise linear map $h\maps \II \to \II$ extending $\set{x_i \mapsto y_i: i ≤ k}$ satisfies $h ≈_ε g ∘ g'_0$.
		There is $n_0 ∈ ω$ such that for every $n ≥ n_0$ we can make sure that the values $y_i$, $i ≤ k$, are of the form $j_i / n$ for pairwise distinct values $j_i ∈ \II_n$ such that $j_0 = 0$ and $j_k = n$.
		Next, we rescale the intervals $[x_i, x_{i + 1}]$.
		Let $m_0 := 0$ and for every $0 < i ≤ k$ let $m_i := m_{i - 1} + \abs{j_i - j_{i - 1}} > m_{i - 1}$ (we know that $j_i ≠ j_{i - 1}$) and $m := m_k$.
		Let $h'$ be the piecewise linear homeomorphism extending $\set{m_i/m \mapsto x_i: i ≤ k}$, so $h ∘ h' = \geom{s}$ for a simplicial surjection $s\maps \II_m \to \II_n$.
		By the construction, $s$ satisfies the assumptions of Lemma~\ref{thm:simplicial_crooked_factorization}, and so there is a simplicial surjection $s'\maps \II_{m'} \to \II_m$ such that $s ∘ s' = c_n$.
		Together, for $g' := g'_0 ∘ h' ∘ \geom{s'}$ we have $g ∘ g' ≈_ε h ∘ h' ∘ \geom{s'} = \geom{s ∘ s'} = \geom{c_n}$.
	\end{proof}
\end{lemma}

\begin{lemma} \label{thm:simplicial_approximation}
	Let $f\maps \II \to \II$ be a continuous map and let $0 < n ∈ ω$.
	If $ε > \frac{1}{2n}$, then there is $m ∈ ω$ and a simplicial map $s\maps \II_m \to \II_n$ such that $\geom{s} ≈_ε f$.
	Moreover, if $f$ is surjective and $ε ≤ \frac{1}{n}$, then $s$ is surjective as well.
	If additionally $f$ is $δ$-crooked for $δ ≤ 2(\frac{3}{4n} - ε)$, then $s$ is crooked.
	Note that the choice $ε = \frac{5}{8n}$ and $δ = \frac{1}{4n}$ works in all cases.
	
	\begin{proof}
		For every $j ≤ n$ let $V_j := (j/n - ε, j/n + ε)$.
		If $ε > \frac{1}{2n}$, the sets $V_j$, $j ≤ n$, cover $\II$.
		We may suppose without loss of generality that $ε ≤ \frac{1}{n}$ so that $V_j ∩ V_{j'} ≠ ∅$ if and only if $\abs{j - j'} ≤ 1$.
		By the continuity of $f$ and the compactness of $\II$ there is a finite open cover $\tuple{U_i: i ≤ N}$ of $\II$ by intervals such that for every $i ≤ N$ there is $j(i) ≤ n$ with $f\im{U_i} ⊆ V_{j(i)}$.
		Moreover, we may suppose that no interval $U_i$ is covered by the others, and so there is $a_i ∈ U_i \setminus ⋃_{i' ≠ i} U_{i'}$.
		We may re-enumerate the intervals $U_i$ so that the sequence $\tuple{a_i}_{i ≤ N}$ is strictly increasing.
		Because of the one-dimensionality of $\II$, it follows that $U_i ∩ U_{i'} ≠ ∅$ if and only if $\abs{i - i'} ≤ 1$, 
		and also that $0 ∈ U_0$ and $1 ∈ U_N$.
		
		Let $m ∈ ω$ be large enough so that for every $i < N$ there is $k_i ∈ \II_m$ satisfying $B_i := [k_i/m, (k_i + 1)/m] ⊆ U_i ∩ U_{i + 1}$.
		Let us define a map $s\maps \II_m \to \II_n$ by putting $s(k_i) := j(i)$ and $s(k_i + 1) := j(i + 1)$ for $i < N$ and letting $s$ be constant on every interval $[k_i + 1, k_{i + 1}]$, $i < N - 1$, as well as on $[0, k_0]$ and $[k_{N - 1} + 1, m]$.
		Let $A_i$, $i ≤ N$, denote the corresponding subintervals of $\II$, i.e. scaled down by $m$.
		It follows from $U_i ∩ U_{i + 1} ≠ ∅$ that $V_{j(i)} ∩ V_{j(i + 1)} ≠ ∅$ and $\abs{j(i) - j(i + 1)} ≤ 1$.
		Hence, $s$ is a simplicial map.
		
		Let $g := \geom{s}$.
		We show $g ≈_ε f$.
		$\II$ is covered by the intervals $A_i$, $i ≤ N$, and $B_i$, $i < N$.
		For every $i ≤ N$ we have $f\im{A_i} ⊆ V_{j(i)}$ since $A_i ⊆ U_i$.
		At the same time, $g$ is constant on $A_i$ with the value $j(i)/n$ being the center of $V_{j(i)}$.
		For $i < N$ we have $f\im{B_i} ⊆ V_{j(i)} ∩ V_{j(i + 1)} = (y' - ε, y + ε)$ where $\set{y, y'} = \set{j(i)/n, j(i + 1)/n}$ with $y < y'$.
		At the same time, $g\im{B_i} ⊆ [y, y']$.
		Together, $\geom{s} ≈_ε f$.
		
		If $ε ≤ \frac{1}{n}$ and $f$ is surjective, then there is $x ∈ \II$ such that $f(x) = 0$ and $g(x) < \frac{1}{n}$.
		There is $i ∈ \II_m$ such that $x ∈ [i/m, (i + 1)/m]$, and so at least one of $g(i/m), g((i + 1)/m) < \frac{1}{n}$.
		Hence, $s(i) = 0$ or $s(i + 1) = 0$.
		Similarly, there is $i' ∈ \II_m$ such that $s(i') = n$, and so $s$ is surjective.
		If $f$ is $δ$-crooked for $δ ≤ 2(\frac{3}{4n} - ε)$, then $g$ is $(δ + 2ε)$-crooked by Proposition~\ref{thm:crooked_calculus}, and $δ + 2ε ≤ \frac{3}{2n}$.
		Hence, $s$ is crooked by Proposition~\ref{thm:n_crooked}.
	\end{proof}
\end{lemma}

\begin{proof}[Proof of Theorem~\ref{thm:crooked_factorization}]
	Let $g$ be an $\I$-map and let $ε > 0$.
	By Lemma~\ref{thm:crooked_approximation} there is $n ∈ ω$ and an $\I$-map $g'$ such that $\frac{1}{n} ≤ ε/2$ and $g ∘ g' ≈_{ε/2} \geom{c_n}$.
	We put $δ := \frac{1}{4n}$ and $ε' := \frac{5}{8n} < \frac{1}{n} ≤ ε/2$.
	
	Let $f$ be any $δ$-crooked $\I$-map.
	By Lemma~\ref{thm:simplicial_approximation} there is a crooked simplicial surjection $s\maps \II_m \to \II_n$ such that $f ≈_{ε'} \geom{s}$.
	By Lemma~\ref{thm:simplest_crooked}, there is a simplicial surjection $s'\maps \II_m \to \II_{\crn(n)}$ such that $s = c_n ∘ s'$.
	Together, we have $f ≈_{ε/2} \geom{s} = \geom{c_n} ∘ \geom{s'} ≈_{ε/2} g ∘ (g' ∘ \geom{s'}$).
\end{proof}

\section{Generic objects}
\label{sec:generic}

We recall the key notion of a \emph{generic object} and show how it applies to the pseudo-arc.
This general notion makes sense in any category and is based on the abstract Banach–Mazur game, which was considered by the second author~\cite{Kubis22}.

\begin{definition} \label{def:BM}
	Let $\K$ be a category.
	We consider the abstract \emph{Banach–Mazur game} played in $\K$.
	The play scheme is as follows.
	We have two players, \emph{Eve} and \emph{Odd}.
	Eve starts the play by picking a $\K$-map $f_0\maps X_0 \from X_1$.
	Odd responds by picking a $\K$-map $f_1\maps X_1 \from X_2$.
	Eve continues with a $\K$-map $f_2\maps X_2 \from X_3$, and so on.
	The outcome of the play is a sequence $f_*$ in $\K$.
	This scheme is denoted by $\BM(\K)$.

	To obtain an actual game, we need to specify the goal for the players.
	For a class $\G$ of sequences in $\K$, $\BM(\K, \G)$ denotes the game played according to the scheme $\BM(\K)$ where Odd, the second player, wins a play $f_*$ if $f_* ∈ \G$.
	Often, we fix a bigger category $\L ⊇ \K$ and a family $\F$ of $\L$-objects, and we consider $\G_\F$: the class of all sequences in $\K$ that have some $X ∈ \F$ as a limit.
	Since the limit of a sequence is determined up to isomorphism, it makes sense to restrict to isomorphism-closed families $\F$.
	In fact, we are mostly interested in the case when $\F$ is the isomorphism class of a single $\L$-object $X$.
	The corresponding class of all sequences in $\K$ whose limit in $\L$ is $X$ is denoted by $\G_X$.
	
	We say that a class $\G$ of sequences in $\K$ is \emph{generic in $\K$} if Odd has a winning strategy in $\BM(\K, \G)$.
	Similarly, for a category $\L ⊇ \K$ and an isomorphism-closed family $\F$ of $\L$-objects we say that $\F$ is \emph{generic in $\tuple{\K, \L}$} if Odd has a winning strategy in $\BM(\K, \G_\F)$.
	In particular, an $\L$-object $X$ is \emph{generic in $\tuple{\K, \L}$} if Odd has a winning strategy in $\BM(\K, \G_X)$.
	If the category $\L$ is clear from the context, we may say “generic over $\K$” instead of “generic in $\tuple{\K, \L}$”, whereas “generic in $\L$” means “generic in $\tuple{\L, \L}$”, i.e. the game is played in $\L$ and not in $\K$.
\end{definition}

\begin{observation}
	The generic object is unique up to isomorphism.
	This is because the rules of the game are symmetric, and because a subsequence has the same limit as the original sequence.
	Let $X, Y$ be generic objects in $\tuple{\K, \L}$.
	Let $f_0$ be Eve's first move.
	Then Odd responds by $f_1$ according to his winning strategy for $X$.
	We let $g_0 := f_1$ be Eve's first move in a parallel game, and we let Odd respond by $g_1$ according to his winning strategy for $Y$.
	Then we put $f_2 := g_1$, we continue the same way.
	In the end, the limit of $f_*$ is $X$, and the limit of $g_*$ is $Y$.
	But $g_*$ is just $f_*$ without the first term.
\end{observation}

A more elaborate construction gives a stronger result.
\begin{proposition}[{\cite[Theorem~6.2]{Kubis22}}]
	Let $\K ⊆ \L$ be categories.
	If $\set{\F_k}_{k ∈ ω}$ are generic families in $\tuple{\K, \L}$, then $⋂_{k ∈ ω} \F_k$ is also generic in $\tuple{\K, \L}$.
\end{proposition}

Let us apply the notion of generic property to hereditarily indecomposable continua.

\begin{proposition} \label{thm:crooked_sequence_generic}
	The family of all crooked sequences is generic in $\I$.
	
	\begin{proof}
		We refine Construction~\ref{thm:crooked_sequence_exists}.
		Let $ε_0 := 1$, and let $f_0$ be Eve's first more.
		After every move $f_n$, we fix $ε_{n + 1} > 0$ such that every $f_{k, n + 1}$, $k < n + 1$, is $\tuple{ε_k / 2^{n + 1 - k}, ε_{n + 1}}$-continuous.
		The Odd's strategy is to play for every odd $n ∈ ω$ an $ε_n$-crooked $\I$-map $f_n$.
		This is a winning strategy since for every $n ∈ ω$ and $ε > 0$ there is odd number $n' ≥ n$ such that $ε_n / 2^{n' - n} < ε$, and so the map $f_{n, n'}$ is $\tuple{ε, ε_{n'}}$-continuous, and the map $f_{n, n' + 1}$ is $ε$-crooked.
	\end{proof}
\end{proposition}

Since every crooked $\I$-sequence has the pseudo-arc as limit (see Section~\ref{sec:pseudo}), we obtain the following.
\begin{corollary} \label{thm:generic_pseudoarc}
	The pseudo-arc is a generic object in $\tuple{\I, σ\I}$.
\end{corollary}

A generic property is also generic over every so-called \emph{dominating subcategory} (cf.~\cite[Section~3.2]{Kubis14}), and vice versa.

\begin{definition} \label{def:dominating_subcategory}
	A subcategory $\D ⊆ \K$ is called \emph{dominating} if it is both
	\begin{itemize}
		\item \emph{cofinal}, i.e. for every $\K$-object $X$, there is a $\K$-map $f\maps X \from Y$ from a $\D$-object $Y$,
		\item \emph{absorbing}, i.e. for every $\D$-object $X$ and every $\K$-map $f\maps X \from Y$ there is a $\K$-map $g\maps Y \from Z$ such that $f ∘ g ∈ \D$.
	\end{itemize}
	Note that every full cofinal subcategory is dominating.
\end{definition}

\begin{example} \label{ex:Peano}
	Let $\PeanoS$ be the category of all Peano continua and all continuous surjections.
	Clearly, $\I$ is a full dominating subcategory of $\PeanoS$.
	In fact, $\PeanoS$ is the largest subcategory of $\MCptS$ in which $\I$ is dominating, and every full subcategory of $\PeanoS$ containing a non-degenerate space is dominating since $\II$ is a continuous image of every non-degenerate connected Tychonoff space, and every Peano continuum is a continuous image of $\II$.
\end{example}

\begin{proposition}[a simplified version of {\cite[Theorem~6.1]{Kubis22}}] \label{thm:strictly_dominating_subcategory}
	Let $\K ⊆ \L$ be categories and let $\D ⊆ \K$ be a dominating subcategory.
	A family $\F$ of $\L$-objects is generic in $\tuple{\K, \L}$ if and only if it is generic in $\tuple{\D, \L}$.
	
		%
\end{proposition}

Directly from Corollary~\ref{thm:generic_pseudoarc}, Example~\ref{ex:Peano}, and Proposition~\ref{thm:strictly_dominating_subcategory} we obtain the following theorem.

\begin{theorem} \label{thm:generic_Peano}
	The pseudo-arc $\PP$ is generic over every dominating subcategory $\K ⊆ \PeanoS$ (in particular, over every full subcategory $\K$ containing a non-degenerate space).
	That is, when playing $\BM(\K)$, the second player may always force the limit of the resulting sequence to be $\PP$.
\end{theorem}

Domination can be defined also for sequences – in presence of the amalgamation property this is called the \emph{Fraïssé sequence} (see \cite[Section~3]{Kubis14} and Definition~\ref{def:Fraisse_sequence} with the paragraph before).

\begin{definition} \label{def:dominating_sequence}
	A sequence $\tuple{X_*, f_*}$ in a category $\K$ is called \emph{dominating} if it is both
	\begin{itemize}
		\item \emph{cofinal}, i.e. for every $\K$-object $Y$, there is $n ∈ ω$ and a $\K$-map $f\maps Y \from X_n$, and
		\item \emph{absorbing}, i.e. for every $n ∈ ω$ and every $\K$-map $f\maps X_n \from Y$ there is $n' ≥ n$ and a $\K$-map $g\maps Y \from X_{n'}$ such that $f ∘ g = f_{n, n'}$.
	\end{itemize}
\end{definition}

A stronger form of the following proposition was proved by Kubiś~\cite[Theorem~6.3]{Kubis22}.
We give a proof in the simpler situation for illustration.

\begin{proposition} \label{thm:strictly_dominating_sequence}
	A dominating sequence $\tuple{X_*, f_*}$ in $\K$ serves as a winning strategy for Odd in $\BM(\K)$.
	If $\tuple{X_*, f_*}$ has a limit $X_∞$ in a larger category $\L ⊇ \K$, then $X_∞$ is a generic object in $\tuple{\K, \L}$.
	
	\begin{proof}
		Let $g_0\maps Y_0 \from Y_1$ be Eve's first move in $\BM(\K)$.
		From the cofinality, Odd may respond with a map $g_1\maps Y_1 \from X_{n_0}$ for some $n_0 ∈ ω$.
		Eve continues with a map $g_2\maps X_{n_0} \from Y_3$.
		From the absorption, Odd may respond with a map $g_3\maps Y_3 \from X_{n_1}$ for some $n_1 ≥ n_0$ such that $g_2 ∘ g_3 = f_{n_0, n_1}$.
		We continue according to this strategy for Odd and end up with a play $g_*$ such that $g_{2k, 2k + 2} = f_{n_{k - 1}, n_k}$ for $k > 1$.
		It follows that $\lim g_* = \lim f_*$ if any of the limits exists.
	\end{proof}
\end{proposition}

We know that the pseudo-arc is generic over $\I$, but it turns out that there is no dominating sequence in the (ordinary) category $\I$ (see Observation~\ref{thm:mountain_climbing}).
However, the situation changes in the approximate setting.

\subsection{Approximate setting}

Recall that by a theorem of Brown~\cite{Brown60} (see also \cite[Chapter~4]{IM12}), two sequences $\tuple{X_*, f_*}$, $\tuple{Y_*, g_*}$ of metric compacta and continuous maps such that $X_* = Y_*$ and $f_n ≈_{ε_n} g_n$ for every $n ∈ ω$ for a suitable sequence $\tuple{ε_n}_{n ∈ ω}$ of strictly positive numbers have homeomorphic limits $X_∞ \homeo Y_∞$.

We shall prove an abstract and refined version of Brown's approximation theorem (Corollary~\ref{thm:Brown}) as well as an approximate back and forth construction (Corollary~\ref{thm:back_and_forth}) that will allow us to generalize the results on dominating subcategories and sequences to the approximate setting.
Since so far the notion of domination was defined in any category, we do not want to limit ourselves to the concrete case of metric compacta for the approximate setting, and so we introduce the following abstraction.

\begin{definition} \label{def:MU-category}
	By an \emph{MU-category} we mean a category $\K$ such that
	\begin{enumerate}
		\item every hom-set $\K(X, Y)$ is an $∞$-metric space;
		\item for every map $f ∈ \K$ we have $d(g ∘ f, h ∘ f) ≤ d(g, h)$ for all compatible maps $g, h ∈ \K$,
			i.e. for every $Y ∈ \Ob(\K)$ the map $(- ∘ f)\maps \K(\cod(f), Y) \to \K(\dom(f), Y)$ between $∞$-metric spaces is non-expansive;
		\item for every map $f ∈ \K$ and $ε > 0$ there is $δ > 0$ such that $f$ is \emph{$\tuple{ε, δ}$-continuous},
			i.e. $d(g, h) < δ$ implies $d(f ∘ g, f ∘ h) < ε$ for all compatible maps $g, h ∈ \K$;
			in other words, for every $X ∈ \Ob(\K)$ the map $(f ∘ -)\maps \K(X, \dom(f)) \to \K(X, \cod(f))$ is uniformly continuous, and moreover the continuity is uniform also across all domains $X$.
	\end{enumerate}
	The letters “M” and “U” in the name refer to “metric” and “uniformity”, respectively.
	By a \emph{metric-enriched category} we mean an MU-category $\K$ such that condition~(iii) is replaced by a stronger condition
	\begin{enumerate}
		\item[(iii')] Every map $f ∈ \K$ is \emph{non-expansive}, i.e. $d(f ∘ g, f ∘ h) ≤ d(g, h)$ for all compatible maps $g, h ∈ \K$,
			i.e. for every $X ∈ \Ob(\K)$ the map $(f ∘ -)\maps \K(X, \dom(f)) \to \K(X, \cod(f))$ non-expansive.
	\end{enumerate}
	An MU-category $\K$ is called \emph{locally complete} if all $∞$-metric spaces $\K(X, Y)$ are complete.
\end{definition}

\begin{remark}
	Metric-enriched categories are exactly categories enriched over the monoidal category of all $∞$-metric spaces and all non-expansive maps where the monoidal product is the cartesian product endowed with the $\ell_1$-metric.
	In the context of Fraïssé theory, metric-enriched categories were considered by Kubiś~\cite{Kubis13}.
	
	Note that the opposite category of an MU-category is not necessarily an MU-category.
	In fact, metric-enriched categories are exactly MU-categories whose opposite categories are also MU-categories.
\end{remark}

\begin{example}
	The category $\MetU$ of all metric spaces and all uniformly continuous maps is an MU-category when the hom-sets are endowed with the supremum $∞$-metric (i.e. we put $d(f, g) = \sup_{x ∈ X} d(f(x), g(x)) ∈ [0, ∞]$ for uniformly continuous maps $f, g\maps X \to Y$ between metric spaces).
	In fact, the notion of MU-category is an abstraction of this particular category.
	Condition~(ii) follows from the fact that the supremum metric is really measured in the codomain space, and it holds for all, not necessarily uniformly continuous, maps.
	On the other hand, condition~(iii) holds because the maps considered are uniformly continuous.
	Non-expansive maps in the sense of (iii') are exactly maps that are non-expansive in the classical sense.
	Hence, the wide subcategory $\Met ⊆ \MetU$ of all non-expansive maps is a metric-enriched category.
	
	Let $\CMetU ⊆ \MetU$ be the full subcategory of all complete metric spaces.
	It is well-known that a uniformly Cauchy sequence of uniformly continuous maps between complete metric spaces uniformly converges to a uniformly continuous map.
	Hence, $\CMetU$ is a locally complete MU-category.
	Note that non-expansivity is a closed property: if for a map $f$ we have $d(f ∘ g, f ∘ h) > d(g, h)$ for some maps $g, h$, then the same is true also for every $f'$ sufficiently close to $f$.
	Hence, $\CMet = \CMetU ∩ \Met$ is a locally complete metric-enriched category.
\end{example}

\begin{definition}
	By an \emph{MU-continuous functor} (or \emph{MU-functor} for short) we mean a functor $F\maps \K \to \L$ between MU-categories such that for every $Y ∈ \Ob(\K)$ and every $ε > 0$ there is $δ > 0$ such that $f ≈_δ g$ implies $F(f) ≈_ε F(g)$ for every $\K$-object $X$ and every $f, g ∈ \K(X, Y)$.
	That means, for a fixed $\K$-object $Y$, not only the maps $F\maps \K(X, Y) \to \L(F(X), F(Y))$ are uniformly continuous, but we also have uniformity across all domains $X$.
	
	MU-continuity serves as a base for other category-theoretic properties of functors in the context of MU-categories:
	an \emph{MU-isomorphism} is an isomorphism $F$ of MU-categories such that both $F$ and $F\inv$ are MU-functors,
	an \emph{MU-equivalence} is an equivalence of MU-categories consisting of a pair of MU-functors.
\end{definition}

\begin{example} \label{ex:cpt_mucat}
	The category $\MCpt$ of all metrizable compacta and all continuous maps can be viewed as an MU-category.
	By choosing a compatible metric on every metrizable compactum, $\MCpt$ becomes a full subcategory of $\CMetU$ since every continuous map on a compact space is uniformly continuous.
	Moreover, the particular choice of metrics does not matter since for any other choice the functor $\id_{\MCpt}$ becomes an MU-isomorphism.
	So we obtain a unique MU-category up to isomorphism.
	
	As a full subcategory of the locally complete MU-category $\CMetU$, $\MCpt$ is itself locally complete.
	Let $\MCptS$ be the subcategory of $\MCpt$ of all non-empty metrizable compacta and all continuous surjections.
	It is easy to see that the limit of a uniformly Cauchy sequence of uniformly continuous surjections between complete metric spaces has a dense image, and so if the spaces are compact, the limit map is surjective.
	This shows that $\MCptS$ is also locally complete.
	Note that our category $σ\I$ is locally complete as well as a full subcategory of $\MCptS$.
\end{example}

\begin{definition} \label{def:discrete}
	We say that an MU-category $\K$ is \emph{discrete at $Y ∈ \Ob(\K)$} if there is $ε > 0$ such that $f ≈_ε g$ implies $f = g$ for every $\K$-object $X$ and all $\K$-maps $f, g\maps X \to Y$.
	This means not only that the metric space $\K(X, Y)$ is (uniformly) discrete, but also that the discreteness is uniform across all domains $X$ (when $Y$ is fixed).
	The whole MU-category $\K$ is \emph{discrete} if it is discrete at every object.
	
	Observe that every functor from a discrete MU-category to an arbitrary MU-category is MU-continuous.
	Also, every discrete MU-category is locally complete.
	Moreover, every category can be viewed as a discrete metric-enriched category.
	It is enough to put on every hom-set $d(f, g) = 1$ if $f ≠ g$ and $d(f, f) = 0$.
\end{definition}

Now we shall prove abstract Brown's approximation theorem and a back and forth construction in the context of MU-categories.

\begin{definition}
	By an \emph{epsilon sequence} for a sequence $f_*$ in an MU-category $\K$ we mean a sequence $\tuple{ε_n}_{n ∈ ω}$ of strictly positive real numbers such that the map $f_{n, n'}$ is $\tuple{ε_n/2^{n' - n}, ε_{n'}}$-continuous for every $n ≤ n' ∈ ω$.
	Such sequences can be seen in \cite[Lemma~5]{MS63} and are closely related to the notion of \emph{Lebesgue sequences} \cite{Brown60}.
	
	Note that for a finite sequence $\tuple{ε_i}_{i < n}$ in $(0, ∞)$ there is $ε_n > 0$ such that $f_{i, n}$ is $\tuple{ε_i/2^{n - i}, ε_n}$-continuous for every $i < n$, so it is easy to build an epsilon sequence inductively.
	Also, if $\K$ is metric-enriched, then $\tuple{2^{-n}}_{n ∈ ω}$ is an epsilon sequence for every sequence in $\K$.
\end{definition}

We shall deal with cones and limits of sequences, which we recalled in Section~\ref{sec:preliminaries}.

\begin{proposition}[Cone transfer] \label{thm:cone_transfer}
	Let $\tuple{X_*, f_*}$ and $\tuple{Y_*, g_*}$ be sequences in a locally complete MU-category $\K$.
	Let $\tuple{ε_n}_{n ∈ ω}$ be an epsilon sequence for $g_*$ and let $φ_* = \tuple{φ_n\maps X_n \to Y_n}_{n ∈ ω}$ be a sequence of $\K$-maps such that $φ_n ∘ f_n ≈_{ε_n} g_n ∘ φ_{n + 1}$ for every $n ∈ ω$.
	\begin{enumerate}
		\item For every $n ≤ n' ≤ n'' ∈ ω$ we have $g_{n, n'} ∘ φ_{n'} ∘ f_{n', n''} ≈_{2ε_n/2^{n' - n}} g_{n, n''} ∘ φ_{n''}$.
		\item For every cone $\tuple{Z, γ_*}$ for $f_*$ there is a cone $C_{φ_*}(γ_*)$ defined by the formula $C_{φ_*}(γ_*)_n := \lim_{n' ≥ n} (g_{n, n'} ∘ φ_{n'} ∘ γ_{n'})$ for $n ∈ ω$.
			We also have $C_{φ_*}(γ_* ∘ h) = C_{φ_*}(γ_*) ∘ h$ for every $h\maps W \to Z$.
		\item For every limit $\tuple{X_∞, f_{*, ∞}}$ of $f_*$ and every limit $\tuple{Y_∞, g_{*, ∞}}$ of $g_*$ there is a unique $\K$-map $φ_∞\maps X_∞ \to Y_∞$ such that $C_{φ_*}(f_{*, ∞}) = g_{*, ∞} ∘ φ_∞$.
			It satisfies $φ_n ∘ f_{n, ∞} ≈_{2ε_n} g_{n, ∞} ∘ φ_∞$ for every $n ∈ ω$.
	\end{enumerate}
	Let, additionally, $\tuple{δ_n}_{n ∈ ω}$ be an epsilon sequence for $f_*$ and let $ψ_* = \tuple{ψ_n\maps Y_{n + 1} \to X_n}_{n ∈ ω}$ be a sequence of $\K$-maps such that $ψ_n ∘ g_{n + 1} ≈_{δ_n} f_n ∘ ψ_{n + 1}$, so we have also the map $C_{ψ_*}$ assigning to every cone for $g_*$ (or equivalently $g_*$ shifted by one) a cone for $f_*$.
	Moreover, suppose that $φ_n ∘ ψ_n ≈_{ε_n} g_n$ and $ψ_n ∘ φ_{n + 1} ≈_{δ_n} f_n$ for every $n ∈ ω$.
	\begin{enumerate}[resume]
		\item $C_{φ_*}$ and $C_{ψ_*}$ are mutually inverse bijections fixing domain objects of cones.
		\item If $f_{*, ∞}$ is a limit of $f_*$, then $C_{φ_*}(f_{*, ∞})$ is a limit of $g_*$.
			Hence, $f_*$ has a limit if and only if $g_*$ has a limit.
		\item For every limit $\tuple{X_∞, f_{*, ∞}}$ of $f_*$ and every limit $\tuple{Y_∞, g_{*, ∞}}$ of $g_*$ the map $φ_∞\maps X_∞ \to Y_∞$ is a $\K$-isomorphism.
	\end{enumerate}
	
	\begin{proof}
		(i) We have
		\begin{talign*}
			d(g_{n, n'} &∘ φ_{n'} ∘ f_{n', n''},\, g_{n, n''} ∘ φ_{n''}) \\
			&≤ ∑_{i < n'' - n'} d(g_{n, n' + i} ∘ φ_{n' + i} ∘ f_{n' + i, n''},\, 
				g_{n, n' + i + 1} ∘ φ_{n' + i + 1} ∘ f_{n' + i + 1, n''}) \\
			&= ∑_{i < n'' - n'} \text{“$g_{n, n' + i} ∘ 
				d(φ_{n' + i} ∘ f_{n' + i},\, g_{n' + i} ∘ φ_{n' + i + 1}) ∘ f_{n' + i + 1, n''}$”} \\
			&< ∑_{i < n'' - n'} \text{“$g_{n, n' + i} ∘ ε_{n' + i} ∘ f_{n' + i + 1, n''}$”} 
				≤ ∑_{i < n'' - n'} ε_n / 2^{n' + i - n} < 2 ε_n / 2^{n' - n}.
		\end{talign*}
		The quoted formulas merely suggest how the bound on the middle part gets transformed by the precomposition and the postcomposition.
		
		(ii) The sequence $\tuple{g_{n, n'} ∘ φ_{n'} ∘ γ_{n'}}_{n' ≥ n}$ is Cauchy in the complete space $\K(Z, Y_n)$ since for every $n ≤ n' ≤ n''$ we have 
		\[
			g_{n, n'} ∘ φ_{n'} ∘ γ_{n'} = g_{n, n'} ∘ φ_{n'} ∘ f_{n', n''} ∘ γ_{n''} ≈_{2ε_n/2^{n' - n}} g_{n, n''} ∘ φ_{n''} ∘ γ_{n''}
		\]
		by (i).
		Hence, the limit exists.
		The family of maps $C_{φ_*}(γ_*)$ is a cone since the postcompositions $(g_{n, n'} ∘ -)$ are continuous in any MU-category.
		Similarly, we have $C_{φ_*}(γ_* ∘ h) = C_{φ_*}(γ_*) ∘ h$ since the precomposition $(- ∘ h)$ is also continuous (even non-expansive).
		
		(iii) Since $C_{φ_*}(f_{*, ∞})$ is a cone for $g_*$ by (ii), $φ_∞$ is the unique $\K$-map $X_∞ \to Y_∞$ witnessing that $g_{*, ∞}$ is a limit of $g_*$.
		By (i) we have $g_n ∘ φ_{n + 1} ∘ f_{n + 1, n'} ≈_{2ε_n / 2} g_{n, n'} ∘ φ_{n'}$ for every $n' ≥ n$.
		Together,
		\begin{align*}
			g_{n, ∞} ∘ φ_∞ &= \lim_{n' ≥ n + 1} (g_{n, n'} ∘ φ_{n'} ∘ f_{n', ∞}) ≈_{≤ε_n} \lim_{n' ≥ n + 1} (g_n ∘ φ_{n + 1} ∘ f_{n + 1, n'} ∘ f_{n', ∞}) \\
				&= g_n ∘ φ_{n + 1} ∘ f_{n + 1, ∞} ≈_{ε_n} φ_n ∘ f_{n, ∞}.
		\end{align*}
		We do the extra step from $n$ to $n + 1$ to recover the strict inequality at the limit.
		
		(iv) Let $γ_*$ be a cone for $g_*$.
		We have \begin{talign*}
			C_{φ_*}(C_{ψ_*}(γ_*))_n &= \lim_{n' ≥ n} (g_{n, n'} ∘ φ_{n'} ∘ C_{ψ_*}(γ_*)_{n'}) \\
				&= \lim_{n' ≥ n} (g_{n, n'} ∘ φ_{n'} ∘ \lim_{n'' ≥ n'} (f_{n', n''} ∘ ψ_{n''} ∘ γ_{n'' + 1})) \\
				&= \lim_{n' ≥ n} \lim_{n'' ≥ n'} (g_{n, n'} ∘ φ_{n'} ∘ f_{n', n''} ∘ ψ_{n''} ∘ γ_{n'' + 1}),
		\end{talign*}
		\[\begin{taligned}
			\text{and \hspace{2em}} γ_n = g_{n, n''} ∘ g_{n''} ∘ γ_{n'' + 1} 
				&≈_{ε_n/ 2^{n'' - n}} g_{n, n''} ∘ (φ_{n''} ∘ ψ_{n''}) ∘ γ_{n'' + 1} \\
			(g_{n, n'} ∘ φ_{n'} ∘ f_{n', n''}) ∘ ψ_{n''} ∘ γ_{n'' + 1} 
				&≈_{2ε_n/2^{n' - n}} (g_{n, n''} ∘ φ_{n''}) ∘ ψ_{n''} ∘ γ_{n'' + 1}
		\end{taligned}\,\Bigr){=}\]
		It follows that $C_{φ_*}(C_{ψ_*}(γ_*))_n ≈_{≤3ε_n/2^{n'-n}} γ_n$ for every $n ≤ n' ∈ ω$, and so $C_{φ_*} ∘ C_{ψ_*} = \id$.
		The situation is symmetric, and we also obtain $C_{ψ_*} ∘ C_{φ_*} = \id$.
		
		(v) If $\tuple{X_∞, f_{*, ∞}}$ is a limit of $f_*$, then $C_{φ_*}(f_{*, ∞})$ is a cone for $g_*$ by (ii).
			For every cone $\tuple{Y, γ_*}$ for $g_*$ we have that $C_{ψ_*}(γ_*)$ is a cone for $f_*$, and so there is a unique $\K$-map $h\maps Y \to X_∞$ such that $C_{ψ_*}(γ_*) = f_{*, ∞} ∘ h$, which is equivalent to 
			\[
				γ_* = C_{φ_*}(C_{ψ_*}(γ_*)) = C_{φ_*}(f_{*, ∞} ∘ h) = C_{φ_*}(f_{*, ∞}) ∘ h
			\]
			by (iv), and so $C_{φ_*}(f_{*, ∞})$ is a limit of $g_*$.
		
		(vi) By (v) we know that $C_{φ_*}(f_{*, ∞})$ is a limit of $g_*$, and so $φ_∞$ is a $\K$-isomorphism since it is the limit factorization map.
	\end{proof}
\end{proposition}

\begin{corollary}[Abstract Brown's approximation theorem] \label{thm:Brown}
	Let $\tuple{X_*, f_*}$ and $\tuple{X_*, g_*}$ be sequences with the same sequences of objects in a locally complete MU-category $\K$.
	If $f_n ≈_{ε_n} g_n$ for every $n ∈ ω$ and an epsilon sequence $\tuple{ε_n}_{n ∈ ω}$ for both $f_*$ and $g_*$, then $f_*$ has a limit if and only if $g_*$ a limit, and any such limits are isomorphic.
	
	\begin{proof}
		We put $φ_n := \id_{X_n}$, $ψ_n := f_n$, and $δ_n := ε_n$ for every $n ∈ ω$.
		We have $φ_n ∘ ψ_n = φ_n ∘ f_n = f_n ≈_{ε_n} g_n = g_n ∘ φ_{n + 1}$ and $ψ_n ∘ φ_{n + 1} = f_n$ and also $ψ_n ∘ g_{n + 1} = f_n ∘ g_{n + 1} ≈_{δ_n} f_n ∘ f_{n + 1} = f_n ∘ ψ_{n + 1}$ since $g_{n + 1} ≈_{ε_{n + 1}} f_{n + 1}$ and $f_n$ is $\tuple{δ_n = ε_n, ε_{n + 1}}$-continuous.
		Now we can apply Proposition~\ref{thm:cone_transfer}.
	\end{proof}
\end{corollary}

\begin{corollary}[Back and forth] \label{thm:back_and_forth}
	Let $\tuple{X_*, f_*}$ and $\tuple{Y_*, g_*}$ be sequences in a locally complete MU-category.
	Let $f_{m_*}$ and $g_{n_*}$ be subsequences with epsilon sequences $\tuple{δ_k}_{k ∈ ω}$ and $\tuple{ε_k}_{k ∈ ω}$, respectively.
	If $h_*$ is a sequence in $\K$ such that for every $k ∈ ω$ we have that
	\begin{enumerate}
		\item $h_{2k}\maps Y_{n_k} \from X_{m_k}$ is $\tuple{ε_k/2, δ_k}$-continuous and $h_{2k + 1}\maps X_{m_k} \from Y_{n_{k + 1}}$ is $\tuple{δ_k/2, ε_{k + 1}}$-continuous,
		\item $h_{2k, 2k + 2} ≈_{ε_k/2} g_{n_k, n_{k + 1}}$ and $h_{2k + 1, 2k + 3} ≈_{δ_k/2} f_{m_k, m_{k + 1}}$,
	\end{enumerate}
	as in Figure~\ref{fig:zigzag},
	then $f_*$ has a limit if and only if $g_*$ has a limit, and for all such limits $\tuple{X_∞, f_{*, ∞}}$ and $\tuple{Y_∞, g_{*, ∞}}$ there is an isomorphism $h_∞\maps X_∞ \to Y_∞$ such that $h_0 ∘ f_{m_0, ∞} ≈_{2ε_0} g_{n_0, ∞} ∘ h_∞$.

\begin{figure}[!ht]
	\centering
	
	\begin{tikzpicture}[
			x = {(4em, 0em)},
			y = {(0em, 5em)},
			text height = 1.5ex,
			text depth = 0.25ex,
			label/.style = {
				edge label = {#1}, 
				every node/.style = {node font=\footnotesize},
			},
			label'/.style = {
				label = {#1},
				swap,
			},
		]
		
		\path
			++(0, 0) node (X0) {$X_0$}
			++(0, -1) node (Y0) {$Y_0$}
			
			++(0.8, 0) node (Yn0) {$Y_{n_0}$}
			++(1, +1) node (Xm0) {$X_{m_0}$}
			++(1, -1) node (Yn1) {$Y_{n_1}$}
			++(1, +1) node (Xm1) {$X_{m_1}$}
			++(1, -1) node (Yn2) {$Y_{n_2}$}
			++(1, +1) node (Xm2) {$X_{m_2}$}
		
			++(1, 0) node (Xend) {$\cdots$}
			++(0, -1) node (Yend) {$\cdots$}
			++(-0.5, 0.5) node[inner sep=0] (hend) {$\quad\cdots$}
			
			(Xend) ++(0.7, 0) node (Xinf) {$X_∞$}
			++(0, -1) node (Yinf) {$Y_∞$}
		;
		
		\path[node font=\small]
			(barycentric cs:Yn0=2.9,Xm0=1,Yn1=1.1) node {$≈_{ε_0/2}$}
			(barycentric cs:Xm0=2.9,Yn1=1,Xm1=1.1) node {$≈_{δ_0/2}$}
			(barycentric cs:Yn1=2.9,Xm1=1,Yn2=1.1) node {$≈_{ε_1/2}$}
			(barycentric cs:Xm1=2.9,Yn2=1,Xm2=1.1) node {$≈_{δ_1/2}$}
			(barycentric cs:Yn2=2.9,Xm2=1,Yend=1.1) node {$≈_{ε_2/2}$}
		;
		
		\graph{
			(X0) <- (Xm0) <-[label=$g_{m_0, m_1}$] (Xm1) <-[label=$g_{m_1, m_2}$] (Xm2) <- (Xend),
			(Y0) <- (Yn0) <-[label'=$f_{n_0, n_1}$] (Yn1) <-[label'=$f_{n_1, n_2}$] (Yn2) <- (Yend),
			(Yinf) <-[label'=$h_∞$] (Xinf),
		};
		\begin{scope}[
			label'/.append style={inner sep=0.2ex},
		]
		\graph{
			(Yn0) <-[label=$h_0$] (Xm0) <-[label'=$h_1$] (Yn1) <-[label=$h_2$] (Xm1)
				<-[label'=$h_3$] (Yn2) <-[label=$h_4$] (Xm2) <- (hend),
		};
		\end{scope}
	\end{tikzpicture}
	
	\caption{The zig-zag sequence in the back and forth construction.}
	\label{fig:zigzag}
\end{figure}

	\begin{proof}
		We put $φ_k := h_{2k}$ and $ψ_k := h_{2k + 1}$ for $k ∈ ω$.
		We have \[
			φ_k ∘ f_{m_k, m_{k + 1}} ≈_{ε_k/2} φ_k ∘ h_{2k + 1, 2k + 3} = h_{2k, 2k + 2} ∘ φ_{k + 1} ≈_{ε_k/2} g_{n_k, n_{k + 1}} ∘ φ_{k + 1},
		\]
		so $φ_k ∘ (f_{m_*})_k ≈_{ε_k} (g_{n_*})_k ∘ φ_{k + 1}$, and similarly $ψ_k ∘ (g_{n_*})_{k + 1} ≈_{δ_k} (f_{m_*})_k ∘ ψ_{k + 1}$. 
		We apply Proposition~\ref{thm:cone_transfer} to the subsequences $f_{m_*}$ and $g_{n_*}$, and use the fact that the limit of a subsequence is the same as the limit of the whole sequence.
	\end{proof}
\end{corollary}

Let us generalize the strict definitions \ref{def:dominating_subcategory} and \ref{def:dominating_sequence} to the context of MU-categories.
The new definitions are closely related to the definitions in \cite[Section~3]{Kubis13}.

\begin{definition} \label{def:dominating}
	Let $\K$ be an MU-category.
	A subcategory $\D ⊆ \K$ is called \emph{dominating} if it is both
	\begin{itemize}
		\item \emph{cofinal}, i.e. for every $\K$-object $X$, there is a $\K$-map $f\maps X \from Y$ from a $\D$-object $Y$,
		\item \emph{absorbing}, i.e. for every $\D$-object $X$, every $ε > 0$, and every $\K$-map $f\maps X \from Y$ there is a $\K$-map $g\maps Y \from Z$ and a $\D$-map $h\maps X \from Z$ such that $f ∘ g ≈_ε h$.
	\end{itemize}
	Similarly, a sequence $\tuple{X_*, f_*}$ in a category $\K$ is called \emph{dominating} if it is both
	\begin{itemize}
		\item \emph{cofinal}, i.e. for every $\K$-object $Y$, there is $n ∈ ω$ and a $\K$-map $f\maps Y \from X_n$, and
		\item \emph{absorbing}, i.e. for every $n ∈ ω$, every $ε > 0$, and every $\K$-map $f\maps X_n \from Y$ there is $n' ≥ n$ and a $\K$-map $g\maps Y \from X_{n'}$ such that $f ∘ g ≈_ε f_{n, n'}$.
	\end{itemize}
	Since now being dominating/absorbing in $\K$ depends on the MU-structure and not just on the underlying category (which can at the same time be viewed as a discrete MU-catetory), we shall call a subcategory or a sequence \emph{strictly dominating/absorbing} if it is dominating/absorbing with respect to the discrete MU-structure on $\K$, as opposed to the original given MU-structure on $\K$.
\end{definition}

\begin{observation} \label{def:locally_dense}
	In a discrete MU-category the new definition is indeed equivalent to the old definition: if $ε$ is small enough, we have an equality.
	We also obtain refinements of some claims, e.g. a cofinal subcategory $\D ⊆ \K$ is dominating not only when it is full, it is enough if it is \emph{locally dense}, meaning that every $\D$-hom-set is dense in the corresponding $\K$-hom-set.
\end{observation}

Propositions~\ref{thm:strictly_dominating_subcategory} and \ref{thm:strictly_dominating_sequence} are still true in the approximate context.

\begin{proposition} \label{thm:dominating_subcategory}
	Let $\K ⊆ \L$ be MU-categories and let $\D ⊆ \K$ be a dominating subcategory.
	A family $\F$ of $\L$-objects is generic in $\tuple{\K, \L}$ if and only if it is generic in $\tuple{\D, \L}$.
	
	\begin{proof}
		Suppose that $\F$ is generic over $\K$.
		We inductively define Odd's winning strategy for $\F$ in $\BM(\D)$ while simultaneously playing $\BM(\K)$.
		Let $f_{2k}$ be Eve's move in a play of $\BM(\D)$.
		We put $g_{2k} := g'_{2k - 1} ∘ f_{2k}$ (and $g_0 := f_0$) to be Eve's corresponding move in a play of $\BM(\K)$.
		We pick $ε_k > 0$ such that the maps $f_{2j + 1, 2k + 1}$ and $g_{2j + 1, 2k + 1}$ are $\tuple{ε_j / 2^{k - j}, ε_k}$-continuous for $j < k$.
		Let $g_{2k + 1}$ be Odd's response in $\BM(\K)$ according to his winning strategy.
		By the absorption there is $g'_{2k + 1} ∈ \K$ and $f_{2k + 1} ∈ \D$ such that $g_{2k + 1} ∘ g'_{2k + 1} ≈_{ε_k} f_{2k + 1}$ 
			(and so later $g_{2k + 1} ∘ g_{2k + 2} = g_{2k + 1} ∘ g'_{2k + 1} ∘ f_{2k + 2} ≈_{ε_k} f_{2k + 1} ∘ f_{2k + 2}$).
		We take $f_{2k + 1}$ as Odd's response in $\BM(\D)$.
		In the end, we see that the subsequences $f_{n_*}$ and $g_{n_*}$ where $n_k := 2k + 1$, $k ∈ ω$, satisfy the assumptions of Corollary~\ref{thm:Brown}, and so have isomorphic limits.
		
		The other implication is somewhat similar.
		Suppose that $\F$ is generic over $\D$.
		Let $f_0\maps X_0 \from X_1$ be Eve's first move in a play of $\BM(\K)$.
		By the cofinality there is a $\K$-map $f'_0\maps X_1 \from Y_1$ such that $Y_1 ∈ \Ob(\D)$.
		We let $g_0 := \id_{Y_1}$ be Eve's first move in a play of $\BM(\D)$, $g_1$ be Odd's response according to his winning strategy for $\F$, and $f_1 := f'_0 ∘ g_1$ be Odd's corresponding response in $\BM(\K)$.
		We continue inductively. Let $k ∈ ω$.
		We pick $ε_k > 0$ such that the maps $f_{2j + 2, 2k + 2}$ and $g_{2j + 2, 2k + 2}$ are $\tuple{ε_j / 2^{k - j}, ε_k}$-continuous for $j < k$.
		Let $f_{2k + 2}$ be Eve's next move in $\BM(\K)$.
		By the absorption there is $f'_{2k + 2} ∈ \K$ and $g_{2k + 2} ∈ \D$ such that $f_{2k + 2} ∘ f'_{2k + 2} ≈_{ε_k} g_{2k + 2}$.
		Let $g_{2k + 3}$ be Odd's response in $\BM(\D)$ according to his winning strategy, and let $f_{2k + 3} := f'_{2k + 2} ∘ g_{2k + 3}$ be Odd's corresponding response in $\BM(\K)$.
		Again, by Corollary~\ref{thm:Brown}, the subsequences $f_{n_*}$ and $g_{n_*}$ where $n_k := 2k + 2$, $k ∈ ω$, have isomorphic limits.
	\end{proof}
\end{proposition}

\begin{proposition} \label{thm:dominating_sequence}
	A dominating sequence $\tuple{X_*, f_*}$ in an MU-category $\K$ serves as a winning strategy for Odd in $\BM(\K)$.
	If $\tuple{X_*, f_*}$ has a limit $X_∞$ in a larger category $\L ⊇ \K$, then $X_∞$ is a generic object in $\tuple{\K, \L}$.
	
	\begin{proof}
		Let $g_0\maps Y_0 \from Y_1$ be Eve's first move in $\BM(\K)$.
		From the cofinality, Odd may respond with a map $g_1\maps Y_1 \from X_{n_0}$ for some $n_0 ∈ ω$.
		Let us fix $ε_0 > 0$.
		Eve continues with a map $g_2\maps X_{n_0} \from Y_3$.
		From the absorption, Odd may respond with a map $g_3\maps Y_3 \from X_{n_1}$ for some $n_1 ≥ n_0$ such that $h_0 := g_2 ∘ g_3 ≈_{ε_0} f_{n_0, n_1}$.
		There is $ε_1 > 0$ such that the maps $f_{n_0, n_1}$ and $h_0$ are $\tuple{ε_0 / 2, ε_1}$-continuous.
		
		We continue inductively: for $k ∈ ω$ we are at $X_{n_k}$ and have $ε_k$.
		Eve plays $g_{2k + 2}\maps X_{n_k} \from Y_{2k + 3}$, and Odd responds with $g_{2k + 3}\maps Y_{2k + 3} \from X_{n_{k + 1}}$ for some $n_{k + 1} ≥ n_k$ such that $h_k := g_{2k + 2} ∘ g_{2k + 3} ≈_{ε_k} f_{n_k, n_{k + 1}}$.
		We pick $ε_{k + 1} > 0$ such that the maps $f_{n_j, n_{k + 1}}$ and $h_{j, k + 1}$ are $\tuple{ε_j / 2^{k + 1 - j}, ε_{k + 1}}$-continuous for $j < k + 1$.
		In the end, by Corollary~\ref{thm:Brown} we have that $\lim h_* = \lim f_{n_*}$, and so $\lim g_* = \lim f_*$, and the specified strategy is winning for $X_∞$ in $\tuple{\K, \L}$.
	\end{proof}
\end{proposition}

\begin{remark} \label{rm:generic_Peano}
	It follows that Theorem~\ref{thm:generic_Peano} is still true when $\PeanoS$ is viewed as an MU-subcategory of $\MCpt$ (as opposed to being a discrete MU-category), and so allowing for more dominating subcategories.
\end{remark}

Next, we show that unlike in the discrete case there is a dominating sequence in the MU-category $\I ⊆ \MCpt$.

\begin{proposition}
	A sequence in $\I$ is dominating if and only if it is crooked.
	
	\begin{proof}
		Obviously, every sequence in $\I$ is cofinal.
		An absorbing sequence $f_*$ in $\I$ is crooked since 
		for every $n ∈ ω$ and $ε > 0$ there is an $ε/3$-crooked map $g\maps X_n \from Y$,
		and there is $n' ≥ n$ and $h\maps Y \from X_{n'}$ such that $g ∘ h ≈_{ε/3} f_{n, n'}$.
		It follows from Proposition~\ref{thm:crooked_calculus} that $f_{n, n'}$ is $ε$-crooked.
		On the other hand, by Theorem~\ref{thm:crooked_factorization}, every crooked sequence in $\I$ is absorbing.
	\end{proof}
\end{proposition}

\begin{remark} \label{thm:Bing_reproved}
	The previous proposition shows that the pseudo-arc is a limit of a dominating sequence in $\I$, and also gives an alternative proof of Bing's theorem.
	A hereditarily indecomposable arc-like continuum is a limit of a crooked sequence in $\I$ by Theorem~\ref{thm:crooked_limit}.
	Now we equivalently know that the sequence is dominating.
	But by Proposition~\ref{thm:dominating_sequence}, the limit of a dominating sequence is a generic object, which is unique.
	So there is a unique (up to isomorphism) hereditarily indecomposable arc-like continuum.
\end{remark}

\section{Fraïssé theory}
\label{sec:fraisse}

So far we have shown that the pseudo-arc is (characterized as) a generic object in $\tuple{\I, σ\I}$ and that it is a $σ\I$-limit of a dominating sequence in $\I$ (Remark~\ref{thm:Bing_reproved}).
In 2006, Irwin and Solecki~\cite{IS06} introduced projective Fraïssé theory for topological structures and characterized the pseudo-arc as (using our language, see below) the unique homogeneous object in $σ\I$.
They essentially considered the category of the finite linear graphs $\II_n$, $n ∈ ω$, and simplicial surjections, and obtained its Fraïssé limit in the category of topological graphs (see Example~\ref{ex:IS}).
The Fraïssé limit is the Cantor space with a special closed equivalence relation such that the quotient topological space is the pseudo-arc, so it is the pre-space and not the pseudo-arc itself what is the Fraïssé limit.
Here we generalize abstract projective Fraïssé theory to approximate setting and prove the results by directly working in the category $\I$ of continuous surjections on the unit interval.

Note that the classical Fraïssé theory, as formulated in \cite[Chapter~7]{Hodges93}, originating in the work by Fraïssé~\cite{Fraisse54}, is concerned with embeddings of finite and countable first order structures, a Fraïssé limit is the union of a countable increasing chain of finite structures (which corresponds to a direct limit), and its extension property is called \emph{injectivity}.
For this reason the classical theory could be called \emph{injective}.
On the other hand, in the projective Fraïssé theory of Irwin and Solecki, extended structure is the domain of a quotient map as opposed to the codomain of an embedding, and the generic object is obtained as the inverse limit of an inverse sequence.
However, in abstract Fraïssé theory formulated using the language of category theory~\cite{DG93}, \cite{Kubis14}, injectivity vs.\ projectivity is a mere convention in definitions of relevant properties – the objects grow either along or against the direction of morphisms, and there is really a single Fraïssé theory.
In this paper we build our approximate abstract Fraïssé theory using the projective convention as it fits the case of continuous surjections on metrizable compacta.

We shall develop the general theory while demonstrating some of the concepts on the pseudo-arc on the way.
The full summary of the pseudo-arc as a Fraïssé limit in our setting can be found in Observation~\ref{thm:mountain_climbing} and Theorem~\ref{thm:pseudo-arc} below.

\subsection{Key notions}

\begin{definition}
	Let $\K ⊆ \L$ be MU-categories.
	An $\L$-object $U$ is 
	\begin{itemize}
		\item \emph{cofinal} in $\tuple{\K, \L}$ if for every $\K$-object $X$ there is an $\L$-map $f\maps X \from U$,
		\item \emph{projective} in $\tuple{\K, \L}$ if for every $\K$-object $X$, every $ε > 0$, every $\L$-map $f\maps X \from U$, and every $\K$-map $g\maps X \from Y$ there is an $\L$-map $h\maps Y \from U$ such that $f ≈_ε g ∘ h$,
		\item \emph{homogeneous} in $\tuple{\K, \L}$ if for every $\K$-object $X$, every $ε > 0$, and every pair of $\L$-maps $f, g\maps X \from U$ there is an $\L$-automorphism $h\maps U \from U$ such that $f ≈_ε g ∘ h$.
	\end{itemize}
	If $\K = \L$, we say just “cofinal/projective/homogeneous in $\L$”.
	The projectivity of $U$ in $\tuple{\K, \L}$ is also called the \emph{extension property}.
\end{definition}

\begin{observation} \label{thm:homogeneous_to_projective}
	A cofinal homogeneous object $U$ in $\tuple{\K, \L}$ is also projective in $\tuple{\K, \L}$.
	
	\begin{proof}
		Let $f\maps X \from U$ be an $\L$-map, let $g\maps X \from Y$ be a $\K$-map, and let $ε > 0$.
		Since $U$ is cofinal, there is an $\L$-map $g'\maps Y \from U$.
		Since $U$ is homogeneous there is an $\L$-automorphism $h\maps U \from U$ such that $f ≈_ε g ∘ (g' ∘ h)$.
	\end{proof}
\end{observation}

Even though the converse of the previous proposition is not true in general, we will see that it holds in several special situations.

\begin{observation} \label{thm:projective_cofinal}
	A projective object $U$ in $\tuple{\K, \L}$ is already cofinal if the category $\K$ is \emph{connected} (i.e. there is no family $∅ ≠ \F ⊊ \Ob(\K)$ such that $\K(X, Y) = ∅$ for every $X ∈ \F$ and $Y ∈ \Ob(\K) \setminus \F$) and if there is any $\L$-map from $U$ to a $\K$-object.
	
	\begin{proof}
		Let $\F$ be the family of all $\K$-objects $X$ such that there is an $\L$-map $f\maps X \from U$.
		Let $g\maps X \from Y$ be a $\K$-map.
		Clearly, if $Y ∈ \F$, then $X ∈ \F$.
		But by the extension property, the other implication holds as well.
	\end{proof}
\end{observation}

\begin{definition}
	A map $f\maps X \to Y$ in an MU-category $\L$ is called a \emph{near-isomorphism} if for every $ε > 0$ there is an isomorphism $g\maps X \to Y$ such that $f ≈_ε g$.
	If $X = Y$, a near-isomorphism is also called a \emph{near-automorphism}.
	Note that in a discrete MU-category, a near-isomorphism is already an isomorphism.
\end{definition}

\begin{observation} \label{thm:near_automorphism}
	Every $\L$-map $f\maps U \to U$ for a homogeneous object $U$ in an MU-category $\L$ is a near-automorphism
	since by the homogeneity, for every $ε > 0$ there is an automorphism $h\maps U \to U$ such that $f ≈_ε \id_U ∘ h$.
\end{observation}

\begin{proposition} \label{thm:near_isomorphism}
	Every $\L$-map $f\maps X \to Y$ between projective objects in a locally complete MU-category $\L$ is a near-isomorphism.
	
	\begin{proof}
		Let $ε_0 := ε/2$ and $h_0 := f$.
		Since $Y$ is projective, there is an $\L$-map $h_1\maps X \from Y$ such that $h_0 ∘ h_1 ≈_{ε_0/2} \id_Y$.
		Let $δ_0 > 0$ be such that $h_0$ is $\tuple{ε_0/2, δ_0}$-continuous.
		Since $X$ is projective, there is an $\L$-map $h_2\maps Y \from X$ such that $h_1 ∘ h_2 ≈_{δ_0/2} \id_X$.
		There is $0 < ε_1 ≤ ε_0/2$ such that $h_1$ is $\tuple{δ_0/2, ε_1}$-continuous.
		We continue this way.
		Let $f_*$ be the sequence in $\L$ consisting of the identities on $X$, let $g_*$ be sequence of the identities on $Y$, and let $\tuple{X, f_{*, ∞}}$ and $\tuple{Y, g_{*, ∞}}$ be the respective limits of $f_*$ and $g_*$ consisting of identities.
		Since we have taken $ε_{k + 1} ≤ ε_k/2$ for every $k ∈ ω$, $\tuple{ε_k}_{k ∈ ω}$ is an epsilon sequence for $f_*$, and similarly $\tuple{δ_k}_{k ∈ ω}$ is an epsilon sequence for $g_*$.
		We have chosen the maps $h_k$, $k ∈ ω$, so that we can apply Corollary~\ref{thm:back_and_forth}.
		Hence, there is an isomorphism $h_∞\maps X \to Y$ such that $f = h_0 ∘ f_{0, ∞} ≈_ε g_{0, ∞} ∘ h_∞ = h_∞$.
	\end{proof}
\end{proposition}

\begin{deftheorem}
	A \emph{Fraïssé object} in a locally complete MU-category $\L$ is an object $U$ satisfying the following equivalent conditions:
	\begin{enumerate}
		\item $U$ is cofinal and projective in $\L$,
		\item $U$ is cofinal and homogeneous in $\L$.
	\end{enumerate}
	Such object $U$ is unique up to isomorphism, and every map $U \to U$ is a near-automorphism.
	
	\begin{proof}
		If $U$ is cofinal and homogeneous, then it is projective by Observation~\ref{thm:homogeneous_to_projective}.
		If $U$ is projective, $f, g\maps X \from U$ are $\L$-maps, and $ε > 0$, then there is an $\L$-map $h\maps U \from U$ such that $f ≈_ε g ∘ h$.
		There is $0 < ε' < ε$ such that $f ≈_{ε'} g ∘ h$ and $δ > 0$ such that $g$ is $\tuple{ε - ε', δ}$-continuous.
		By Proposition~\ref{thm:near_isomorphism} there is an automorphism $h'\maps U \from U$ such that $h ≈_δ h'$, and so $f ≈_{ε'} g ∘ h ≈_{ε - ε'} g ∘ h'$, and so $U$ is homogeneous.
		
		Let $U, U'$ be Fraïssé objects in $\L$.
		By the cofinality, there is an $\L$-map $f\maps U \to U'$.
		By Proposition~\ref{thm:near_isomorphism}, $f$ is a near-isomorphism, and so there is an isomorphism $U \to U'$.
	\end{proof}
\end{deftheorem}

\begin{remark}
	By the result of Irwin and Solecki~\cite[Theorem~4.2]{IS06}, the pseudo-arc is homogeneous in the locally complete MU-category $σ\I$.
	By the classical result of Mioduszewski~\cite{Mioduszewski62} (also reproved in \cite{IS06}), the pseudo-arc is cofinal in $σ\I$.
	It follows that the pseudo-arc is a Fraïssé object in $σ\I$.
\end{remark}

\begin{observation} \label{thm:fraisse_object_is_generic}
	We have the following three tiers of objects $U$ in a locally complete MU-category $\L$:
	(i) a Fraïssé object $\implies$ (ii) a generic object $\implies$ (iii) a cofinal object.
	Also, a cofinal object $U$ such that every $\L$-map $U \to U$ is a near-isomorphism is generic.
	
	\begin{proof}
		Clearly, a generic object in $\L$ is cofinal in $\L$ since Eve may start a play of $\BM(\L)$ at any object.
		If $U$ is Fraïssé in $\L$, then it is cofinal and every self-map is a near-automorphism by Observation~\ref{thm:near_automorphism}.
		In such situation we describe a winning strategy for Odd and $U$ in $\BM(\L)$.
		Let $f_0\maps X_0 \from X_1$ be Eve's first move.
		Since $U$ is cofinal, there is an $\L$-map $f_1\maps X_1 \from U$, which shall be Odd's response.
		Eve reacts with a map $f_2\maps U \from X_3$.
		Odd will again respond with a map $f_3\maps X_3 \from U$, and so on.
		Let $g_*$ be the subsequence of $f_*$ whose all objects are copies of $U$.
		It is enough to show that $\lim g_* = U$.
		This follows from the fact that every map $g_n$ is a near-automorphism.
		Let $ε_0 > 0$.
		There is an isomorphism $h_0\maps U \from U$ such that $g_0 ≈_{ε_0} h_0$ and $ε_1 > 0$ such that both $g_0$ and $h_0$ are $\tuple{ε_0/2, ε_1}$-continuous.
		We may continue so that the obtained sequence $h_*$ of isomorphisms satisfies the assumptions of Corollary~\ref{thm:Brown}.
		It follows that the limit of $g_*$ (and so of $f_*$) is isomorphic to the limit of $h_*$, which is clearly $U$.
	\end{proof}
\end{observation}

\begin{remark}
	Note that for MU-categories $\K ⊆ \L$, a cofinal/projective/homogeneous object in $\L$ is such also in $\tuple{\K, \L}$, but this is not the case with being generic – by taking a smaller subcategory $\K$ we are restricting both players in $\BM(\K)$.
	So for example the fact that the pseudo-arc is generic over $\I$ does not directly follow from its genericity in $σ\I$.
\end{remark}

Next we describe a way how to obtain a Fraïssé object – as a \emph{Fraïssé limit} of a suitable MU-subcategory $\K ⊆ \L$.
For this we need the notion of \emph{Fraïssé sequence}.
Fraïssé sequences were explicitly introduced by Kubiś~\cite{Kubis14} as a key tool for abstract Fraïssé theory.
It was defined as a dominating sequence in an injective and transfinite setting, but often used in presence of the amalgamation property.
Here we work in a countable approximate projective setting, and we reserve the name Fraïssé sequence for the case with the amalgamation property.

\begin{definition} \label{def:Fraisse_sequence}
	Let $\K$ be an MU-category.
	We say that a sequence $f_*$ in $\K$ is \emph{projective} or that it has the \emph{extension property} if for every $\K$-object $Z$, every $ε > 0$, and every $\K$-maps $f\maps Z \from X_n$ and $g\maps Z \from Y$ there is $n' ≥ n$ and a $\K$-map $g'\maps Y \from X_{n'}$ such that $f ∘ f_{n, n'} ≈_ε g ∘ g'$.
	Clearly, every projective sequence is absorbing.
	
	We say that $\K$ has the \emph{amalgamation property} if for every $\K$-object $Z$, every $ε > 0$, and every pair of $\K$-maps $f\maps Z \from X$ and $g\maps Z \from Y$ there is a $\K$-object $W$ and there are $\K$-maps $f'\maps X \from W$ and $g'\maps Y \from W$ such that $f ∘ f' ≈_ε g ∘ g'$.
	
	On one hand, a projective sequence in $\K$ can be viewed as a generalization of a projective object in $\K$.
	On the other hand, projectivity of a sequence can be viewed as a form of the amalgamation property along the sequence.
	In fact, by the next proposition, a dominating sequence in $\K$ is projective if and only if $\K$ has the amalgamation property.
	Such a sequence is called a \emph{Fraïssé sequence}.
\end{definition}

\begin{proposition} \label{thm:projective_dominating_sequence}
	A dominating sequence $f_*$ in an MU-category $\K$ is projective if and only if $\K$ has the amalgamation property.
	
	\begin{proof}
		Suppose that $f_*$ is projective.
		Let $f\maps Z \from X$ and $g\maps Z \from Y$ be $\K$-maps, and let $ε > 0$.
		By the cofinality, there is a $\K$-map $f'\maps X \from X_n$ for some $n ∈ ω$.
		By the projectivity, there is a $n' ≥ n$ and a $\K$-map $g'\maps Y \from X_{n'}$ such that $f ∘ (f' ∘ f_{n, n'}) ≈_ε g ∘ g'$, which proves the amalgamation property.
		
		Suppose that $\K$ has the amalgamation property.
		Let $f\maps Z \from X_n$ and $g\maps Z \from Y$ be $\K$-maps.
		By the amalgamation property there are $\K$-maps $f'$ and $g'$ such that $f ∘ f' ≈_{ε/2} g ∘ g'$.
		There is $δ > 0$ such that $f$ is $\tuple{ε/2, δ}$-continuous.
		Since $f_*$ is absorbing, there is $n' ≥ n$ and a $\K$-map $g''$ such that $f' ∘ g'' ≈_δ f_{n, n'}$.
		Altogether we have $f ∘ f_{n, n'} ≈_{ε/2} f ∘ f' ∘ g'' ≈_{ε/2} g ∘ (g' ∘ g'')$, which proves the projectivity.
	\end{proof}
\end{proposition}

The notion of a Fraïssé object was introduced in \cite[p.~1762]{Kubis14} in the discrete injective setting by a condition corresponding to the constant identity sequence being dominating.
In the terminology used here, that would correspond to a “dominating object”.
The following proposition clarifies the relationship between the notions.

\begin{proposition} \label{thm:fraisse_object}
	For an object $U$ in a locally complete MU-category $\L$, the following conditions are equivalent.
	\begin{enumerate}
		\item $U$ is a Fraïssé object in $\L$.
		\item $U$ is cofinal in $\L$, $\L$ has the amalgamation property, and for every $ε > 0$ and $\L$-map $f\maps U \from X$ there is an $\L$-map $g\maps X \from U$ such that $f ∘ g ≈_ε \id_U$.
		\item $U$ is cofinal in $\L$, $\L$ has the amalgamation property, and every $\L$-map $U \to U$ is a near-automorphism.
	\end{enumerate}
	
	\begin{proof}
		Observe that $U$ is a Fraïssé object if and only if $\tuple{\id_U}_{n ∈ ω}$ is a Fraïssé sequence, and similarly $U$ satisfies (ii) if and only if $\L$ has the amalgamation property and $\tuple{\id_U}_{n ∈ ω}$ is a dominating sequence.
		Hence, the equivalence of (i) and (ii) follows from Proposition~\ref{thm:projective_dominating_sequence}.
		Given the above equivalence, we have already established that (i) resp.\ (ii) implies (iii).
		
		It remains to show that (iii) implies that $U$ is homogeneous in $\L$.
		Let $f, g\maps X \from U$ be $\L$-maps and let $ε > 0$.
		By the amalgamation property, there are $\L$-maps $f', g'\maps U \from Y$ such that $f ∘ f' ≈_{ε/3} g ∘ g'$.
		By the cofinality, there is an $\L$-map $h\maps Y \from U$.
		Since the maps $f' ∘ h$ and $g' ∘ h$ are near-automorphisms, there are automorphisms $f'', g''\maps U \from U$ such that $f ∘ f'' ≈_{ε/3} f ∘ (f' ∘ h) ≈_{ε/3} g ∘ (g' ∘ h) ≈_{ε/3} g ∘ g''$.
		Hence, $f ≈_ε g ∘ (g'' ∘ (f'')\inv)$.
	\end{proof}
\end{proposition}

Recall that our goal is to obtain a Fraïssé object.
We can do this by taking the limit of a Fraïssé sequence, but extra assumptions on $\tuple{\K, \L}$ are needed.

\begin{definition} \label{def:free_completion}
	Let $\K ⊆ \L$ be MU-categories. We say that $\tuple{\K, \L}$ is a \emph{free completion} if it satisfies the following conditions.
	\begin{itemize}
		\item[(L1)] $\L$ is locally complete and every $\K$-sequence has a limit in $\L$.
		\item[(L2)] Every $\L$-object is a limit of a $\K$-sequence.
		\item[(F1)] For every $\K$-sequence $f_*$, every its limit $\tuple{X_∞, f_{*, ∞}}$ in $\L$, every $\K$-object $Y$, every $ε > 0$, and every $\L$-map $h\maps Y \from X_∞$ there is $n ∈ ω$ and a $\K$-map $g\maps Y \from X_n$ such that $h ≈_ε g ∘ f_{n, ∞}$.
		\item[(F2)] For every $\K$-object $Y$ and every $ε > 0$ there is $δ > 0$ such that for every $\K$-sequence $f_*$, every its limit $\tuple{X_∞, f_{*, ∞}}$ in $\L$, every $n ∈ ω$, and all $\K$-maps $g, g'\maps Y \from X_n$ such that $g ∘ f_{n, ∞} ≈_δ g' ∘ f_{n, ∞}$ there is $n' ≥ n$ such that $g ∘ f_{n, n'} ≈_ε g' ∘ f_{n, n'}$.
		\item[(C)] For every $\K$-sequence $f_*$, every its limit $\tuple{X_∞, f_{*, ∞}}$, and every $ε > 0$ there is $n ∈ ω$ and $δ > 0$ such that for all $\L$-maps $h, h'\maps X_∞ \from Y$ such that $f_{n, ∞} ∘ h ≈_δ f_{n, ∞} ∘ h'$ we have $h ≈_ε h'$.
	\end{itemize}
	Conditions (L1) and (L2) are called the \emph{limit conditions}, (F1) and (F2) are called the \emph{factorization conditions}, and (C) is called the \emph{continuity condition}.
\end{definition}
Before discussing some ideas behind the conditions and some examples where the conditions are fulfilled, let us apply them in one of the main theorems of the section.

\begin{theorem}[Characterization of the Fraïssé limit] \label{thm:fraisse_limit}
	Let $\tuple{\K, \L}$ be a free completion of an MU-category.
	The following conditions are equivalent for an $\L$-object $U$:
	\begin{enumerate}
		\item $U$ is an $\L$-limit of a Fraïssé sequence in $\K$,
		\item $U$ is cofinal and projective in $\tuple{\K, \L}$,
		\item $U$ is cofinal and homogeneous in $\tuple{\K, \L}$,
		\item $U$ is a Fraïssé object in $\L$.
	\end{enumerate}
	Moreover, such object $U$ is unique up to isomorphism, and every $\K$-sequence with limit $U$ is Fraïssé.
	Such object is called the \emph{Fraïssé limit} of $\K$ in $\L$.
	
	\begin{proof}
		First, the implication (iv)$\implies$(iii) is trivial, and (iii)$\implies$(ii) follows from Observation~\ref{thm:homogeneous_to_projective}.
		Also, the uniqueness follows from the uniqueness of the Fraïssé object, as well as from the uniqueness of the generic object since the limit of a Fraïssé sequence is a generic object by Proposition~\ref{thm:dominating_sequence}.
		
		Let $\tuple{X_*, f_*}$ be a sequence in $\K$ and let $\tuple{U, f_{*, ∞}}$ be its limit in $\L$.
		If $U$ is cofinal and projective in $\tuple{\K, \L}$, then $f_*$ is a Fraïssé sequence in $\K$.
		Let $X$ be a $\K$-object.
		Since $U$ is cofinal in $\tuple{\K, \L}$, there is an $\L$-map $f\maps X \from U$.
		By (F1) there is a $\K$-map $g\maps X \from X_n$ for some $n ∈ ω$, so $f_*$ is cofinal in $\K$.
		Let $f\maps Z \from X_n$ and $g\maps Z \from Y$ be $\K$-maps and let $ε > 0$.
		Let us consider $δ > 0$ from (F2) for $Z$ and $ε$, and let $δ' > 0$ be such that $g$ is $\tuple{δ/2, δ'}$-continuous.
		Since $U$ is projective in $\tuple{\K, \L}$, there is an $\L$-map $h\maps Y \from U$ such that $f ∘ f_{n, ∞} ≈_{δ/2} g ∘ h$.
		By (F1) there is a $\K$-map $g'\maps Y \from X_{n'}$ for some $n' ≥ n$ such that $h ≈_{δ'} g' ∘ f_{n', ∞}$.
		Altogether, we have $f ∘ f_{n, ∞} ≈_{δ/2} g ∘ h ≈_{δ/2} g ∘ g' ∘ f_{n', ∞}$.
		By the choice of $δ$ there is $n'' ≥ n'$ such that $f ∘ f_{n, n''} ≈_ε g ∘ (g' ∘ f_{n', n''})$.
		Hence, $f_*$ is projective in $\K$.
		By (L2) there is always some sequence $f_*$ for $U$, and so the implication (ii)$\implies$(i) follows.

		To prove (i)$\implies$(iii), let $\tuple{X_*, f_*}$ be a Fraïssé sequence in $\K$ and suppose that $\tuple{U, f_{*, ∞}}$ is its limit in $\L$.
		Let $X$ be a $\K$-object.
		Since $f_*$ is cofinal in $\K$, there is a $\K$-map $f\maps X \from X_n$ for some $n ∈ ω$, and so $f ∘ f_{n, ∞}$ is an $\L$-map $X \from U$.
		Hence, $U$ is cofinal in $\tuple{\K, \L}$.
		Next, let $ε > 0$ and let $f, g\maps X \from U$ be $\L$-maps.
		By (F1) there are $\K$-maps $f'\maps X \from X_m$ and $g'\maps X \from X_{n_0}$ such that $f ≈_{ε/4} f' ∘ f_{m, ∞}$ and $g ≈_{ε/4} g' ∘ f_{n_0, ∞}$.
		Since $f_*$ is projective, there is a $\K$-map $h_0\maps X_{n_0} \from X_{m_0}$ for some $m_0 ≥ m$ such that $f' ∘ f_{m, m_0} ≈_{ε/4} g' ∘ h_0$.
		Let $ε_0 > 0$ be such that $g'$ is $\tuple{ε/4, 2ε_0}$-continuous.
		Since $f_*$ is absorbing, there is a $\K$-map $h_1\maps X_{m_0} \from X_{n_1}$ for some $n_1 ≥ n_0$ such that $h_0 ∘ h_1 ≈_{ε_0/2} f_{n_0, n_1}$.
		Let $δ_0 > 0$ be such that $h_0$ is $\tuple{ε_0/2, δ_0}$-continuous.
		Since $f_*$ is absorbing, there is a $\K$-map $h_2\maps X_{n_1} \from X_{m_1}$ for some $m_1 ≥ n_1$ such that $h_1 ∘ h_2 ≈_{δ_0/2} f_{m_0, m_1}$.
		We pick $ε_1 > 0$ such that $f_{n_0, n_1}$ is $\tuple{ε_0/2, ε_1}$-continuous and $h_1$ is $\tuple{δ_0/2, ε_1}$-continuous, and we continue in a similar manner so that the zig-zag sequence $h_*$ satisfies the assumptions of Corollary~\ref{thm:back_and_forth} for $f_{m_*}$ and $f_{n_*}$.
		Hence, there is an automorphism $h_∞\maps U \from U$ such that $h_0 ∘ f_{m_0, ∞} ≈_{2ε_0} f_{n_0, ∞} ∘ h_∞$.
		Altogether we have 
		\[
			f ≈_{ε/4} f' ∘ f_{m, ∞} ≈_{ε/4} g' ∘ h_0 ∘ f_{m_0, ∞} ≈_{ε/4} g' ∘ f_{n_0, ∞} ∘ h_∞ ≈_{ε/4} g ∘ h_∞,
		\] and so $U$ is homogeneous in $\tuple{\K, \L}$.
		
		It remains to conclude (iv) from the other conditions.
		Suppose that $\tuple{U, f_{*, ∞}}$ is a limit of a Fraïssé sequence $\tuple{X_*, f_*}$.
		We show that $U$ is cofinal in $\L$.
		Let $Y$ be an $\L$-object.
		By (L2) there is a sequence $\tuple{Y_*, g_*}$ in $\K$ with a limit $\tuple{Y, g_{*, ∞}}$ in $\L$.
		Let $\tuple{ε_k}_{k ∈ ω}$ be an epsilon sequence for $g_*$.
		Since $f_*$ is cofinal, there is a $\K$-map $φ_0\maps Y_0 \from X_{n_0}$ for some $n_0 ∈ ω$.
		Since $f_*$ is projective, we may inductively choose a $\K$-map $φ_{k + 1}\maps Y_{k + 1} \from X_{n_{k + 1}}$ for some $n_{k + 1} > n_k$ such that $φ_k ∘ f_{n_k, n_{k + 1}} ≈_{ε_k} g_{k, k + 1} ∘ φ_{k + 1}$ for every $k ∈ ω$.
		By Proposition~\ref{thm:cone_transfer}~(iii) applied to $φ_*, f_{n_*}, g_*$ there is an $\L$-map $φ_∞\maps Y \from U$.
		
		To show that $U$ is homogeneous in $\L$ let $X$ be an $\L$-object, let $ε > 0$, and let $f, g\maps X \from U$ be $\L$-maps.
		By (L2) there is a $\K$-sequence $f_*$ with $\L$-limit $\tuple{X, f_{*, ∞}}$.
		By (C) there is suitable $n ∈ ω$ and $δ > 0$ for $ε$.
		Since $U$ is homogeneous in $\tuple{\K, \L}$, there is an automorphism $h\maps U \from U$ such that $f_{n, ∞} ∘ f ≈_δ f_{n, ∞} ∘ g ∘ h$, and so $f ≈_ε g ∘ h$ by the choice of $n$ and $δ$.
	\end{proof}
\end{theorem}

\begin{remark} \label{thm:cofinality_trivial}
	If the category $\K$ is connected, then by Observation~\ref{thm:projective_cofinal}, condition~(ii) is equivalent to $U$ just being projective in $\tuple{\K, \L}$.
	Similarly, if $\L$ is connected (which turns out to be equivalent to $\K$ being connected), then (iv) is equivalent to $U$ being just projective in $\L$.
	However, the situation is different with homogeneity – we do not have cofinality and uniqueness.
	If $\L$ is a \emph{thin} category (i.e. a category where the hom-sets are degenerate), then every object is homogeneous in $\L$.
	In our example, if we considered the MU-category $\I ∪ \set{*}$ with a degenerate space and constant surjections added, then the point $*$ would be homogeneous in $σ\I ∪ \set{*}$ as well (besides the pseudo-arc).
	The reason the pseudo-arc is the only homogeneous object in $\tuple{\I, σ\I}$ is simply because every $σ\I$-object is cofinal in $\tuple{\I, σ\I}$ trivially.
\end{remark}

\begin{remark}
	At a first look it may seem that the conditions defining the free completion are just technical conditions for the proof of Theorem~\ref{thm:fraisse_limit} to work, but it can be shown that a free completion for $\K$ always exists and that it is unique up to MU-equivalence.
	The conditions are also meaningful on their own.
	(L1) says that all necessary limits exist, while (L2) makes sure that all $\L$-objects are “close to” $\K$.
	Altogether, they say that $\tuple{\K, \L}$ is a “completion”.
	(F1) is a factorization existence condition, while (F2) guarantees that the factorizations are “continuously unique”.
	Altogether, they say that $\K$-objects are “small” with respect to limits of sequences.
	
	If $\L$ is discrete at every $\K$-object, then the factorization conditions simplify: (F1) says that for every $\L$-map $h\maps Y \from X_∞$ to a $\K$-object there is a $\K$-map $g\maps Y \from X_n$ such that $h = g ∘ f_{n, ∞}$, and (F2) says that if $g ∘ f_{n, ∞} = g' ∘ f_{n, ∞}$ for some $\K$-maps $g, g'$, then already $g ∘ f_{n, n'} = g' ∘ f_{n, n'}$ for some $n' ≥ n$: the equality is witnessed before reaching the limit, so the factorization of $h$ is essentially unique.
	The discrete combination of (F1) and (F2) is closely related to the notion of \emph{finitely presentable object}~\cite[Definition~1.1]{AR94}, but here we are in the projective setting, and $\K$-objects are finitely presentable only with respect to $\K$-sequences (as opposed to any directed diagram in $\L$).
	The factorization conditions are also closely related to the notion of a \emph{resolution} and to Morita's conditions defining an \emph{expansion} (see \cite[6.1 and 7.1]{Mardesic00}), which are used in shape theory.
	
	Any limit $\tuple{X_∞, f_{*, ∞}}$ separates maps to the limit object, i.e. for $\L$-maps $h ≠ h'\maps X_∞ \from Y$ there is $n ∈ ω$ such that $f_{n, ∞} ∘ h ≠ f_{n, ∞} ∘ h'$.
	Condition~(C) says that this happens MU-continuously.
\end{remark}

\subsection{Obtaining a free completion and the pseudo-arc}

Next we look at how the conditions defining the free completion can be fulfilled, and then we conclude the story of the pseudo-arc as a Fraïssé limit.

Intuitively, obtaining a free completion works as follows.
We start with a “complete” ambient MU-category $\L$, choose a “nice” subcategory $\K ⊆ \L$ and form its closure $σ\K$ under limits of sequences.
Then $\tuple{\K, σ\K}$ is a free completion.
Alternatively, we already start with a free completion $\tuple{\K, \L}$.
Then $\tuple{\F, σ\F}$ is a free completion for every $\F ⊆ \K$ that is “nicely placed” in $\L$.

\begin{definition} \label{def:sigma_closure}
	Let $\K ⊆ \L$ be MU-categories.
	The \emph{local closure} $\clo{\K} ⊆ \L$ is the MU-subcategory of $\L$ with the same objects as $\K$ such that every hom-set $\clo{\K}(X, Y)$ is the closure of $\K(X, Y)$ in the $∞$-metric space $\L(X, Y)$.
	The MU-category $\K$ is said to be \emph{locally closed} in $\L$ if $\clo{\K} = \K$.
	
	We define the \emph{$σ$-closure} $σ\K ⊆ \L$ (sometimes denoted by $σ_\L \K$) as the smallest subcategory $\L' ⊆ \L$ that contains $\K$, is locally closed, contains all $\L$-limits of $\K$-sequences, and these remain limits from the point of view of $\L'$.
	The MU-category $\K$ is \emph{$σ$-closed} in $\L$ if $σ\K = \K$.
\end{definition}

\begin{remark} \label{rmk:sigma_closure}
	The $σ$-closure is well-defined, but we postpone technical details of the general definition to Appendix~\ref{sec:appendix}.
	Intuitively, the $σ$-closure consists of morphisms that can be approximated by $\K$-maps, and $σ\K$-maps are analogous to \emph{$\K$-admissible embeddings}~\cite{Masumoto20} from the setup based on continuous model theory (e.g. compare the properties listed in \cite[Remark~2.5]{JV22} with Appendix~\ref{sec:appendix}).
	
	Below we give a simpler description of the $σ$-closure under an extra hypothesis.
	For now we explicitly mention that $σ\K$-objects are exactly $\L$-limit objects of $\K$-sequences, and that if $\tuple{\K, \L}$ satisfies (L1), then limits of $\K$-sequences are exactly the same from the point of view of $σ\K$ and $\L$, i.e. for a $\K$-sequence $f_*$, an $\L$-cone $f_{*, ∞}$ is an $\L$-limit of $f_*$ if and only if it is an $σ\K$-limit of $f_*$ (see Corollary~\ref{thm:absolute}).
	
	Moreover, $σ\K$ is a \emph{replete} subcategory of $\L$, meaning that it is closed under isomorphic copies and isomorphisms, i.e. for every $σ\K$-object $X$ and $\L$-isomorphism $f\maps X \to Y$ to an $\L$-object $Y$, both $Y$ and $f$ are in $σ\K$.
	This follows directly from the fact that limits are defined up to isomorphism.
\end{remark}

\begin{remark} \label{rmk:sigma_closed}
	To some extent it does not matter in which ambient MU-category $\L$ we take the $σ$-closure of $\K$.
	Namely, let $\L$ be an MU-category such that every sequence has a limit, let $\K ⊆ \K' ⊆ \L$ be subcategories, and let $\L' = σ\K' ⊆ \L$.
	By the previous remark, limits of $\K'$-sequences are the same in $\L'$ and $\L$ (in particular if $\K' = \L'$ is $σ$-closed).
	It follows that $σ_{\L'} \K = σ_\L \K ⊆ \L'$.
	Moreover, a subcategory $\L'' ⊆ \L'$ is $σ$-closed in $\L'$ if and only if it is $σ$-closed in $\L$.
	Also note that $σ$-closed subcategories are stable under intersections, and that a full subcategory $\L' ⊆ \L$ is $σ$-closed if and only if it is closed under limits of sequences.
\end{remark}

\begin{example}
	$\MCpt$ is $σ$-closed in $\Top$ as a full subcategory closed under limits of sequences.
	Strictly speaking, $\Top$ is not an MU-category, but $\MCpt$ would be locally closed with respect to any MU-structure since it is full.
	The category of all metrizable continua and all continuous maps $\MCont$ is $σ$-closed in $\MCpt$, again as a full subcategory closed under limits of sequences.
	$\MCptS$ and $\MContS := \MCont ∩ \MCptS$ are also $σ$-closed in $\MCpt$.
	If $\tuple{X_∞, f_{*, ∞}}$ is a $\MCpt$-limit of a $\MCptS$-sequence, then $X_∞$ is non-empty, the maps $f_{n, ∞}$ are surjective, and if for a continuous map $h\maps Y \to X_∞$ from a metrizable compactum $Y$ the maps $f_{n, ∞} ∘ h$ are surjective, then $h$ is surjective as well.
	This is all well-known.
	Let us comment on the last point.
	The sets $(f_{n, ∞})\preim{U}$ for $n ∈ ω$ and $U ⊆ X_n$ open form an open base of $X_∞$, and so the map $h$ has dense image.
	We have already mentioned in Example~\ref{ex:cpt_mucat} that $\MCptS$ is locally complete.
\end{example}

It might be tempting define the $σ$-closure as follows.
Given two $\K$-sequences $\tuple{X_*, f_*}$, $\tuple{Y_*, g_*}$ with $\L$-limits $\tuple{X_∞, f_{*, ∞}}$, $\tuple{Y_∞, g_{*, ∞}}$, we put an $\L$-map $h\maps X_∞ \to Y_∞$ to $σ\K$ if and only if it can be approximated by a sequence of $\K$-maps $φ_* = \tuple{φ_n\maps X_{m_n} \to Y_n}_{n ∈ ω}$ as in Proposition~\ref{thm:cone_transfer}, which is more or less equivalent to the condition that for every $n ∈ ω$ and $ε > 0$ there is $m ∈ ω$ and a $\K$-map $φ_n\maps Y_n \from X_m$ such that $φ_n ∘ f_{m, ∞} ≈_ε g_{n, ∞} ∘ h$.

However, the defining condition may depend on the choice of the sequences $\tuple{X_*, f_*}$ and $\tuple{Y_*, g_*}$.
In order to be sufficiently closed under limits, we would need to include $h$ in $σ\K$ if the condition is fulfilled for \emph{some} pair of sequences $\tuple{X_*, f_*}$ and $\tuple{Y_*, g_*}$ (and even that may not be enough – see Example~\ref{ex:non_sigma_consistent}), while to assure (F1) we would need to include $h$ only if it satisfies the condition for \emph{all} pairs of sequences.
An extra assumption that makes the condition independent on the choice of sequences is the following one.

\begin{definition} \label{def:sigma_consistent}
	We say that an MU-category $\K ⊆ \L$ is \emph{$σ$-consistent} if for every pair of $\K$-sequences $\tuple{X_*, f_*}$, $\tuple{Y_*, g_*}$ with $\L$-limits $\tuple{X_∞, f_{*, ∞}}$, $\tuple{Y_∞, g_{*, ∞}}$ such that $X_∞ = Y_∞$ and every $n ∈ ω$ and $ε > 0$ there is $m ∈ ω$ and a $\K$-map $h\maps Y_n \from X_m$ such that $h ∘ f_{m, ∞} ≈_ε g_{n, ∞}$.
\end{definition}

\begin{observation} \label{thm:sigma_consistent_necessary}
	Note that $σ$-consistency could be also dubbed “(F1) for limits” as it exactly says that $\tuple{\K, \L}$ satisfies (F1) but only for $\L$-maps $h\maps Y \from X_∞$ that are of the form $g_{n, ∞}$ for a $\K$-sequence $\tuple{Y_*, g_*}$.
	This together with the fact that every $\L$-limit cone of a $\K$-sequence $g_*$ is also a $σ\K$-limit cone of $g_*$ means that $σ$-consistency is necessary for $\tuple{\K, σ\K}$ to satisfy (F1).
\end{observation}

\begin{remark}
	Using $σ$-consistency inductively for a given pair of sequences, we can build a back and forth $\K$-sequence $h_*$ with $h_{2k}\maps Y_{n_k} \from X_{m_k}$ and $h_{2k + 1}\maps X_{m_k} \from Y_{n_{k + 1}}$ such that $h_{2k} ∘ f_{m_k, ∞} ≈_{ε_k} g_{n_k, ∞}$ and $h_{2k + 1} ∘ g_{n_{k + 1}, ∞} ≈_{δ_k} f_{n_k, ∞}$ for every $k ∈ ω$ for arbitrarily small sequences $\tuple{ε_k}_{k ∈ ω}, \tuple{δ_k}_{k ∈ ω}$.
	In the discrete case in the context of projective Fraïssé theory, such property called \emph{consistency} was considered in \cite[after Lemma~2.5]{ChKR25}.
	
	More generally, $σ$-consistency means that every $\L$-isomorphism between limits of $\K$-sequences is witnessed from $\K$ by a back and forth sequence, giving a converse to Corollary~\ref{thm:back_and_forth} – see Remark~\ref{rem:sigma_consistency_back_forth}.
	Another characterization of $σ$-consistency is given in Proposition~\ref{thm:AE_sigma_consistent}.
\end{remark}

Recall that an MU-category $\K ⊆ \L$ is \emph{locally dense} (see Observation~\ref{def:locally_dense}) if $\K(X, Y) ⊆ \L(X, Y)$ is dense for all $\K$-objects $X, Y$.
In particular, every full subcategory is locally dense.
It is easy to see that if $\tuple{\K, \L}$ is a free completion, then $\K ⊆ \L$ is locally dense.
The proof of the following observation is clear and left to the reader.

\begin{observation} \label{thm:sigma_conditions_hereditary}
	The properties (L1), (F2), (C) of $\tuple{\K, \L}$ are hereditary in $\K$, and (F1) and $σ$-consistency are hereditary in $\K$ with respect to locally dense subcategories.
	More precisely, let $\F ⊆ \K ⊆ \L$ be MU-categories.
	\begin{enumerate}
		\item If $\tuple{\K, \L}$ satisfies (L1), then $\tuple{\F, \L}$ satisfies (L1).
		\item If $\tuple{\K, \L}$ satisfies (F1) and $\F ⊆ \K$ is locally dense, then $\tuple{\F, \L}$ satisfies (F1).
		\item If $\K ⊆ \L$ is $σ$-consistent and $\F ⊆ \K$ is locally dense, then $\F ⊆ \L$ is $σ$-consistent.
		\item If $\tuple{\K, \L}$ satisfies (F2), then $\tuple{\F, \L}$ satisfies (F2).
		\item If $\tuple{\K, \L}$ satisfies (C), then $\tuple{\F, \L}$ satisfies (C).
	\end{enumerate}
\end{observation}

Now we are ready to summarize when the conditions of being a free completion are fulfilled by the $σ$-closure.

\begin{proposition} \label{thm:sigma_conditions}
	Let $\K ⊆ \L$ be MU-categories such that $\tuple{\K, \L}$ satisfies (L1).
	\begin{enumerate}
		\item $\tuple{\K, σ\K}$ satisfies (L1) and (L2).
		\item $\tuple{\K, σ\K}$ satisfies (F1) if and only if $\K ⊆ \L$ is $σ$-consistent.
		\item $\tuple{\K, σ\K}$ satisfies (F1) with $σ\K ⊆ \L$ being full if and only if $\tuple{\K, \L}$ satisfies (F1).
		\item $\tuple{\K, σ\K}$ satisfies (F2) if and only if $\tuple{\K, \L}$ satisfies (F2).
		\item $\tuple{\K, σ\K}$ satisfies (C) if $\tuple{\K, \L}$ satisfies (C).
		\item $\tuple{\K, σ\K}$ is a free completion if $\tuple{\K, \L}$ satisfies (L1), (F2), (C), and if $\K ⊆ \L$ is $σ$-consistent.
		\item $\tuple{\K, σ\K}$ is a free completion with $σ\K ⊆ \L$ being full if $\tuple{\K, \L}$ satisfies (L1), (F1), (F2), and (C).
		\item $\tuple{\K, σ\K}$ is a free completion with $σ\K = \L$ if and only if $\tuple{\K, \L}$ is a free completion.
	\end{enumerate}
	
	\begin{proof}
		By Remark~\ref{rmk:sigma_closure}, the limits of $\K$-sequences are the same from the point of view of $\L$ and $σ\K$.
		We use this in the proof of all claims.
		Claim~(i) also uses the fact that $σ\K$ is locally closed in $\L$ and that $σ\K$-objects are exactly $\L$-limit objects of $\K$-sequences.
		Claims~(ii) and (iii) follow from Corollary~\ref{thm:F1}.
		Claim~(iv) is clear.
		The difference in (C) between $\L$ and $σ\K$ is only in the family of the maps $h, h'\maps X_∞ \from Y$ to be MU-continuously separated, so (C) for $\L$ implies (C) for $σ\K$, which is claim~(v).
		Claims (vi) and (vii) follow from the previous claims.
		One implication in (viii) is trivial.
		For the other, if $\tuple{\K, \L}$ is a free completion, then $σ\K ⊆ \L$ is full by (iii) and also wide by (L2).
	\end{proof}
\end{proposition}

\begin{corollary} \label{thm:free_completion_hereditary}
	Let $\tuple{\K, \L}$ be a free completion of MU-categories.
	For an MU-category $\F ⊆ \K$ we consider $σ\F ⊆ \L$.
	
	\begin{enumerate}
		\item $\tuple{\F, σ\F}$ is a free completion if and only if $\F ⊆ \L$ is $σ$-consistent.
		\item $\tuple{\F, σ\F}$ is a free completion with $σ\F ⊆ \L$ being full if and only if $\F ⊆ \K$ is locally dense, in particular if $\F ⊆ \K$ is full.
	\end{enumerate}
	
	\begin{proof}
		By Observation~\ref{thm:sigma_conditions_hereditary}, $\tuple{\F, \L}$ satisfies (L1), (F2), and (C).
		Hence, (i) follows from Proposition~\ref{thm:sigma_conditions}~(vi) and (ii).
		Given the above, by Proposition~\ref{thm:sigma_conditions}~(iii), $\tuple{\F, σ\F}$ is a free completion with $σ\F ⊆ \L$ being full if and only if $\tuple{\F, \L}$ satisfies (F1), which is equivalent to $\F ⊆ \K$ being locally dense – one implication follows from Observation~\ref{thm:sigma_conditions_hereditary}~(ii), the other from an easy observation that $\F$ has to be locally dense even in $\L$ by (F1).
	\end{proof}
\end{corollary}

Next, we describe some concrete situations such that $\tuple{\K, \L}$ satisfies (F2) or (C).

\begin{definition}
	By a \emph{metric epimorphism} in an MU-category $\K$ we mean a $\K$-map $f\maps X \to Y$ such that for every $Z ∈ \Ob(\K)$ and every $g, h ∈ \K(Y, Z)$ we have $d(g ∘ f, h ∘ f) = d(g, h)$.
	Clearly, a metric epimorphism in $\K$ is also a metric epimorphism in any MU-subcategory containing it,
	and every metric epimorphism is an epimorphism.
	On the other hand, if $\K$ is discrete with the $0$-$1$ metric, then metric epimorphisms are precisely epimorphisms.
	The dual notion of \emph{metric monomorphism} was considered in \cite[page~7]{Kubis13}.
\end{definition}

\begin{example}
	In $\MetU$ (recall Table~\ref{tab:categories}), metric epimorphisms are exactly epimorphisms, i.e. uniformly continuous maps with dense image.
	Similarly, in $\MCpt$, metric epimorphisms are exactly epimorphisms, i.e. continuous surjections.
	Hence, the MU-category $\MCptS$ consists of metric epimorphisms.
\end{example}

\begin{observation}
	Let $\K ⊆ \L$ be MU-categories.
	$\tuple{\K, \L}$ satisfies (F2) (even with $δ = ε$) in the following two cases.
	\begin{enumerate}
		\item $\L$ consists of metric epimorphisms.
		\item $\L = \MCpt$.
	\end{enumerate}
	
	\begin{proof}
		Let $\tuple{X_*, f_*}$ be a sequence in $\K$ with a limit $\tuple{X_∞, f_{*, ∞}}$ in $\L$, let $g, g'\maps Y \from X_n$ be $\K$-maps, and let $ε > 0$ such that $g ∘ f_{n, ∞} ≈_ε g' ∘ f_{n, ∞}$.
		In case (i) we can just cancel out $f_{n, ∞}$ and obtain $g ≈_ε g'$.
		In case (ii) let us put $A := \set{x_* ∈ ∏_{k ∈ ω} X_k: d(g(x_n), g'(x_n)) ≥ ε}$.
		Note that $A$ is a closed subset of the product.
		We also put $F_m := \set{x_* ∈ ∏_{k ∈ ω} X_k: x_i = f_{i, j}(x_j)$ for every $i ≤ j ≤ m}$ for $m ∈ ω$, and $F_∞ := ⋂_{m ≥ n} F_m$.
		Without loss of generality, $F_∞$ with the restricted projections is the chosen limit of $\tuple{X_*, f_*}$.
		For every $n ≤ n' ≤ ∞$ we have $g ∘ f_{n, n'} ≈_ε g' ∘ f_{n, n'}$ if and only if $A ∩ F_{n'} = ∅$.
		Hence, the claim follows from compactness: if $A ∩ F_∞ = ∅$, then $A ∩ F_{n'} = ∅$ for some $n' ≥ n$.
	\end{proof}
\end{observation}

\begin{definition}
	Let $\L$ be an MU-category.
	We say that an $\L$-map $f\maps X \to Y$ is \emph{$ε$-monic} for $ε > 0$ if for every pair of $\L$-maps $h, h'$ such that $f ∘ h = f ∘ h'$ we have $h ≈_ε h'$.
	Similarly, the map $f$ is \emph{$\tuple{ε, δ}$-monic} for $ε, δ > 0$ if for every pair of $\L$-maps $h, h'$ such that $f ∘ h ≈_δ f ∘ h'$ we have $h ≈_ε h'$.
	Note that condition~(C) for a pair $\tuple{\K, \L}$ says that for every $\K$-sequence $f_*$ with an $\L$-limit $\tuple{X_∞, f_{*, ∞}}$ and $ε > 0$ there is $n ∈ ω$ and $δ > 0$ such that $f_{n, ∞}$ is $\tuple{ε, δ}$-monic.
	Clearly, an $ε$-monic or $\tuple{ε, δ}$-monic map in $\L$ is such also in every MU-subcategory $\L' ⊆ \L$ containing it.
\end{definition}

\begin{remark}
	Recall that a continuous surjection $f\maps X \to Y$ between metric compacta is an \emph{$ε$-map} (see \cite[2.11]{Nadler92} or \cite{MS63}) if all fibers $f\fiber{y}$, $y ∈ Y$, have diameter $< ε$ (sometimes $≤ ε$ is used).
	These are exactly $ε$-monic $\MCptS$-maps since by the compactness both conditions are equivalent to $f(x) = f(x') \implies d(x, x') < ε$.
	Also, $f$ is $\tuple{ε, δ}$-monic if and only if $d(f(x), f(x')) < δ \implies d(x, x') < ε$.
\end{remark}

\begin{observation} \label{thm:epsilon_monic}
	For every sequence $\tuple{X_*, f_*}$ and its limit $\tuple{X_∞, f_{*, ∞}}$ in $\MCpt$ and every $ε > 0$ there is $n ∈ ω$ such that the map $f_{n, ∞}$ is $ε$-monic.
	For every $ε$-monic $\MCpt$-map $f\maps X \to Y$ there is $δ > 0$ such that $f$ is $\tuple{ε, δ}$-monic.
	It follows that $\tuple{\K, \MCpt}$ satisfies (C) for every $\K ⊆ \MCpt$.
	
	\begin{proof}
		For the first part we take for every $x ∈ X_∞$ an open set $x ∈ U_x ⊆ X_∞$ of diameter $< ε$, and $n_x ∈ ω$ and an open set $V_x ⊆ X_{n_x}$ such that $x ∈ f_{n_x, ∞}\preim{V_x} ⊆ U_x$.
		By the compactness there is a finite set $F ⊆ X_∞$ such that $⋃_{x ∈ F} f_{n_x, ∞}\preim{V_x} = X_∞$, and so $\V := \set{f_{n_x, n}\preim{V_x}: x ∈ F}$ where $n := \max\set{n_x: x ∈ F}$ is a cover of $f_{n, ∞}\im{X_∞} ⊆ X_n$ such that for every $V ∈ \V$, $f_{n, ∞}\preim{V}$ has diameter $< ε$. This proves that $f_{n, ∞}$ is $ε$-monic.
		For the second part it is enough to put $δ := \min\set{d(f(x), f(x')): d(x, x') ≥ ε} > 0$.
	\end{proof}
\end{observation}

By combining the previous observations with Proposition~\ref{thm:sigma_conditions}~(ii) we obtain the following corollary.

\begin{corollary} \label{thm:compact_conditions}
	$\tuple{\K, σ\K}$ is a free completion for every $σ$-consistent $\K ⊆ \MCpt$.
\end{corollary}

Next, we prove that full categories of connected polyhedra are $σ$-consistent.
Recall that an \emph{(abstract) simplicial complex} is a pair $\tuple{V, \S}$ where $V$ is a set and $\S$ is a hereditary family of non-empty finite subsets of $V$ covering it.
A \emph{simplicial map} $\tuple{V, \S} \to \tuple{V', \S'}$ is a map $f\maps V \to V'$ such that $f\im{S} ∈ \S'$ for every $S ∈ \S$.
The \emph{geometric realization} $\geom{\tuple{V, \S}}$ of a simplicial complex is the metric space $\set{x ∈ \II^V: ∑x = 1$ and $\supp(x) ∈ \S}$ where $\supp(x) := \set{v ∈ V: x(v) > 0}$ and the metric is inherited from the supremum metric on $\II^V$.
The topology is equivalently inherited from the product topology on $\II^V$.
The topology is compact if and only if $V$ is finite.
Each point $x ∈ \geom{\tuple{V, \S}}$ can also be uniquely written as $∑_{v ∈ V} x(v)\, \geom{v}_V$ where $\geom{v}_V ∈ \II^V$ is the characteristic function of $\set{v}$.
The \emph{geometric realization} of a simplicial map $f\maps \tuple{V, \S} \to \tuple{V', \S'}$ is the continuous map $\geom{f}\maps \geom{\tuple{V, \S}} \to \geom{\tuple{V', \S'}}$ defined by $∑_{v ∈ V} x(v)\, \geom{v}_V \mapsto ∑_{v ∈ V} x(v)\, \geom{f(v)}_{V'}$, i.e. $\bigl(\geom{f}(x)\bigr)(v') = ∑_{v ∈ f\fiber{v'}} x(v)$.
The geometric realization is a functor from simplicial complexes and simplicial maps to metric spaces and continuous maps.
On the full subcategory $\set{\II_n: n ∈ ω}$ the functor is isomorphic to the geometric realization functor defined in Section~\ref{sec:continua}, so the definitions are consistent.

Recall that a \emph{polyhedron} is a (necessarily metrizable compact) space homeomorphic to the geometric realization of a finite simplicial complex.
Let $\CPolS$ denote the category of all non-empty connected polyhedra and all continuous surjections.
Let $\P$ be a family of non-empty connected polyhedra (identified also with the corresponding full subcategory of $\CPolS$).
Recall that a metrizable continuum $X$ is called \emph{$\P$-like} if for every $ε > 0$ there is an $ε$-map from $X$ onto a $\P$-object.
This is equivalent (see \cite{MS63}) to being a limit of a $\P$-sequence, so $\P$-like continua are exactly $σ\P$-objects.
One implication follows from Observation~\ref{thm:epsilon_monic}, the other one from the following proposition.
Since we have all the needed machinery at hand, we prove the second implication in Corollary~\ref{thm:P-like}.
Also note that the notion of arc-like continuum and the category $σ\I$ fit into this scheme.

\begin{proposition}[{\cite[Lemma~4]{MS63}}] \label{thm:polyhedral_factorization}
	For every continuous surjection $f\maps X \to P$ from a metrizable continuum onto a connected polyhedron and for every $ε > 0$ there is $δ > 0$ such that for every connected polyhedron $P'$ and every $δ$-monic continuous surjection $f'\maps X \to P'$ there is a continuous surjection $g\maps P' \to P$ such that $f ≈_ε g ∘ f'$.
\end{proposition}

\begin{corollary}[{\cite[Theorem~1]{MS63}}] \label{thm:P-like}
	Let $\P ⊆ \CPolS$ be a full subcategory and let $X$ be a metrizable compactum.
	If for every $ε > 0$ there is an $ε$-map from $X$ onto a $\P$-object, then $X$ is a $σ\P$-object.
	
	\begin{proof}
		We inductively build a sequence $\tuple{δ_n}_{n ∈ ω}$ of strictly positive numbers converging to zero, a sequence of $δ_n$-maps $φ_n\maps X \to P_n$, a $\P$-sequence $\tuple{P_*, g_*}$, and an epsilon sequence $\tuple{ε_n}_{n ∈ ω}$ for $g_*$ such that every $φ_n$ is $\tuple{δ_n, 4ε_n}$-monic and $φ_n ≈_{ε_n} g_n ∘ φ_{n + 1}$ for $n ∈ ω$.
		This is possible:
			$φ_n$ depending only on $δ_n$ exists by the assumption; 
			$ε_n$ depending only on $δ_n$ and $\tuple{ε_k, g_k}_{k < n}$ exists by Observation~\ref{thm:epsilon_monic};
			$δ_{n + 1} ≤ 1/n$ depending only on $φ_n$ and $ε_n$ is obtained by Proposition~\ref{thm:polyhedral_factorization},
			and so a suitable map $g_n$ depending on $φ_n, ε_n, φ_{n + 1}$ exists as well.
		
		Let $\tuple{P_∞, g_{*, ∞}}$ be a limit of $g_*$ in $\MCptS$.
		By Proposition~\ref{thm:cone_transfer}~(iii) applied to the identity sequence on $X$, $φ_*$, and $g_*$ in $\MCptS$ there is a continuous surjection $φ_∞\maps X \to P_∞$ such that $φ_n ≈_{2ε_n} g_{n, ∞} ∘ φ_∞$ for every $n ∈ ω$.
		To show that $φ_∞$ is a homeomorphism it is enough to prove that it is one-to-one.
		Let $Y$ be any metrizable compactum and let $h, h'\maps Y \to X$ be any continuous maps such that $φ_∞ ∘ h = φ_∞ ∘ h'$.
		For every $n ∈ ω$ we have $φ_n ∘ h ≈_{2ε_n} g_{n, ∞} ∘ φ_∞ ∘ h = g_{n, ∞} ∘ φ_∞ ∘ h' ≈_{2ε_n} φ_n ∘ h'$.
		Since $φ_n$ is $\tuple{δ_n, 4ε_n}$-monic, we have $h ≈_{δ_n} h'$.
		Since $δ_n ≤ 1/n$ and $n$ was arbitrary, we have $h = h'$, and so $φ_∞$ is one-to-one.
	\end{proof}
\end{corollary}

Recall that $\MContS$ denotes the full MU-subcategory of $\MCptS$ of all metrizable continua (see also Table~\ref{tab:categories}).

\begin{theorem} \label{thm:polyhedral_categories}
	$\tuple{\CPolS, \MContS}$ is a free completion.
	Hence, for every locally dense (in particular, full) subcategory $\P ⊆ \CPolS$ we have that $\tuple{\P, σ\P}$ is a free completion with $σ\P ⊆ \MContS$ being full, and for every $\P ⊆ \CPolS$ that is $σ$-consistent in $\MContS$ we have that $\tuple{\P, σ\P}$ is a free completion.
	
	\begin{proof}
		Since $\MContS$ is $σ$-closed in $\MCpt$, we have $σ\CPolS ⊆ \MContS$.
		Let $f_*$ be a $\CPolS$-sequence with limit $\tuple{X_∞, f_{*, ∞}}$, let $h\maps Y \from X_∞$ be a continuous surjection onto a connected polyhedron, and let $ε > 0$.
		Take $δ > 0$ obtained from Proposition~\ref{thm:polyhedral_factorization}.
		By Observation~\ref{thm:epsilon_monic} there is $n ∈ ω$ such that $f_{n, ∞}$ is $δ$-monic, and so there is a continuous surjection $g\maps Y \from X_n$ such that $h ≈_ε g ∘ f_{n, ∞}$.
		Hence, $\tuple{\CPolS, \MContS}$ satisfies (F1), and $\CPolS$ is $σ$-consistent in $\MContS$ or equivalently in $\MCpt$.
		By Corollary~\ref{thm:compact_conditions}, $\tuple{\CPolS, σ\CPolS}$ is a free completion, and by Proposition~\ref{thm:sigma_conditions}~(iii), $σ\CPolS ⊆ \MContS$ is full.
		
		It remains to show $σ\CPolS = \MContS$, namely $\Ob(\MContS) ⊆ \Ob(σ\CPolS)$, i.e. every non-empty metrizable continuum $X$ is a limit of a $\CPolS$-sequence.
		By Corollary~\ref{thm:P-like} it is enough to show that $X$ admits an $ε$-map onto a polyhedron for every $ε > 0$, which is known~\cite[Example~1]{MS63}. See also \cite[Theorem~2.13]{Nadler92}.
		
		The rest follows from Corollary~\ref{thm:free_completion_hereditary}.
	\end{proof}
\end{theorem}

\begin{observation} \label{thm:mountain_climbing}
	To wrap up the story of the pseudo-arc, observe that $σ\I$ is a full subcategory of $\MContS$, and so the general definition is consistent with the particular definition of $σ\I$ from Section~\ref{sec:continua}.
	Also observe that $\tuple{\I, σ\I}$ is a free completion, and so we can use the characterization of the Fraïssé limit (Theorem~\ref{thm:fraisse_limit}).
	We already know that $\PP$ is the $σ\I$-limit of a dominating sequence in $\I$ (Remark~\ref{thm:Bing_reproved}).
	To show that $\PP$ is the Fraïssé limit of $\I$ in $σ\I$ it remains to observe that since $\I$ consists of a single object, every dominating sequence in $\I$ is projective:
	if $f_*$ is a dominating sequence in $\I$, then for all $\I$-maps $f, g$ and every $ε > 0$ there is $n ∈ ω$ and there are $\I$-maps $f', g'$ such that $f ∘ f' ≈_{ε/2} f_{0, n} ≈_{ε/2} g ∘ g'$, and so we have the amalgamation property, which is enough by Proposition~\ref{thm:projective_dominating_sequence}.

	In fact, it is well-known that suitable dense subcategories $\J ⊆ \I$ have even the strict amalgamation property (i.e. the amalgamation property with respect to the discrete MU-structure), from which the amalgamation property of $\I$ clearly follows.
	Results of this type are known as “mountain climbing theorems” or “mountain climbers' theorems”.
	Homma~\cite{Homma52} proved the result for $\J$ consisting of all nowhere constant maps,
	Sikorski and Zarankiewicz~\cite{SZ55} for $\J$ consisting of all piecewise monotone maps.
	(The $\J$-maps $f$ considered were limited by the condition $f(0) = 0$ and $f(1) = 1$, but for every $f ∈ \J$ one can take a piecewise linear map $f'$ such that $f ∘ f' ∈ \J$ and fixes the end-points.)
	On the other hand, the whole $\I$ does not have the strict amalgamation property (see \cite{Homma52}).
	It follows that there is no strictly dominating sequence in $\I$.
		
	The amalgamation property of $\I$ also follows directly from Theorem~\ref{thm:crooked_factorization}.
	Let $f, g$ be $\I$-maps and let $ε > 0$.
	We take a suitable $δ > 0$ for $g$ and $ε$ obtained by Theorem~\ref{thm:crooked_factorization}.
	There is $δ' > 0$ such that $f$ is $\tuple{δ, δ'}$-continuous.
	Let $f'$ be any $δ'$-crooked $\I$-map.
	Then $f ∘ f'$ is $δ$-crooked and by the choice of $δ$ there is an $\I$-map $g'$ such that $f ∘ f' ≈_ε g ∘ g'$.
\end{observation}

Combining everything together, we obtain the following result.

\begin{theorem} \label{thm:pseudo-arc}
	The pseudo-arc $\PP$ is characterized (up to homeomorphism) by any of the following conditions.
	\begin{enumerate}
		\item $\PP$ is a hereditarily indecomposable arc-like continuum.
		\item $\PP$ is the Fraïssé limit of $\I$ in $σ\I$ (which already consists of several equivalent conditions).
		\item $\PP$ is a generic object in $σ\I$ (i.e. when the game is played in $σ\I$).
		\item $\PP$ is a generic object in $\tuple{\I, σ\I}$ (i.e. when the game is played in $\I$).
		\item $\PP$ is a generic object over any dominating subcategory of $\PeanoS$.
	\end{enumerate}
\end{theorem}

\subsection{Fraïssé categories and ages}

We have turned the problem of obtaining a Fraïssé object in a locally complete MU-category $\L$ into obtaining a Fraïssé limit of a free completion $\tuple{\K, \L}$, which reduces to obtaining a Fraïssé sequence in $\K$.
Now we give a characterization of existence of a Fraïssé sequence.
This is more or less a translation of the known results~\cite[Corollary~3.8]{Kubis14}, \cite[Theorem~3.3]{Kubis13} to the context of MU-categories.

\begin{definition}
	A non-empty MU-category $\K$ is \emph{Fraïssé} if
	\begin{itemize}
		\item $\K$ is \emph{directed}, i.e. for every pair of $\K$-objects $X$ and $Y$ there are $\K$-maps $f\maps X \from Z$ and $g\maps Y \from Z$ for a $\K$-object $Z$,
		\item $\K$ has the amalgamation property,
		\item $\K$ has a countable dominating subcategory.
	\end{itemize}
\end{definition}

Of course, the theorem is that an MU-category has a Fraïssé sequence if and only if it is Fraïssé. We start with some preparations.

We have defined dominating subcategories and sequences (Definition~\ref{def:dominating}).
Let us for convenience define a more general notion of a dominating functor so that a subcategory is dominating if and only if the corresponding inclusion functor is dominating, and a sequence is dominating if and only if it is dominating when viewed as a functor from $\tuple{ω, ≤}$.

\begin{definition}
	A functor $F\maps \K \to \L$ between MU-categories (not necessarily MU-continuous) is called \emph{dominating} if it is both
	\begin{itemize}
		\item \emph{cofinal}, i.e. for every $\L$-object $X$, there is a $\K$-object $Y$ and an $\L$-map $f\maps X \from F(Y)$,
		\item \emph{absorbing}, i.e. for every $\K$-object $X$, every $ε > 0$, and every $\L$-map $f\maps F(X) \from Y$ there is an $\L$-map $g\maps Y \from F(Z)$ and a $\K$-map $h\maps X \from Z$ such that $f ∘ g ≈_ε F(h)$.
	\end{itemize}
\end{definition}

\begin{lemma}
	Let $F\maps \K \to \L$ and $G\maps \L \to \M$ be functors between MU-categories.
	\begin{enumerate}
		\item If $F$ and $G$ are cofinal, then $G ∘ F$ is cofinal.
		\item If $F$ and $G$ are absorbing and $G$ is MU-continuous, then $G ∘ F$ is absorbing.
	\end{enumerate}
	
	\begin{proof}
		(i) Let $X$ be an $\M$-object.
		Since $G$ is cofinal, there is an $\L$-object $Y$ and an $\M$-map $f\maps X \from G(Y)$.
		Since $F$ is cofinal, there is a $\K$-object $Z$ and an $\L$-map $g\maps Y \from F(Z)$.
		Hence, $f ∘ G(g)\maps X \from G(F(Z))$.
		
		(ii) Let $X$ be a $\K$-object, let $ε > 0$, and let $f\maps G(F(X)) \from Y$ be an $\M$-map.
		Since $G$ is absorbing, there is an $\L$-map $f'\maps F(X) \from Y'$ and an $\M$-map $g\maps Y \from G(Y')$ such that $G(f') ≈_{ε/2} f ∘ g$.
		Let $δ > 0$ be a witness of MU-continuity of $G$ at $F(X)$ for $ε/2$.
		Since $F$ is absorbing, there is a $\K$-map $f''\maps X \from Y''$ and an $\L$-map $g'\maps Y' \from F(Y'')$ such that $F(f'') ≈_δ f' ∘ g'$.
		Altogether, we have $G(F(f'')) ≈_{ε/2} G(f') ∘ G(g') ≈_{ε/2} f ∘ (g ∘ G(g'))$.
	\end{proof}
\end{lemma}

\begin{corollary}
	Dominating MU-functors are stable under composition.
\end{corollary}

\begin{corollary} \label{thm:dominating_sequence_transitivity}
	If $\tuple{X_*, f_*}$ is a dominating sequence in $\D$, and $\D$ is a dominating subcategory of $\K$, then $\tuple{X_*, f_*}$ is dominating in $\K$.
\end{corollary}

\begin{lemma} \label{thm:fraisse_dominating}
	Every countable dominating subcategory $\D_0$ of a Fraïssé MU-category $\K$ can be extended to a countable dominating subcategory $\D$ with $\Ob(\D) = \Ob(\D_0)$ that is Fraïssé.
	
	\begin{proof}
		For every pair of $\D_0$-objects $X$ and $Y$ we pick some $\K$-maps $v_{X, Y, 0}\maps X \from V_{X, Y}$ and $v_{X, Y, 1}\maps Y \from V_{X, Y}$ witnessing that $\K$ is directed.
		We can manage to have $V_{X, Y} ∈ \Ob(\D_0)$ since $\D_0$ is cofinal.
		Similarly, for every pair of $\D_0$-maps $f\maps Z \from X$ and $g\maps Z \from Y$ and for $n ∈ ω$ we pick some $\K$-maps $w_{f, g, n, 0}\maps X \from W_{f, g, n}$ and $w_{f, g, n, 1}\maps Y \from W_{f, g, n}$ such that $f ∘ w_{f, g, n, 0} ≈_{1/n} g ∘ w_{f, g, n, 1}$, witnessing that $\K$ has the amalgamation property.
		Again, we can have $W_{f, g, n} ∈ \Ob(\D_0)$.
		We let $\D_1 ⊆ \K$ be the subcategory generated by $\D_0 ∪ \set{v_{X, Y, 0}, v_{X, Y, 1}: X, Y ∈ \Ob(\D_0)} ∪ \set{w_{f, g, n, 0}, w_{f, g, n, 1}: f, g ∈ \D_0$ with $\cod(f) = \cod(g)$, $n ∈ ω}$.
		Clearly, $\D_1$ is still countable and dominating in $\K$ since $\Ob(\D_1) = \Ob(\D_0)$.
		Moreover, it is directed, and it has the amalgamation property with respect to spans from $\D_0$.
		However, we have potentially added new spans, so we need to repeat the second part of the procedure to obtain a chain $\D_0 ⊆ \D_1 ⊆ \D_2 ⊆ \cdots$.
		Finally, we put $\D := ⋃_{n ∈ ω} \D_n$, which is still countable and dominating in $\K$, and is directed and has the amalgamation property.
		Altogether, $\D$ is Fraïssé.
	\end{proof}
\end{lemma}

\begin{lemma} \label{thm:countable_fraisse}
	Let $\K$ be a countable MU-category.
	We may view the set $\K$ of all morphisms as a countable discrete space and consider the Polish space $\K^ω$.
	The subset $\S ⊆ \K^ω$ consisting of all sequences in $\K$ (i.e. elements of $\K^ω$ consisting of composable maps) is closed and so a Polish space on its own.
	If $\K$ is Fraïssé, then the set of all Fraïssé sequences is dense $G_δ$ in $\S$.
	
	\begin{proof}
		$\S$ is indeed closed since for every $f_* ∈ \K^ω \setminus \S$ there is $n ∈ ω$ such that $\dom(f_n) ≠ \cod(f_{n + 1})$, and so $g_* ∉ \S$ for every $g_* ∈ \K^ω$ with $g_n = f_n$ and $g_{n + 1} = f_{n + 1}$.
		
		For every $\K$-object $X$, let $\D_X ⊆ \S$ be the set of all $\K$-sequences $\tuple{X_*, f_*}$ such that there is $n ∈ ω$ and a $\K$-map $f\maps X \from X_n$.
		The set $\D_X$ is open since the membership in $\D_X$ is again witnessed by a single coordinate.
		The set $\D_X$ is dense since for every finite composable sequence $\tuple{f_k\maps X_k \from X_{k + 1}}_{k < n}$ there are $\K$-maps $f_n\maps X_n \from X_{n + 1}$ and $f\maps X \from X_{n + 1}$ by the fact that $\K$ is directed.
		
		For every $n ∈ ω$, every $\K$-map $f\maps X \from Y$ and $ε > 0$ let $\E_{n, f, ε} ⊆ \S$ be the set of all $\K$-sequences $\tuple{X_*, f_*}$ such that $X_n ≠ X$ or there is $n' ≥ n$ and a $\K$-map $g\maps Y \from X_{n'}$ such that $f ∘ g ≈_ε f_{n, n'}$.
		The set $\E_{n, f, ε}$ is open since if $f_* ∈ \E_{n, f, ε}$, this is witnessed by a finite initial segments $\tuple{f_k}_{k < n'}$ for some $n' ≥ n$, and so every $\K$-sequence $g_*$ with $\tuple{g_k}_{k < n'} = \tuple{f_k}_{k < n'}$ is in $\E_{n, f, ε}$.
		The set $\E_{n, f, ε}$ is dense since for every finite composable sequence $\tuple{f_k\maps X_k \from X_{k + 1}}_{k < m}$ with $m ≥ n$ and with $X_n = X$ there are $\K$-maps $f_m\maps X_m \from X_{m + 1}$ and $g\maps Y \from X_{m + 1}$ such that $f_{n, m} ∘ f_m ≈_ε f ∘ g$ by the amalgamation property.
		
		Finally, the set of all dominating sequences is clearly 
		\[ \textstyle
			⋂\set{\D_X: X ∈ \Ob(\K)} ∩ ⋂\set{\E_{n, f, 1/k}: n ∈ ω,\, f ∈ \K,\, k ∈ ω},
		\]
		which is dense $G_δ$ by Baire category theorem and by the previous claims.
		Fraïssé sequences are exactly dominating sequences by Proposition~\ref{thm:projective_dominating_sequence}.
	\end{proof}
\end{lemma}

\begin{theorem} \label{thm:fraisse_category}
	An MU-category $\K$ has a Fraïssé sequence if and only if it is a Fraïssé MU-category.
	
	\begin{proof}
		Suppose $\tuple{X_*, f_*}$ is a Fraïssé sequence in $\K$.
		Clearly $\K ≠ ∅$.
		$\K$ is directed since by the cofinality, for all $\K$-objects $X, Y$ there are $\K$-maps $f\maps X \from X_m$ and $g\maps Y \from X_n$ for some $m, n ∈ ω$, so $X_{\max(m, n)}$ works.
		$\K$ has the amalgamation property by Proposition~\ref{thm:projective_dominating_sequence}.
		The countable subcategory generated by $\set{f_n: n ∈ ω}$ (which may contain more compositions than $f_{n, m}$ since we may have $X_n = X_m$ for some $n, m ∈ ω$) is dominating since the sequence is dominating.
		We conclude that $\K$ is Fraïssé.
		
		Suppose that $\K$ is Fraïssé.
		By Lemma~\ref{thm:fraisse_dominating} there is a countable dominating subcategory $\D ⊆ \K$ that is Fraïssé.
		By Lemma~\ref{thm:countable_fraisse} and the fact that $\K$ and so $\D$ is non-empty, there is a Fraïssé sequence $f_*$ in $\D$.
		By Corollary~\ref{thm:dominating_sequence_transitivity}, $f_*$ is dominating in $\K$.
		By Proposition~\ref{thm:projective_dominating_sequence}, $f_*$ is Fraïssé in $\K$.
	\end{proof}
\end{theorem}

Let us formulate two observations helpful for showing that a given MU-category is Fraïssé.

\begin{observation} \label{thm:locally_separable}
	We may call an MU-category $\K$ \emph{locally separable} if every hom-set $\K(X, Y)$ is separable.
	It is well-known that $\MCpt$ (and so every $\K ⊆ \MCpt$) is locally separable (see e.g. \cite[(4.19)]{Kechris95}).
	
	If a locally separable MU-category $\K$ has a countable cofinal subcategory $\C$ (equivalently, a countable \emph{cofinal family of objects}), then $\K$ has a countable dominating subcategory $\D$: it is enough to let $\D ⊆ \K$ be the subcategory generated by $⋃\set{\D_{X, Y}: X, Y ∈ \Ob(\C)}$ where every $\D_{X, Y}$ is a countable dense subset of $\K(X, Y)$.
	It follows that every locally separable MU-category with a countable directed cofinal subcategory (in particular with a cofinal object, e.g. a full subcategory $∅ ≠ \K ⊆ \PeanoS$) is Fraïssé if and only if it has the amalgamation property.
\end{observation}

\begin{observation} \label{thm:wide_dominating}
	Let $\K$ be an MU-category and $\D ⊆ \K$ a wide dominating subcategory.
	If $\D$ has the amalgamation property, then $\K$ has the amalgamation property as well.
	
	\begin{proof}
		For every pair of $\K$-maps $f_0, f_1$ with a common codomain and $ε > 0$ there are $\K$-maps $f'_i, g_i$, $i < 2$, such that $f_i ∘ f'_i ≈_{ε/3} g_i ∈ \D$, and there are $\D$-maps $g'_i$ such that $g_0 ∘ g'_0 ≈_{ε/3} g_1 ∘ g'_1$.
		Altogether, we have $f_0 ∘ (f'_0 ∘ g'_0) ≈_ε f_1 ∘ (f'_1 ∘ g'_1)$.
	\end{proof}
\end{observation}

We have connected the notions of a Fraïssé object, a Fraïssé sequence, and a Fraïssé category.
These concepts may sometimes overlap – a Fraïssé category $\K$ of small objects may already have a Fraïssé object, corresponding to a constant identity Fraïssé sequence (see Proposition~\ref{thm:fraisse_object}).
But even if it is not the case, $\K$ has the Fraïssé limit in a free completion $\L$, which is a Fraïssé object, and so $\L$ itself is a Fraïssé category.
To summarize, we formulate the following observation.

\begin{observation}
	Let $\tuple{\K, \L}$ be a free completion of MU-categories.
	\begin{enumerate}
		\item If $\K$ is locally complete with a Fraïssé object $U$, then $\K$ is a Fraïssé category, and $U$ is also a Fraïssé object in $\L$.
		\item If $\K$ is a Fraïssé category, then $\L$ is a locally complete Fraïssé category with a Fraïssé object.
	\end{enumerate}
\end{observation}

In the classical Fraïssé theory of first-order structures, ages of countable structures are often discussed~\cite[Chapter~7]{Hodges93}, while the ambient pair $\tuple{\K, \L}$ forming a free completion is implicit.
The classical formulation is that a countable structure is homogeneous if and only if it is the Fraïssé limit of its age, and that this yields a one-to-one correspondence between countable homogeneous structures and (hereditary) Fraïssé classes \cite[Theorems~7.1.2 and 7.1.7]{Hodges93}.
We shall formulate the corresponding characterization of Fraïssé limits in our setting, so the connection to the classical Fraïssé theory is more direct.
Later, we shall view the classical theory as a special case of our setting.

\begin{definition}
	Let $\K$ be an MU-category.
	A subcategory $\F ⊆ \K$ is called \emph{hereditary} if for every $f ∈ \K$ we have $f ∈ \F$ if $\dom(f) ∈ \Ob(\F)$,
	i.e. $\F$ is a “downwards closed” full subcategory, where “downwards” means from the domain to the codomain, as is natural in the projective setting.
	Similarly, $\F$ is called a \emph{component} if it is both “downwards” and “upwards” closed full subcategory, i.e. $f ∈ \F$ whenever $\dom(f) ∈ \Ob(\F)$ or $\cod(f) ∈ \Ob(\F)$ for $f ∈ \K$ (this is related to the notion of connected category – see Observation~\ref{thm:projective_cofinal}).
	
	Let $\K ⊆ \L$ be MU-categories.
	For every $\L$-object $U$ let $\Age(U)$ (or more precisely $\Age_{\K, \L}(U)$) denote the full subcategory of $\K$ corresponding to the class of objects $\set{X ∈ \Ob(\K): \L(U, X) ≠ ∅}$.
	Clearly, this is a hereditary subcategory.
	In this context we will identify classes of $\K$-objects with the corresponding full subcategories.
\end{definition}

\begin{observation}
	Let $\tuple{\K, \L}$ be a free completion of MU-categories.
	If we start with an $\L$-object $U$, we have that $\tuple{\Age(U), \sAge(U)}$ is a free completion and that $\sAge(U) ⊆ \L$ is full by Corollary~\ref{thm:free_completion_hereditary}.
	By (L2) of $\tuple{\K, \L}$, $U$ is a $\sAge(U)$-object.
	Altogether, the characterization of the Fraïssé limit (Theorem~\ref{thm:fraisse_limit}) applies.
	
	On the other hand, if we start with a full subcategory $\F ⊆ \K$ that is Fraïssé, we have that $\tuple{\F, σ\F}$ is a free completion with a Fraïssé limit $U$.
	However, $\F ⊆ \Age(U)$ and $σ\F ⊆ \sAge(U)$ may be strict inclusions, and $U$ may not be the Fraïssé limit of its age.
	The inclusions are equalities if and only if $\F$ is hereditary, and $U$ is the Fraïssé limit of its age if and only if the age has the amalgamation property.
	But in any case, $U$ is generic in $\tuple{\Age(U), \sAge(U)}$ since $\F ⊆ \Age(U)$ is dominating (see Proposition~\ref{thm:dominating_subcategory}), and so $U$ is the only candidate for the Fraïssé limit.
\end{observation}

\begin{theorem} \label{thm:age_limit}
	Let $\tuple{\K, \L}$ be a free completion of MU-categories and let $U$ be an $\L$-object.
	\begin{enumerate}
		\item $U$ is homogeneous in $\tuple{\K, \L}$ if and only if $U$ is the Fraïssé limit of $\Age(U)$ in $\sAge(U)$.\kern-0.75em
		\item $U$ is projective in $\tuple{\K, \L}$ if and only if $U$ is the Fraïssé limit of $\Age(U)$ in $\sAge(U)$ and $\Age(U)$ is a component of $\K$.
	\end{enumerate}
	
	\begin{proof}
		We use the first part of the previous observation.
		By the definition of $\Age(U)$ and by the fact that $\sAge(U) ⊆ \L$ is full, we have that $U$ is homogeneous in $\tuple{\K, \L}$ if and only if it is homogeneous in $\tuple{\Age(U), \sAge(U)}$.
		Clearly, $U$ is always cofinal in $\tuple{\Age(U), \sAge(U)}$, hence claim~(i) follows.
		Similarly, $U$ is projective in $\tuple{\K, \L}$ if and only if it is projective in $\tuple{\Age(U), \sAge(U)}$ and $\Age(U)$ is a component of $\K$.
		Hence, claim~(ii) follows.
	\end{proof}
\end{theorem}

\begin{corollary}
	A projective object in a free completion $\tuple{\K, \L}$ of MU-categories is homogeneous.
\end{corollary}

\begin{observation}
	For a free completion $\tuple{\K, \L}$ we may preorder homogeneous objects in $\tuple{\K, \L}$ by putting $U ≤ V$ if $\Age(U) ⊆ \Age(V)$.
	This is a well-defined preorder, and two homogeneous objects are $≤$-equivalent if and only if they have the same ages, and so if and only if they are $\L$-isomorphic, as Fraïssé limits.
\end{observation}

\subsection{Applications}

Let us demonstrate the theory built so far on several examples, and let us clearly summarize how the classical Fraïssé theory and the projective Fraïssé theory are instances of the abstract theory.

\begin{example}[Cantor space] \label{ex:Cantor}
	Let $\FinS$ be the category of non-empty finite sets and surjective maps.
	These are equivalently non-empty finite discrete spaces with continuous surjections, and so $\FinS$ may be viewed as an MU-subcategory of $\MCptS$ (the category of non-empty metrizable compacta and continuous surjections).
	We will observe that $σ\FinS$ is the category of all non-empty zero-dimensional metrizable compacta and all continuous surjections.
	Clearly every limit $\tuple{X_∞, f_{*, ∞}}$ of a $\FinS$-sequence $\tuple{X_*, f_*}$ is a non-empty zero-dimensional metrizable compactum.
	$\tuple{\FinS, \MCptS}$ satisfies (F1): continuous surjections $g\maps X_∞ \to Y$ onto finite spaces correspond to finite clopen partitions of $X_∞$.
	As in Observation~\ref{thm:epsilon_monic}, there is $n ∈ ω$ such that the partition of $X_∞$ induced by $f_{n, ∞}$ refines the one induced by $g$, and so there is a surjection $h\maps X_n \to Y$ such that $h ∘ f_{n, ∞} = g$.
	It follows that $\tuple{\FinS, σ\FinS}$ is a free completion and that $σ\FinS ⊆ \MCptS$ is full (Proposition~\ref{thm:sigma_conditions}~(iii), Corollary~\ref{thm:compact_conditions}).
	Since clearly for every zero-dimensional metrizable compactum $X$ and $ε > 0$ there is an $ε$-map onto a finite space, $X$ is a $σ\FinS$-object as in Corollary~\ref{thm:P-like}.
	
	It is easy to see that $\FinS$ is a Fraïssé category, and so there is a Fraïssé limit $U$ of $\tuple{\FinS, σ\FinS}$.
	The limit $U$ is characterized by projectivity in $\tuple{\FinS, σ\FinS}$ since cofinality follows (Remark~\ref{thm:cofinality_trivial}), i.e. by the condition that for every continuous surjection $f\maps U \to F$ onto a finite space and every $\FinS$-map $g\maps F' \to F$ there is a continuous surjection $h\maps U \to F'$ such that $f = g ∘ h$.
	Equivalently, for every (necessarily finite) clopen partition of $U$ and every finite refinement pattern there is a corresponding refining clopen partition of $U$.
	Since every $\FinS$-map $g\maps F' \to F$ is a composition of “simple” maps, namely finite surjections splitting just one point into two, $U$ is projective in $\tuple{\FinS, σ\FinS}$ if and only if it has no isolated points.
	The well-known theorem that there is a unique non-empty zero-dimensional metrizable compact space without isolated points as well as the fact that the Fraïssé limit is the Cantor space $2^ω$ follow.
	As a byproduct, we obtain the known facts that $2^ω$ maps onto every non-empty zero-dimensional metrizable compactum, that every continuous surjection $2^ω \to 2^ω$ is a near-homeomorphism, as well as that for all continuous surjections $f, g\maps 2^ω \to Z$ onto a zero-dimensional compactum and every $ε > 0$ there is a homeomorphism $h\maps 2^ω \to 2^ω$ such that $f ≈_ε g ∘ h$.
	It is also easy to see that a $\FinS$-sequence $\tuple{X_*, f_*}$ is Fraïssé if and only if it “splits every point”, i.e. for every $n$ and $x ∈ X_n$ there is $n' ≥ n$ such that $f_{n, n'}\fiber{x}$ is non-degenerate.
	The winning strategy for Odd in $\BM(\FinS)$ for the Cantor space obviously consists of splitting every point, e.g. by playing the projection $X_n \from X_n × 2$.
\end{example}

The example above was discrete in the sense that we had even strict commutativity of diagrams into small objects.
Let us summarize how the theory simplifies in the discrete setting in general.

\begin{remark}[Discrete Fraïssé theory] \label{rm:discrete_fraisse}
	Let $\K ⊆ \L$ be categories viewed also as discrete MU-categories (Definition~\ref{def:discrete}).
	All the properties involving commutativity of diagrams up to arbitrarily small $ε > 0$ simplify to strict commutativity.
	For example, an $\L$-object $U$ is homogeneous in $\tuple{\K, \L}$ if and only if for every pair of $\L$-maps $f, g\maps U \to X$ to a $\K$-object there is an automorphism $h\maps U \to U$ such that $f = g ∘ h$.
	Similarly, $\K$ has the amalgamation property if and only if for every pair of $\K$-maps $f, g$ with common codomain there are $\K$-maps $f', g'$ such that $f ∘ f' = g ∘ g'$.
	
	The pair $\tuple{\K, \L}$ rarely satisfies the discrete version of the continuity condition (C).
	Hence, we say that $\tuple{\K, \L}$ is a \emph{discrete free completion} if it satisfies (L1), (L2), (F1), (F2), which simplify to the following:
	\begin{itemize}
		\item[(L1)] Every $\K$-sequence has a limit in $\L$.
		\item[(L2)] Every $\L$-object is a limit of a $\K$-sequence.
		\item[(F1)] For every $\K$-sequence $f_*$, every its limit $\tuple{X_∞, f_{*, ∞}}$ in $\L$, and every $\L$-map $h\maps Y \from X_∞$ to a $\K$-object there is $n ∈ ω$ and a $\K$-map $g\maps Y \from X_n$ such that $h = g ∘ f_{n, ∞}$.
		\item[(F2)] For every $\K$-sequence $f_*$, every its limit $\tuple{X_∞, f_{*, ∞}}$ in $\L$, every $n ∈ ω$, and every pair of $\K$-maps $g, g'\maps Y \from X_n$ to a $\K$-object such that $g ∘ f_{n, ∞} = g' ∘ f_{n, ∞}$ there is $n' ≥ n$ such that $g ∘ f_{n, n'} = g' ∘ f_{n, n'}$.
	\end{itemize}
	Since Theorem~\ref{thm:fraisse_limit} uses (C) only to prove large homogeneity, we still have the following characterization of Fraïssé limits in the discrete context.
	For a discrete free completion $\tuple{\K, \L}$ and an $\L$-object $U$, the following conditions are equivalent:
	\begin{enumerate}
		\item $U$ is an $\L$-limit of a Fraïssé sequence in $\K$,
		\item $U$ is cofinal and projective in $\tuple{\K, \L}$,
		\item $U$ is cofinal and homogeneous in $\tuple{\K, \L}$.
	\end{enumerate}
	Moreover, such object $U$ is unique up to isomorphism, is cofinal in $\L$, and every $\K$-sequence with limit $U$ is Fraïssé.
	
	It turns out that we may alter the MU-structure on $\L$ so that the discrete free completion $\tuple{\K, \L}$ becomes a free completion of MU-categories.
	For every $\L$-object $X$ we fix any $\K$-sequence $\tuple{X_*, f_*}$ with limit $\tuple{X, f_{*, ∞}}$ and for every pair of $\L$-maps $g, h\maps X \from Y$ we put 
	\[
		d(g, h) := \inf\set{1/(n + 1): f_{n, ∞} ∘ g = f_{n, ∞} ∘ h} ∈ [0, ∞],
	\]
	so $g ≈_{1/n} h$ if and only if $f_{n, ∞} ∘ g = f_{n, ∞} ∘ h$.
	This turns $\L$ into a locally complete MU-category.
	Up to MU-isomorphism, the resulting MU-category does not depend on the choice of $\K$-sequences and limit maps.
	$\tuple{\K, \L}$ with this induced MU-structure is a free completion.
	Every $\K$-object is still discrete in $\L$ with the induced MU-structure, so the commutativity of diagrams to $\K$-objects simplify as before, and hence the Fraïssé limit is the same as in the discrete setting, but additionally is homogeneous and projective in $\L$ (with respect to the induced MU-structure).
	This explains how our MU-setup can be viewed as a “conservative extension” of the discrete setup.
	
	\begin{proof}
		We show that $d$ is a well-defined complete (ultra)metric on every $\L(Y, X)$.
		We have $d(g, h) = 0 \iff g = h$ by the uniqueness of a limit factorizing map.
		The map $d$ satisfies the ultrametric triangle inequality since if $g ≈_{1/n} h$ and $h ≈_{1/m} k$ for $n ≤ m$, then $f_{n, ∞} ∘ g = f_{n, ∞} ∘ h = f_{n, ∞} ∘ k$, and so $g ≈_{1/n} k$.
		The metric $d$ is complete since for every Cauchy sequence $\tuple{g_k}_{g ∈ ω}$ in $\L(Y, X)$ and $n ∈ ω$ the sequence $\tuple{f_{n, ∞} ∘ g_k}_{k ∈ ω}$ is eventually constant with value $γ_n$, and $\tuple{Y, γ_*}$ is a cone for $f_*$.
		The limit factorizing map $γ_∞$ is the limit of the sequence $\tuple{g_k}_{g ∈ ω}$.
		
		$\L$ becomes an MU-category with the induced structure.
		Clearly, if $f_{n, ∞} ∘ g = f_{n, ∞} ∘ h$, then $f_{n, ∞} ∘ g ∘ k = g_{n, ∞} ∘ h ∘ k$, so $d(g ∘ k, h ∘ k) ≤ d(g, h)$.
		On the other hand, for every $f\maps X \to Z$, a chosen $\K$-sequence $\tuple{Z_*, z_*}$ with limit $\tuple{Z, z_{*, ∞}}$, and $m ∈ ω$ there is $n ∈ ω$ and a $\K$-map $f'\maps Z_m \from X_n$ such that $f' ∘ f_{n, ∞} = z_{m, ∞} ∘ f$ by (F1).
		Hence, if $f_{n, ∞} ∘ g = f_{n, ∞} ∘ h$, then $z_{m, ∞} ∘ f ∘ g = z_{m, ∞} ∘ f ∘ h$, and $f$ is $\tuple{1/m, 1/n}$-continuous.
		
		The induced MU-structure is independent on the choice of $\K$-sequences and limit maps up to MU-isomorphism.
		Let $\tuple{X'_*, f'_*}$ be another $\K$-sequence with limit $\tuple{X, f'_{*, ∞}}$.
		For every $m ∈ ω$ there is $n ∈ ω$ and a map $u\maps X'_m \from X_n$ such that $u ∘ f_{n, ∞} = f_{m, ∞}$ by (F1), and so for any $\L$-object $Y$ and any $\L$-maps $g, h\maps X \from Y$ if $g ≈_{1/n} h$, then $g ≈'_{1/m} h$, where $≈'$ comes from the alternative MU-structure.
		Hence, $\id_\L$ is MU-continuous, and a symmetric argument gives that $\id_\L$ is an MU-isomorphism.
		
		Every $\K$-object $X$ is still discrete in $\L$.
		We can choose constant identity sequence and the constant identity cone for definition of the metric.
		Then for every $ε > 0$ and all $\L$-maps $g, h\maps X \from Y$, if $g ≈_ε h$, then $g = h$.
		It follows that $\tuple{\K, \L}$ with the induced structure satisfies (L1), (L2), (F1), (F2) since the conditions with respect to the discrete MU-structure (with the exception of local completeness in (L1), which we have proved separately in the induced MU-structure).
		Condition (C) is true as well: the metric $d$ is defined exactly so that (C) becomes true.
	\end{proof}
\end{remark}

\begin{remark}[Classical Fraïssé theory] \label{rm:classical_fraisse}
	Let $L$ be a first-order language, let $\L$ be the opposite category of all countably generated $L$-structures and all embeddings (we formally take the opposite category so that the natively injective setting fits our projective setting), and let $\K$ be the full subcategory of all finitely generated $L$-structures.
	$\tuple{\K, \L}$ is a discrete free completion, and so a free completion with respect to the induced MU-structure on $\L$.
	(L1) and (L2) follow from the fact that countably generated structures are exactly limits of sequences of finitely generated structures, where the limit is essentially the union of an increasing countable chain.
	(F1) is true since every finitely generated substructure of the union of an increasing chain $⋃_{n ∈ ω} A_n$ is covered by some $A_n$.
	(F2) is true since all $\L$-maps are opposite monomorphisms.
	The fact that the conditions for being a free completion are always true in the classical setting seems to be the reason why these conditions are usually not spelled out explicitly.
	It follows that for every class $\F$ of finitely generated $L$-structures (viewed as a full subcategory of $\K$) we have that $\tuple{\F, σ\F}$ is a free completion and $σ\F ⊆ \L$ is full, and hence the characterization of Fraïssé limits applies.
	We also recover the correspondence between homogeneous countable $L$-structures and hereditary Fraïssé classes (Theorem~\ref{thm:age_limit}).
	
	The large homogeneity coming from the induced MU-structure is not interesting in the classical case: a countably generated structure $U$ is homogeneous in $σ\F$ if for every structure $X$ in $σ\F$, every pair of embeddings $f, g\maps X \to U$, any choice of increasing chain of finitely generated structures $⋃_{n ∈ ω} X_n = X$, and every $n ∈ ω$ there is an automorphism $h\maps U \to U$ such that $h ∘ g = f$ when restricted to $X_n$, which is the same as homogeneity in $\tuple{\F, σ\F}$.
	On the other hand, the topology on every $\L(X, Y)$ coming from the induced MU-structure is the topology of pointwise convergence, which is the correct topology when studying the automorphism group of the Fraïssé limit, e.g. in the context of the KPT correspondence~\cite{KPT05}.
	Moreover, the left uniformity of the completely metrizable topological group $\Aut(U)$ is induced by the metric coming from the induced MU-structure, while the two-sided uniformity is complete.
	For more details on abstractly inducing the topology on $\Aut(U)$ in the context of the KPT correspondence see \cite[Construction~2.6]{BBDK24}.
\end{remark}

\begin{remark}[Projective Fraïssé theory] \label{rm:projective_fraisse}
	Let us precisely describe how the projective Fraïssé theory of Irwin and Solecki~\cite{IS06} fits into our framework.
	Irwin and Solecki consider certain epimorphisms of \emph{topological $L$-structures}.
	For a first-order language $L$, a topological $L$-structure is an $L$-structure endowed with a zero-dimensional compact metrizable topology making the operations continuous and the relations closed (as subsets of the corresponding products).
	By a \emph{quotient map} of topological $L$-structures we mean a map that is both quotient of the corresponding $L$-structures and of the corresponding compact topological spaces, i.e. a continuous surjective $L$-homomorphism such that every relation holding in the codomain has a witness in the domain.
	Topological $L$-structures and quotient maps form an ambient category $\L$.
	
	Let $\K ⊆ \L$ denote the full subcategory of all finite $L$-structures.
	An $\L$-limit of a $\K$-sequence is the limit in $\MCptS$ with relations and operations defined coordinatewise.
	$\tuple{\K, \L}$ satisfies (L1), (F1), (F2).
	We have that $σ\K ⊆ \L$ is full and that $\tuple{\K, σ\K}$ is a discrete free completion and so a free completion with the induced MU-structure.
	(If $L$ is a relational language, then $σ\K = \L$, but this is not true in general.
	By \cite[Lemma~2.5]{IS06}, a topological $L$-structure $X$ is a $σ\K$-object if and only if every clopen partition of $X$ can be refined by partition induced by a congruence, but this is not the case e.g. for $2^ω$ with the shift map for $L$ consisting of one unary operation.)
	Since all $\L$-objects are metrizable compacta and all $\L$-maps are continuous, we may view $\L$ as an MU-category also as in Example~\ref{ex:cpt_mucat}.
	However, on $σ\K$ this “native” MU-structure is the same as the induced one.
	Note that Example~\ref{ex:Cantor} is a special case corresponding to the empty language.

	A class $\F$ of topological $L$-structures can be identified with a full subcategory $\F ⊆ \L$.
	Irwin and Solecki define a Fraïssé limit of $\F$ as (using our language) an $\L$-object $U$ that is cofinal and homogeneous in $\tuple{\F, \L}$ with respect to the discrete MU-structure and that for every $ε > 0$ admits an $ε$-monic $\L$-map onto a $\K$-object.
	However, if $\F ⊆ \K$, i.e. if all $\F$-objects are finite structures (and so $\F$ is a discrete MU-category), the last condition becomes equivalent to being a $σ\F$-object, and discrete homogeneity becomes equivalent to approximate homogeneity.
	In this case, $\tuple{\F, σ\F}$ is a free completion with $σ\F ⊆ \L$ being full, and the two definitions of a Fraïssé limit by Irwin–Solecki and ours coincide.
	
	More generally, classes $\F$ with special morphisms have been considered – first time in \cite{PS22} (we would denote their $\C^ω$ by $σ\C$).
	This corresponds to taking a non-full subcategory $\F ⊆ \K$.
	However for the theory to work, it is necessary for $\F$ to be $σ$-consistent in $\L$.
	This assumption is often neglected, but see \cite[after Lemma~2.5]{ChKR25}.
\end{remark}

\begin{example} \label{ex:IS}
	We consider the language with one binary relation.
	Let $\I_Δ$ be the category of all finite linear graphs $\II_n$, $n ∈ ω$, with quotient maps.
	By the general remarks above, $\tuple{\I_Δ, σ\I_Δ}$ is a free completion.
	It is easy to see that $\I_Δ$ is directed and countable.
	Irwin and Solecki~\cite{IS06} proved that $\I_Δ$ has the amalgamation property (for a simpler proof see \cite[Remark~3.11]{Kwiatkowska14}).
	So by the general theory we have the Fraïssé limit $ℙ_Δ$ of $\I_Δ$ in $σ\I_Δ$.
	Irwin and Solecki showed that $ℙ_Δ$ is the Cantor space with a closed equivalence relation (with only classes of size one or two) and that the corresponding topological quotient is the pseudo-arc $ℙ$.
	This way they characterized $ℙ$ as the unique homogeneous object in $σ\I$.
	By our theory we have the large homogeneity of $ℙ_Δ$ in $σ\I_Δ$, i.e. for all quotient maps $f, g\maps ℙ_Δ \to X$ onto a $σ\I_Δ$-object and every $ε > 0$ there is an automorphism $h\maps ℙ_Δ \to ℙ_Δ$ with $f ≈_ε g ∘ h$.
	We may consider the full subcategory $\Pσ\I ⊆ σ\I_Δ$ of \emph{arc-like pre-spaces}, i.e. of those $σ\I_Δ$-structures whose closed relation is an equivalence.
	The operation of taking the topological quotient induces an essentially surjective functor $\Pσ\I \to σ\I$ by \cite[Lemma~4.3, 4.5, and 4.7]{IS06}.
	Note that $ℙ_Δ$ is a Fraïssé object also in the category of arc-like pre-spaces $\Pσ\I$, which is disjoint from $\I_Δ$.
\end{example}

We have built our approximate Fraïssé theory on MU-categories along with the running example of the pseudo-arc and the arc-like continua, which culminated in Theorem~\ref{thm:pseudo-arc}.
Because of a result by Russo~\cite{Russo79}, this is essentially the only case of a class of $\P$-like continua for a family $\P$ of connected polyhedra with a Fraïssé limit.
However, there are Fraïssé limits of categories of circle-like continua with restricted classes of morphisms. This is the content of the next section.
We summarize the situation of full polyhedral categories in the following theorem.
Let $\SS$ denote a circle and $*$ a one-point space.

\begin{theorem} \label{thm:polyhedra}
	Let $\P$ be a non-empty family of non-empty connected polyhedra, identified with the corresponding full subcategory of $\CPolS$, so $σ\P$ is the category of all $\P$-like continua (we consider families $\P$ up to homeomorphic copies).
	\begin{enumerate}
		\item $σ\P$ has a Fraïssé object if and only if $\P$ has the amalgamation property, which happens if and only if $\P ⊆ \set{*, \II}$.
			The Fraïssé object is the pseudo-arc $\PP$, unless $\P = \set{*}$, in which case the Fraïssé object is $*$.
		\item $σ\P$ has a generic object if and only if $σ\P$ has a cofinal object, which happens if and only if $\P ⊆ \set{*, \II, \SS}$.
			The generic object is the \emph{universal pseudo-solenoid $\PP_Π$} (see Section~\ref{sec:pseudo-solenoid}), unless $\P ⊆ \set{*, \II}$ and we are in case~(i).
		\item There is always a generic object in $\tuple{\P, σ\P}$ and it is the pseudo-arc, unless $\P = \set{*}$.
	\end{enumerate}
\end{theorem}

Note there is a difference in being generic in $\tuple{\P, σ\P}$ and being generic in $σ\P$ – the Banach–Mazur game is played in $\P$ and $σ\P$, respectively.

	\begin{proof}
		We have that $\tuple{\P, σ\P}$ is a free completion by Theorem~\ref{thm:polyhedral_categories}, so Theorem~\ref{thm:fraisse_limit} applies.
		By Theorem~\ref{thm:fraisse_category}, we have a Fraïssé limit if and only if $\P$ is a Fraïssé MU-category, which reduces to having the amalgamation property by Observation~\ref{thm:locally_separable}.
		
		Claim~(iii) and the positive part of (i) follow from Theorem~\ref{thm:pseudo-arc}.
		By a result of Russo~\cite{Russo79}, there is no cofinal object in $σ\P$ unless $\P ⊆ \set{*, \II, \SS}$, hence we have the negative part of (ii).
		From this and from the fact that $\set{\SS}$ does not have the amalgamation property (\cite[Example~1]{Rogers70}, see also Proposition~\ref{thm:circle_no_amalgamation} below) we have the negative part of (i).
		The positive part of (ii) in the case $\P ⊆ \set{*, \I}$ follows from (i).
		If $\SS ∈ \P$, then the universal pseudo-solenoid $\PP_Π$ is cofinal in $σ\P$ by Corollary~\ref{thm:universal_pseudo-solenoid}.
		Since also every continuous surjection $\PP_Π \to \PP_Π$ is a near-homeomorphism (again by \ref{thm:universal_pseudo-solenoid}), $\PP_Π$ is generic in $σ\P$ by Observation~\ref{thm:fraisse_object_is_generic}.
	\end{proof}

\section{Circle-like continua and pseudo-solenoids}
\label{sec:circlelike}

The last section is devoted to the main application of our theory in this article.
Let $\SS$ denote the unit circle in the complex plane, and let $\S$ be the category of all continuous surjections of $\SS$.
By Theorem~\ref{thm:polyhedral_categories}, $σ\S$ is the category of all circle-like continua and all continuous surjections, and $\tuple{\S, σ\S}$ is a free completion.
Hence, by Theorem~\ref{thm:fraisse_limit} and \ref{thm:fraisse_category} there is a Fraïssé limit of $\S$ in $σ\S$ if and only if $\S$ is a Fraïssé MU-category, which reduces to the amalgamation property by Observation~\ref{thm:locally_separable}.
However, Rogers~\cite{Rogers70} have shown that $\S$ does not have the strict amalgamation property, and the same example shows that $\S$ does not have the amalgamation property (see Proposition~\ref{thm:circle_no_amalgamation} below).
Anyway, we would not obtain any new Fraïssé object since the only candidate for the limit is the pseudo-arc since it is generic over $\S$ by Theorem~\ref{thm:generic_Peano}. 
However, it turns out that the problem with amalgamation concerns \emph{degree zero maps}, and can be fixed by restricting the set of allowed \emph{degrees}.

Let $φ\maps \RR \to \SS$ be the canonical surjection $x \mapsto e^{ix}$.
We endow $\SS$ with the metric induced by $φ$, that is, the inner metric of the curve.
Recall that every continuous map $f\maps \SS \to \SS$ admits a continuous \emph{lifting} $\tilde{f}\maps ℝ \to ℝ$, i.e. a continuous map satisfying $φ ∘ \tilde{f} = f ∘ φ$.
The family of all such liftings is $\set{\tilde{f} + 2πk: k ∈ ℤ}$ where $\tilde{f}$ is any of them.

There is also a unique number $d ∈ ℤ$ such that $\tilde{f}(x + 2π) = \tilde{f}(x) + 2πd$ for every $x ∈ ℝ$.
This number is called the \emph{degree} or the \emph{winding number} of $f$ and is denoted by $\deg(f)$.
Every continuous map of nonzero degree is surjective.
We have $\deg(\id_\SS) = 1$ and $\deg(f ∘ g) = \deg(f) ⋅ \deg(g)$ for all continuous maps $f, g\maps \SS \to \SS$.

\begin{lemma} \label{thm:close_degrees}
	Let $f, g\maps \SS \to \SS$ be continuous.
	If $f ≈_π g$, then $\deg(f) = \deg(g)$.
	
	\begin{proof}
		Let $\tilde{f}, \tilde{g}\maps ℝ \to ℝ$ be the liftings of $f$ and $g$, respectively, such that $\tilde{f}(0) ≈_π \tilde{g}(0)$.
		Since the metric on $\SS$ is induced by the projection $φ$, we have $d(φ(x), φ(y)) = \abs{x - y}$ whenever $\abs{x - y} ≤ π$, i.e. $φ$ is an isometry on any interval of length $≤ π$.
		It follows that $\tilde{f} ≈_π \tilde{g}$ since otherwise there would be $x ∈ ℝ$ such that $\abs{\tilde{f}(x) - \tilde{g}(x)} = π = d(f(φ(x)), g(φ(x)))$.
		We have $2π \deg(f) = \tilde{f}(x + 2π) - \tilde{f}(x) ≈_{2π} \tilde{g}(x + 2π) - \tilde{g}(x) = 2π \deg(g)$ for every $x ∈ ℝ$, and hence $\deg(f) ≈_1 \deg(g)$.
		Since the degrees are integers, we are done.
	\end{proof}
\end{lemma}

\begin{remark}
	We view the monoid $ℤ$ with the multiplication as a category with one object and endow it with the discrete metric.
	The facts above concerning the degree can be summarized by saying that $\deg\maps \S \to ℤ$ is an MU-functor.
\end{remark}

\begin{proposition}[{based on \cite[Example~1]{Rogers70}}] \label{thm:circle_no_amalgamation}
	Let $f\maps \SS \to \SS$ be the map $z \mapsto z^2$, i.e. the canonical map of degree $2$.
	Let $g\maps \SS \to \SS$ be the map $z \mapsto z^2$ for $\Im(z) ≥ 0$ and $z \mapsto z^{-2}$ for $\Im(z) ≤ 0$, i.e. the surjection of degree $0$ corresponding to the tent map.
	There are no continuous surjections $f', g'\maps X \to \SS$ from a connected space $X$ such that $f' ∘ f ≈_π g' ∘ g$.
	In particular, the category $\S$ does not have the amalgamation property.
	
	\begin{proof}
		Let us endow $\SS^2$ with the $\ell_1$ metric and let us put $G := \set{\tuple{x, y} ∈ \SS^2: f(x) = g(y)}$ and $G_ε := \set{\tuple{x, y} ∈ \SS^2: f(x) ≈_ε g(y)}$ for every $ε > 0$.
		Pairs of continuous maps $f', g'\maps \SS \to \SS$ are in one-to-one correspondence with continuous maps $h\maps \SS \to \SS^2$.
		Moreover, $f ∘ f' = g ∘ g'$ or $f ∘ f' ≈_ε g ∘ g'$ if and only if $\rng(h) ⊆ G$ or $\rng(h) ⊆ G_ε$, respectively.
		Also $f'$ or $g'$ is surjective if and only if $π_1\im{\rng(h)} = \SS$ or $π_2\im{\rng(h)} = \SS$, respectively.
		
		In our case, $G$ consists of two similar components separated by distance of $π$, and the second projection of each component is just a half of the circle.
		For our map $f$, it is easy to see that whenever we have $f(x) ≈_ε y$ there is $x' ≈_{ε/2} x$ such that $f(x') = y$.
		It follows that for every $\tuple{x, y} ∈ G_ε$ we have $d(\tuple{x, y}, G) < ε/2$.
		Hence, $G_π$ is contained in the $π/2$-neighborhood of $G$, which has two components, and the second projection of each component still does not cover the whole circle (by one point), so there is no suitable map $h$.
	\end{proof}
\end{proposition}

\begin{definition}
	Let $Π$ denote the set of all primes.
	For every $P ⊆ Π$ let $\S_P ⊆ \S$ be the MU-subcategory of all $\S$-maps whose degree uses only primes from $P$, i.e. $\S_P := \deg\preim{D_P}$ where $D_P$ is the set of all $0 ≠ k ∈ ℤ$ such that if a prime $p$ divides $k$, then $p ∈ P$.
	For example, we have $D_∅ = \set{-1, 1}$, $D_{\set{2}} = \set{±2^n: n ∈ ℕ_0}$, and $D_Π = ℤ \setminus \set{0}$.
	Also, let $σ\S_P$ be the $σ$-closure of $\S_P$ in $\MCpt$ or equivalently in $σ\S$ (see Remark~\ref{rmk:sigma_closed}).
\end{definition}

\begin{proposition}
	$\tuple{\S_P, σ\S_P}$ is a free completion for every $P ⊆ Π$.
	
	\begin{proof}
		By Theorem~\ref{thm:polyhedral_categories} it is enough to prove that $\S_P$ is $σ$-consistent.
		So let $\tuple{X_*, f_*}$ and $\tuple{Y_*, g_*}$ be sequences in $\S_P$ with limits $\tuple{X_∞, f_{*, ∞}}$ and $\tuple{Y_∞, g_{*, ∞}}$ such that $X_∞ = Y_∞$.
		Let $n ∈ ω$ and $0 < ε < π/2$.
		Since $\tuple{\S, σ\S}$ is a free completion, there is $m ∈ ω$ and an $\S$-map $f\maps X_m \to Y_n$ such that $f ∘ f_{m, ∞} ≈_ε g_{n, ∞}$.
		We show that $f$ is an $\S_P$-map, i.e. $\deg(f) ∈ D_P$.
		
		There is $δ > 0$ such that $f$ is $\tuple{ε, δ}$-continuous, $n' ≥ n$, and an $\S$-map $g\maps Y_{n'} \to X_m$ such that $g ∘ g_{n', ∞} ≈_δ f_{m, ∞}$.
		It follows that $f ∘ g ∘ g_{n', ∞} ≈_ε f ∘ f_{m, ∞} ≈_ε g_{n, ∞}$, and so $f ∘ g ≈_{2ε} g_{n, n'}$.
		Since $2ε < π$, by Lemma~\ref{thm:close_degrees} we have $\deg(f) ⋅ \deg(g) = \deg(f ∘ g) = \deg(g_{n, n'}) ∈ D_P$.
		Finally, $\deg(f), \deg(g) ∈ D_P$ since $D_P$ is closed under divisors.
	\end{proof}
\end{proposition}

\begin{proposition}[Rogers~{\cite[Uniformization Theorem]{Rogers70}}]
	Let $\S' ⊆ \S$ denote the subcategory of all piecewise linear maps (i.e. continuous surjections $f\maps \SS \to \SS$ such that the restriction of the lifting $\tilde{f}\maps [0, 2π] \to ℝ$ is piecewise linear).
	For all $\S'$-maps $f, g$ of non-zero degree there are $\S'$-maps $f', g'$ such that $f ∘ f' = g ∘ g'$ and $\deg(f ∘ f') = \deg(g ∘ g') = \lcm(\deg(f), \deg(g))$, so the degrees of $f'$ and $g'$ are non-zero and optimal.
\end{proposition}

It follows that $\S' ∩ \S_P$ has the strict amalgamation property for every $P ⊆ Π$.
Since piecewise linear surjections are dense in all continuous surjections, and since close maps have equal degrees, $\S' ∩ \S_P$ is dominating in $\S_P$, and so $\S_P$ has the amalgamation property by Observation~\ref{thm:wide_dominating}.
If we also use Observation~\ref{thm:locally_separable}, we obtain the following.

\begin{corollary}
	$\S_P$ is a Fraïssé MU-category for every $P ⊆ Π$.
\end{corollary}

\begin{theorem} \label{thm:fraisse_circle-like}
	For every $P ⊆ Π$, there is a Fraïssé limit $ℙ_P$ of $\S_P$ in $σ\S_P$.
	The continuum $ℙ_P$ is characterized (up to a homeomorphism) by any of the following conditions.
	\begin{enumerate}
		\item $ℙ_P$ is a projective object in $\tuple{\S_P, σ\S_P}$.
		\item $ℙ_P$ is a homogeneous object in $\tuple{\S_P, σ\S_P}$.
		\item $ℙ_P$ is a generic object in $\tuple{\S_P, σ\S_P}$.
		\item $ℙ_P$ is a projective object in $σ\S_P$.
		\item $ℙ_P$ is a homogeneous object in $σ\S_P$.
		\item $ℙ_P$ is a generic object in $σ\S_P$.
	\end{enumerate}
	
	\begin{proof}
		Since $\S_P$ is a Fraïssé MU-category and $\tuple{\S_P, σ\S_P}$ is a free completion, we have the Fraïssé limit characterized by any of the conditions in Theorem~\ref{thm:fraisse_limit}.
		As with the pseudo-arc, the cofinality is trivial (see Remark~\ref{thm:cofinality_trivial}), and so $ℙ_P$ is characterized by any of (i), (ii), (iv), (v).
		From the general theory we have uniqueness up to a $σ\S_P$-isomorphism, but $σ\S_P$-isomorphisms are exactly homeomorphisms between $σ\S_P$-objects from the construction of $σ$-closure (see Remark~\ref{rmk:sigma_closure}).
		Also, the genericity conditions (iii), (vi) are weaker than being a Fraïssé limit (Proposition~\ref{thm:dominating_sequence}, Observation~\ref{thm:fraisse_object_is_generic}), and yet guarantee uniqueness.
	\end{proof}
\end{theorem}

Recall that a hereditarily indecomposable circle-like continuum is called a \emph{pseudo-solenoid}.
We will observe that $\PP_P$ is the unique \emph{$P$-adic pseudo-solenoid}, but first we need to recall the notion of a \emph{type} of a circle-like continuum according to degrees of maps in the corresponding inverse sequence.
Moreover, the Fraïssé limit properties of $\PP_P$ are stated in terms of $σ\S_P$-maps, but which continuous surjections are $σ\S_P$-maps?
The statement that $σ\S_P$-maps are exactly those $σ\S$-maps that can be sufficiently approximated by an $\S_P$-map will be made more precise by considering a certain extension of the notion of the degree of a map.

\subsection{Types of circle-like continua and continuous surjections}

\cdef \SN {%
	\clo{ℕ}%
}

\cdef \SQ {%
	\clo{ℚ}%
}

\cdef \absdeg [1]{%
	\abs{\deg(#1)}%
}

\cdef \SubQ {%
	\operatorname{Sub}(ℚ)%
}

Let us recall the notion of \emph{supernatural numbers}.
By $ℕ$ we denote the set of all natural numbers including $0$, and by $ℕ_+$ we denote the set of all (strictly) positive natural numbers.
A positive supernatural number is a function $s\maps Π \to ℕ ∪ \set{∞}$ representing the formal product $∏_{p ∈ Π} p^{s(p)}$.
We denote the set of all positive supernatural numbers $(ℕ ∪ \set{∞})^Π$ by $\SN_+$.
The set $\SN_+$ carries a natural multiplication defined by $(s ⋅ s')(p) = s(p) + s'(p)$ for every $p ∈ Π$, and a natural order defined by $s ≤ s'$ if $s(p) ≤ s'(p)$ for every $p ∈ Π$.
In fact, the order is exactly the divisibility order induced by the multiplication, i.e. $s ≤ s'$ if and only if there is some $t$ with $t ⋅ s = s'$.
Since there is a natural infinitary addition on the set $ℕ ∪ \set{∞}$, we may define an infinitary multiplication on $\SN_+$ by $(∏_{i ∈ I} s_i)(p) = ∑_{i ∈ I} s_i(p)$.
We have the monoid and divisibility order embedding $ℕ_+ \to \SN_+$ mapping $n$ to the supernatural number $s$ corresponding to the prime decomposition of $n$, i.e. $n = ∏_{p ∈ Π} p^{s(p)}$.
Because of this, we view $ℕ_+ ⊆ \SN_+$.
Together, we have for example $2^∞ ∈ \SN_+$, $s = ∏_{p ∈ Π} p^{s(p)}$ for every $s ∈ \SN_+$, $1 ∈ \SN_+$ is the divisibility minimum, and $∏_{p ∈ Π} p^∞ ∈ \SN_+$ is the divisibility maximum.
Finally, we define the set of all supernatural numbers $\SN := \SN_+ ∪ \set{0}$, and we extend the multiplication by $s ⋅ 0 = 0$ and the order by $s ≤ 0$ for every $s ∈ \SN$.
We also have the extended embedding $ℕ \to \SN$.

We consider an equivalence $\sim$ on $\SN$ where $0$ is equivalent only to itself and where for $s, s' ∈ \SN_+$ we have $s \sim s'$ if and only if $s\fiber{∞} = (s')\fiber{∞}$ and $\set{p ∈ Π: s(p) ≠ s'(p)}$ is finite.
We denote the $\sim$-equivalence class of $s ∈ \SN$ by $[s]$, but we also identify an equivalence class with its representative when convenient, e.g. we may view $1 = [1] = ℕ_+ ∈ \SN/{\sim}$.
It is easy to see that $\sim$ is indeed an equivalence and that it is moreover a congruence: for every $s \sim s'$ and $t \sim t'$ we have $s ⋅ t \sim s' ⋅ t'$.
Hence, we have the induced multiplication on $\SN/{\sim}$, and the divisibility order induced by the multiplication is induced also by the original divisibility order, i.e. for $S, S' ∈ \SN/{\sim}$ there is some $T ∈ \SN/{\sim}$ such that $T ⋅ S = S'$ if and only if there are some $s ∈ S$ and $s' ∈ S'$ with $s ≤ s'$, in which case there is $t ∈ \SN$ with $t ⋅ s = s'$ and moreover $t ⋅ S ⊆ S'$.
The order is indeed antisymmetric since the equivalence is convex: if $s ≤ s'$ and $s \sim s'$, then also $s \sim t$ for every $t ∈ [s, s']$, and if $S ≤ S' ≤ S$, then there is $s ∈ S$ and $t, t' ∈ \SN$ such that $t ⋅ s ∈ S'$ and $t' ⋅ t ⋅ s ∈ S$, so $s ≤ t ⋅ s ≤ t' ⋅ t ⋅ s \sim s$ and $s \sim t ⋅ s$ and $S = S'$.

\begin{definition}
	For $S, S' ∈ \SN_+/{\sim}$ we say that a map $f\maps S \to S'$ is a \emph{positive multiplication} if for some $t ∈ \SN_+$ we have $f(s) = t ⋅ s$ for every $s ∈ S$.
	Similarly, for $S, S' ∈ \SN/{\sim}$ we say that a map $f\maps S ∪ \set{0} \to S' ∪ \set{0}$ is a \emph{multiplication} if for some $t ∈ \SN$ we have $f(s) = t ⋅ s$ for every $s ∈ S$.
	Note that a multiplication $S ∪ \set{0} \to S' ∪ \set{0}$ is either the extension of a positive multiplication $S \to S'$ by $0 \mapsto 0$ (in which case we also call it positive), or it is a constant zero map (i.e. the multiplication by $0$).
	
\end{definition}

\begin{observation} \label{thm:multiplication}
	A multiplication $f\maps S ∪ \set{0} \to S' ∪ \set{0}$ mapping $s ∈ S$ to $s' ∈ S' ∪ \set{0}$ exists if and only if $s ≤ s'$, and it is unique.
	We know that $s ≤ s'$ if and only if there is $t ∈ \SN$ such that $t ⋅ s = s'$, and every such $t$ induces a multiplication since $\sim$ is a congruence.
	If $s = 0$ (i.e. $S = \set{0}$) and $s' ≠ 0$, then there is no solution to the equation $t ⋅ s = s'$.
	If $s ≠ 0$ and $s' = 0$, then the unique solution is $t = 0$.
	If $s = s' = 0$, then every $t ∈ \SN$ represents $f$.
	Otherwise, the multiplication $f$ is positive and $t(p)$ for $p ∈ Π \setminus s\fiber{∞}$ is uniquely determined by $t ⋅ s = s'$, while $t(p)$ for $p ∈ s\fiber{∞}$ can be arbitrary.
	If $t ⋅ s = t' ⋅ s$, then $t(p) = t'(p)$ for every $p ∈ Π \setminus s\fiber{∞}$, and so $t ⋅ u = t' ⋅ u$ for every $u ∈ S$ since $s\fiber{∞} = u\fiber{∞}$.
\end{observation}

\begin{definition}
	We define the category of \emph{positive types} $\T_+$, where the set of objects is $\SN_+/{\sim}$ and a $\T_+$-map $S \to S'$ is a positive multiplication.
	Composition is the actual composition and identities are the actual identities.
	We also define the category of \emph{types} $\T$ by freely adding a zero object to $\T_+$, i.e. there is one new object $0$, and between any two $\T$-objects there is one new map, a zero map, such that any composition with a zero map is a zero map.
	Equivalently, we may view $\T$-objects as $\SN/{\sim}$ and $\T$-maps $S \to S'$ as multiplications $S ∪ \set{0} \to S' ∪ \set{0}$.
	
	For every $P ⊆ Π$ let $P^∞$ denote the supernatural number $∏_{p ∈ P} p^∞$.
	The types of the form $[P^∞]$ together with $[0]$ are exactly the types with a $≤$-smallest representative.
	Sometimes we denote types by their representatives and call $[0]$ the “type $0$”, $[1] = S_∅$ the “type $1$”, and $[Π^∞]$ the “type $∞$”.
	Types are ordered as $\SN/{\sim}$.
	Clearly, $1$ is the minimum, $0$ the maximum, and $[Π^∞]$ the maximum among positive types.
	
	Note that for every positive type $S$ and $s ∈ S$, the $\T_+$-maps from $S$ are in a one-to-one correspondence with positive supernatural numbers $s' ≥ s$: by Observation~\ref{thm:multiplication} for every $s' ≥ s$ there is a unique positive multiplication $[s] \to [s']$ mapping $s$ to $s'$.
	(In particular, $\T_+$-maps from $1$ are in one-to-one correspondence with $\SN_+$: every $s ∈ \SN_+$ induces the multiplication by $s\maps [1] \to [s]$.)
	It follows that there is a $\T_+$-map $S \to S'$ if and only is $S ≤ S'$.
	It also follows that $∞$ is the terminal object in $\T_+$ since $Π^∞$ is the unique representative of $[Π^∞]$.
\end{definition}

\begin{notation}
	Let $\tuple{X_*, u_*}$ be an $\S$-sequence with a $σ\S$-limit $\tuple{X_∞, u_{*, ∞}}$.
	For every $n_0 ∈ ω$ we write $\absdeg{u_{n_0, ∞}} := ∏_{n ≥ n_0} \absdeg{u_n} ∈ \SN$.
	Similarly, for every $\S$-map $f\maps X_{n_0} \to Y$ we write $\absdeg{f ∘ u_{n_0, ∞}} := \absdeg{f} ⋅ \absdeg{u_{n_0, ∞}}$.
\end{notation}

\begin{lemma} \label{thm:supernatural_coherence}
	Let $\tuple{X_*, u_*}$, $\tuple{Y_*, v_*}$ be sequences in $\S$ and let $\tuple{X_∞, u_{*, ∞}}$, $\tuple{Y_∞, v_{*, ∞}}$ be their limits in $σ\S$ such that $X_∞ = Y_∞$.
	If $f\maps X_{m_0} \to Y_{n_0}$ is an $\S$-map such that $f ∘ u_{m_0, ∞} ≈_π v_{n_0, ∞}$, then $\absdeg{f ∘ u_{m_0, ∞}} = \absdeg{v_{n_0, ∞}}$.
	
	\begin{proof}
		First, for every $m_1 ≥ m_0$ there is arbitrarily large $n_1 ≥ n_0$ and an $\S$-map $g\maps Y_{n_1} \to X_{m_1}$ such that $f ∘ u_{m_0, m_1} ∘ g ≈_π v_{n_0, n_1}$, and so $\abs{\deg(f ∘ u_{m_0, m_1} ∘ g)} = \abs{\deg(v_{n_0, n_1})}$ by Lemma~\ref{thm:close_degrees}.
		To find such map $g$, we take $ε > 0$ such that $f ∘ u_{m_0, ∞} ≈_{π - ε} v_{n_0, ∞}$ and $δ > 0$ such that $f ∘ u_{m_0, m_1}$ is $\tuple{ε, δ}$-continuous.
		By (F1) there is $g\maps Y_{n_1} \to X_{m_1}$ with $n_1 ≥ n_0$ arbitrarily large such that $g ∘ v_{n_1, ∞} ≈_δ u_{m_1, ∞}$, and so
		\[
			f ∘ u_{m_0, m_1} ∘ g ∘ v_{n_1, ∞} ≈_ε f ∘ u_{m_0, ∞} ≈_{π - ε} v_{n_0, ∞}.
		\]
		We get $f ∘ u_{m_0, m_1} ∘ g ≈_π v_{n_0, n_1}$ since $v_{n_1, ∞}$ is a surjection.
		We have 
		\[
			\absdeg{f ∘ u_{m_0, m_1}} ≤ \absdeg{f ∘ u_{m_0, m_1} ∘ g} = \absdeg{v_{n_0, n_1}} ≤ \absdeg{v_{n_0, ∞}}
		\]
		for arbitrary $m_1 ≥ m_0$, and so $\absdeg{f ∘ u_{m_0, ∞}} ≤ \absdeg{v_{n_0, ∞}}$.
		
		Since we could have chosen $δ < π$ above, we can apply the already proved half of the lemma to $g\maps Y_{n_1} \to X_{m_1}$, i.e. we have 
		$\absdeg{g ∘ v_{n_1, ∞}} ≤ \absdeg{u_{m_1, ∞}}$.
		Hence we obtain the other inequality: 
		\[
			\absdeg{v_{n_0, ∞}} = \absdeg{f ∘ u_{m_0, m_1} ∘ g ∘ v_{n_1, ∞}} ≤ \absdeg{f ∘ u_{m_0, ∞}}.
			\qedhere
		\]
	\end{proof}
\end{lemma}

\begin{construction} \label{con:type_functor}
	We define a contravariant \emph{type functor} $T\maps σ\S \to \T$.
	For every circle-like continuum $X$ we define its type $T(X) ∈ \SN/{\sim}$ as follows.
	Pick any $\S$-sequence $\tuple{X_*, u_*}$ with limit $\tuple{X, u_{*, ∞}}$.
	If $\deg(u_n) = 0$ for infinitely many $n ∈ ω$, we put $T(X) := 0$.
	Otherwise, we take any $n_0 ∈ ω$ such that $\deg(u_n) ≠ 0$ for every $n ≥ n_0$ and put $T(X) := \absdeg{u_{n_0, ∞}}$ viewed as an element of $\SN/{\sim}$.
	We say that $T(X)$ is the \emph{type} of $X$ as well as the \emph{type} of $\tuple{X_*, u_*}$.
	
	For every continuous surjection $f\maps X \to Y$ between circle-like continua we take any $\S$-sequences $\tuple{X_*, u_*}$, $\tuple{Y_*, v_*}$ with limits $\tuple{X, u_{*, ∞}}$, $\tuple{Y, v_{*, ∞}}$, and any $\S$-map $f_0\maps X_{m_0} \to Y_{n_0}$ such that $v_{n_0, ∞} ∘ f ≈_π f_0 ∘ u_{m_0, ∞}$ and such that $s := \absdeg{v_{n_0, ∞}} ∈ T(Y)$ and $s' := \absdeg{u_{m_0, ∞}} ∈ T(X)$ (i.e. each of them is positive if possible),
	and we let $T(f)$ be the unique multiplication $T(Y) ∪ \set{0} \to T(X) ∪ \set{0}$ that maps $s$ to $\absdeg{f_0} ⋅ s'$.
	
	\begin{proof}
		We need to prove that everything is well-defined and that $T$ is indeed a functor.
		We prove a series of claims.
		
		\begin{enumerate}[label=(\arabic*), widest=10]
			\item $T(X)$ does not depend on the choice of the sequence $\tuple{X_*, u_*}$ and $n_0 ∈ ω$.
				Let $T(X, u_*)$ denote the type obtained from the sequence $\tuple{X_*, u_*}$.
				Clearly, $\absdeg{u_{n_0, ∞}} \sim \absdeg{u_{n_1, ∞}}$ for every $n_1 ≥ n_0$ if $\deg(u_{n_0, n_1}) ≠ 0$, so $T(X, u_*)$ is well-defined.
				Let $\tuple{X'_*, u'_*}$ be another $\S$-sequence with limit $\tuple{X, u'_{*, ∞}}$.
				We show that $T(X, u_*) ≤ T(X, u'_*)$, and so by a symmetrical argument they are equal.
				By (F1), for every $n'_0 ∈ ω$ there is $n_0 ∈ ω$ and an $\S$-map $g\maps X_{n_0} \to X'_{n'_0}$ such that $g ∘ u_{n_0, ∞} ≈_π u'_{n'_0, ∞}$ and so by Lemma~\ref{thm:supernatural_coherence} we have $\absdeg{g ∘ u_{n_0, ∞}} = \absdeg{u'_{n'_0, ∞}}$.
				If the right-hand side is non-zero for some $n'_0$, we have $T(X, u_*) ∋ \absdeg{u_{n_0, ∞}} ≤ \absdeg{u'_{n'_0, ∞}} ∈ T(X, u'_*)$.
				Otherwise, $T(X, u_*) ≤ 0 = T(X, u'_*)$.
			
			\item Let $T(f, u_*, v_*, f_0)$ denote the definition of $T(f)$ for fixed sequences $u_*$, $v_*$ and a fixed $\S$-map $f_0\maps X_{m_0} \to Y_{n_0}$.
				We claim this is well-defined.
				By Observation~\ref{thm:multiplication}, $T(f, u_*, v_*, f_0)$ is a well-defined multiplication $T(Y) ∪ \set{0} \to T(X) ∪ \set{0}$ if $\absdeg{v_{n_0, ∞}} ≤ \absdeg{f_0 ∘ u_{m_0, ∞}}$.
				We have $f_0 ∘ u_{m_0, ∞} ≈_{π - ε} v_{n_0, ∞} ∘ f$ for some $ε > 0$.
				By (F1), for every $n_1 ≥ n_0$ there is $m_1 ≥ m_0$ and $f_1\maps X_{m_1} \to Y_{n_1}$ with $f_1 ∘ u_{m_1, ∞} ≈_δ v_{n_1, ∞} ∘ f$ for $δ > 0$ such that $v_{n_0, n_1}$ is $\tuple{ε, δ}$-continuous.
				Hence, we have 
				\begin{gather*}
					v_{n_0, n_1} ∘ f_1 ∘ u_{m_1, ∞} ≈_ε v_{n_0, ∞} ∘ f ≈_{π - ε} f_0 ∘ u_{m_0, ∞}, \\
					\absdeg{v_{n_0, n_1}} ≤ \absdeg{v_{n_0, n_1} ∘ f_1 ∘ u_{m_1, ∞}} = \absdeg{f_0 ∘ u_{m_0, ∞}}
				\end{gather*}
				by Lemma~\ref{thm:supernatural_coherence}.
				Since $n_1$ was arbitrary, we are done.
				
				\item $T(f, u_*, v_*, f_0)$ does not depend on $f_0$, and so can be denoted by $T(f, u_*, v_*)$.
					Note that is this trivial if $T(X) = 0$ or $T(Y) = 0$.
					Otherwise, let $f_1\maps X_{m_1} \to Y_{n_1}$ be another suitable $\S$-map.
					Without loss of generality, $n_0 ≤ n_1$.
					We have $f_i ∘ u_{m_i, ∞} ≈_{π - ε_i} v_{n_i, ∞} ∘ f$ for some $ε_i > 0$ and $i ∈ \set{0, 1}$.
					Let $δ ∈ (0, ε_1)$ be such that $v_{n_0, n_1}$ is $\tuple{ε_0, δ}$-continuous.
					By (F1) there is $m ≥ m_0, m_1$ and an $\S$-map $g\maps X_m \to Y_{n_1}$ such that $g ∘ u_{m, ∞} ≈_δ v_{n_1, ∞} ∘ f$.
					
					Since $δ < ε_1$, we have $g ∘ u_{m, ∞} ≈_π f_1 ∘ u_{m_1, ∞}$.
					Hence, $\absdeg{g} = \absdeg{f_1 ∘ u_{m_1, m}}$ and $T(f, u_*, v_*, g) = T(f, u_*, v_*, f_1)$.
					Since $v_{n_0, n_1}$ is $\tuple{ε_0, δ}$-continuous, we have $f_0 ∘ u_{m_0, ∞} ≈_{π - ε_0} v_{n_0, ∞} ∘ f ≈_{ε_0} v_{n_0, n_1} ∘ g ∘ u_{m, ∞}$.
					Hence, $\absdeg{f_0 ∘ u_{m_0, m}} = \absdeg{v_{n_0, n_1} ∘ g}$ and
					\[
						T(f, u_*, v_*, f_0) = T(f, u_*, v_*, f_0 ∘ u_{m_0, m}) = T(f, u_*, v_*, v_{n_0, n_1} ∘ g)
					\]
					since all three multiplications map $s$ to $\abs{\deg(f_0)} ⋅ s'$.
					Also, $T(f, u_*, v_*, g)$ maps $\absdeg{v_{n_1, ∞}}$ to $\absdeg{g ∘ u_{m_1, ∞}}$, and so maps $\absdeg{v_{n_0, ∞}}$ to $\absdeg{v_{n_0, n_1} ∘ g ∘ u_{m_1, ∞}}$.
					Hence,
					\[
						T(f, u_*, v_*, v_{n_0, n_1} ∘ g) = T(f, u_*, v_*, g) = T(f, u_*, v_*, f_1),
					\]
					and altogether, we have $T(f, u_*, v_*, f_0) = T(f, u_*, v_*, f_1)$.
				
				\item $T(f, u_*, v_*)$ contravariantly preserves composition, i.e. if $g\maps Y \to Z$ is another $σ\S$-map and $\tuple{Z_*, w_*}$ is an $\S$-sequence with limit $\tuple{Z, w_{*, ∞}}$, we have 
					\[
						T(g ∘ f, u_*, w_*) = T(f, u_*, v_*) ∘ T(g, v_*, w_*).
					\]
					We pick $k_0 ∈ ω$ such that $t := \absdeg{w_{k_0, ∞}} ∈ T(Z)$.
					There is $n_0 ∈ ω$ and an $\S$-map $g_0\maps Y_{n_0} \to Z_{k_0}$ such that $g_0 ∘ v_{n_0, ∞} ≈_{π/2} w_{k_0, ∞} ∘ g$.
					By possibly replacing $n_0$ with $n'_0$ and $g_0$ with $g_0 ∘ v_{n_0, n'_0}$ we may arrange that $t' := \absdeg{v_{n_0, ∞}} ∈ T(Y)$.
					There is $m_0 ∈ ω$ such that $t'' := \absdeg{u_{m_0, ∞}} ∈ T(X)$ and an $\S$-map $f_0\maps X_{m_0} \to Y_{n_0}$ with $f_0 ∘ u_{m_0, ∞} ≈_δ v_{n_0, ∞} ∘ f$ where $δ ∈ (0, π)$ is such that $g_0$ is $\tuple{π/2, δ}$-continuous.
					We have
					\[
						g_0 ∘ f_0 ∘ u_{m_0, ∞} ≈_{π/2} g_0 ∘ v_{n_0, ∞} ∘ f ≈_{π/2} w_{k_0, ∞} ∘ g ∘ f.
					\]
					Hence, $T(g ∘ f, u_*, w_*)$ maps $t$ to $\absdeg{g_0 ∘ f_0} ⋅ t''$.
					If $T(g, v_*, w_*)$ is the multiplication by $r$ and $T(f, u_*, v_*)$ is the multiplication by $r'$, we have $r' ⋅ r ⋅ t = r' ⋅ \absdeg{g_0} ⋅ t' = \absdeg{g_0 ∘ f_0} ⋅ t''$, so $T(g ∘ f, u_*, w_*)$ is the multiplication by $r' ⋅ r$.
				
				\item $T(\id_X, u_*, u'_*) = \id_{T(X)}$ for every pair of sequences $u_*, u'_*$ with limit $X$.
					It will follow that $T(f, u_*, v_*)$ does not depend on the choice of $u_*$ and $v_*$ since for every alternative pair of sequences $u'_*$ and $v'_*$ we would have 
					\[
						T(f, u'_*, v'_*) = T(\id_X, u'_*, u_*) ∘ T(f, u_*, v_*) ∘ T(\id_Y, v_*, v'_*) = T(f, u_*, v_*).
					\]
					It will also follow that $T(\id_X) = \id_{T(X)}$.
					So let $\tuple{X_*, u_*}$ and $\tuple{X'_*, u'_*}$ be $\S$-sequences with limits $\tuple{X, u_{*, ∞}}$ and $\tuple{X, u'_{*, ∞}}$.
					There is an $\S$-map $f_0\maps X_{n_0} \to X'_{n'_0}$ such that $f_0 ∘ u_{n_0, ∞} ≈_π u'_{n'_0, ∞}$ and $\absdeg{u_{n_0, ∞}} ∈ T(X)$ and $\absdeg{u'_{n'_0, ∞}} ∈ T(Y)$.
					Hence, $T(\id_X, u_*, u'_*)$ maps $\absdeg{u'_{n'_0, ∞}}$ to $\absdeg{f_0 ∘ u_{n_0, ∞}}$.
					But by Lemma~\ref{thm:supernatural_coherence} we have $\absdeg{u'_{n'_0, ∞}} = \absdeg{f_0 ∘ u_{n_0, ∞}}$, and so $T(\id_X, u_*, u'_*) = \id_X$.
				\qedhere
		\end{enumerate}
	\end{proof}
\end{construction}

\begin{observation}
	Let $I\maps \S \to σ\S$ denote the inclusion functor, and let $J\maps ℤ \to \T$ be the functor mapping every $k ∈ ℤ$ to the multiplication by $\abs{k}\maps [1] \to [1]$.
	Then we have $T ∘ I = J ∘ \deg$, so $T$ may be viewed as an extension of the degree / winding number from $\S$ to $σ\S$.
	To see this, for an $\S$-map $f$, we just consider the constant $\S$-sequence consisting of $\id_\SS$, so we have $T(\SS) =  [1]$ and $T(f)(1) = \abs{\deg(f)}$.
	
	Note that it is natural to forget the sign of the degree of an $\S$-map.
	To consider the degree of a continuous surjection $f\maps X \to Y$ between \emph{some} circles (as opposed to \emph{the} circle $\SS$), we would need to choose homeomorphisms $g\maps \SS \to X$ and $h\maps Y \to \SS$ and to put $\deg(f) := \deg(h ∘ f ∘ g)$.
	However, different choices of homeomorphisms $g, h$ may change the sign of the degree.
	
	In fact, if we consider the MU-functor $\abs{\deg}\maps \S \to ℕ$, where $ℕ$ is viewed as a discrete MU-category with a single object, and if by $σℕ$ we denote the opposite category of $\T$, endowed with the discrete MU-structure, then it can be shown that $\tuple{ℕ, σℕ}$ is a free completion of MU-categories, and $T = σ\abs{\deg}\maps σ\S \to σℕ$ is the unique MU-continuous extension.
	This also gives an example of a free completion of MU-categories where the MU-structure is discrete (cf.~Remark~\ref{rm:discrete_fraisse}).
	We will prove the MU-continuity in Observation~\ref{thm:T_MU-continuous}.
\end{observation}

\begin{observation}
	The category $\T$ is rigid and skeletal, i.e. if $f\maps S \to S'$ is a $\T$-isomorphism, then $S = S'$ and $f = \id$.
	This follows from Observation~\ref{thm:multiplication} because for $s ∈ S$ we have $s ≤ f(s) ≤ f\inv(f(s)) = s$.
	It follows that for a homeomorphism $f\maps X \to Y$ of circle-like continua we have $T(X) = T(Y)$ and $T(f) = \id$.
\end{observation}

\begin{remark} \label{rm:Cech_cohomology}
	Classification of solenoids and pseudo-solenoids is often stated in terms of the first \emph{Čech cohomology} $H^1$.
	Let us make a clear connection between $H^1$ and our type functor $T$.
	The functor $H^1\maps σ\S \to \Ab$ can be equivalently described as $[-, \SS]$, i.e. for every circle-like continuum $X$, $[X, \SS]$ denotes the set of all continuous surjections $X \to \SS$ up to homotopy, endowed with the group structure induced by the pointwise multiplication in $\SS ⊆ \CC$,
	see e.g. \cite[page~42]{Krasinkiewicz76}.
	
	Moreover, if $\tuple{X_*, u_*}$ is an $\S$-sequence with limit $\tuple{X, u_{*, ∞}}$, applying $[-, \SS]$ preserves the limit.
	To see this, every $σ\S$-map $X \to \SS$ approximately factorizes through $u_*$ by (F1), and any sufficiently close maps to $\SS$ are homotopic.
	On the other hand, for any $\S$-maps $f, g\maps X_n \to \SS$ such that $f ∘ u_{n, ∞}$ and $g ∘ u_{n, ∞}$ are homotopic, there are maps $f ∘ u_{n, ∞} = h_0 ≈_π h_1 ≈_π \cdots ≈_π h_k = g ∘ u_{n, ∞}\maps X \to \SS$, which can be approximately factorized through $u_*$, and so	there is $m ≥ n$ such that $f ∘ u_{n, m}$ and $g ∘ u_{n, m}$ are homotopic.
	
	Since the degree induces an isomorphism $[\SS, \SS] \to ℤ$, the abelian groups $H^1(X)$ are limits of sequences of group endomorphisms of the additive group $ℤ$, and so are isomorphic to subgroups of $ℚ$.
	Together, we may view $H^1\maps σ\S \to \SubQ$, where $\SubQ$ denotes the category of all subgroups of $ℚ$ and group homomorphisms.
	
	Now for every subgroup $\set{0} ≠ G ⊆ ℚ$ we may consider $n := \min(G ∩ ℤ_+)$, and for every $q ∈ G$ we put $s(q, G) := \sup_{\SN}\set{k ∈ ℕ_+: q/k ∈ G}$.
	Then we have 
	\[
		G = \set{n l / k: l ∈ ℤ, k ∈ ℕ_+, k ≤ s(n, G)}.
	\]
	Pick $q ∈ G \setminus \set{0}$ and put $Q(G) := [s(q, G)] ∈ \SN/{\sim}$.
	For every group homomorphism $f\maps G \to G'$ we have $s(q, G) ≤ s(f(q), G')$, and we may let $Q(f)\maps Q(G) \to Q(G')$ be the supernatural multiplication $s(q, G) \mapsto s(f(q), G')$.
	In fact, $Q(G)$ and $Q(f)$ do not depend on $q$, and the construction gives a functor $Q\maps \SubQ \to \T$.
	It turns out that $T = Q ∘ H^1$, and $Q$ is the “rigid skeletal modification” of $\SubQ$, i.e. $f\maps G \to G'$ is an isomorphism if and only if $Q(G) = Q(G')$ and $Q(f) = \id$.
\end{remark}

\begin{observation} \label{thm:T_MU-continuous}
	The type functor $T\maps σ\S \to \T$ is MU-continuous when $\T$ is endowed with the $0$-$1$ discrete MU-structure (and $T$ is viewed as a covariant functor from $σ\S$ to the opposite category of $\T$).
	That means, for every $σ\S$-object $Y$ there is $δ > 0$ such that for every $σ\S$-object $X$ and all $σ\S$-maps $f, g\maps X \to Y$ such that $f ≈_δ g$ we have $T(f) = T(g)$.
	Moreover, if $Y = \SS$, the choice $δ = π$ works.
	
	\begin{proof}
		This is trivial if $T(Y) = 0$.
		Otherwise, by (L2) there is an $\S$-sequence $\tuple{Y_*, v_*}$ with limit $\tuple{Y, v_{*, ∞}}$ such that $\absdeg{v_{0, ∞}} ∈ T(Y)$.
		We take $δ > 0$ such that $v_{0, ∞}$ is $\tuple{π, δ}$-continuous.
		If $Y = \SS$, we can have $v_{0, ∞} = \id_\SS$ and $δ = π$.
		Now let $f, g\maps X \to Y$ be $σ\S$-maps such that $f ≈_δ g$.
		We have $v_{0, ∞} ∘ f ≈_{π - ε} v_{0, ∞} ∘ g$ for some $ε ∈ (0, π)$.
		Let $\tuple{X_*, u_*}$ be an $\S$-sequence with limit $\tuple{X, u_{*, ∞}}$.
		By (F1) there is an $\S$-map $f_0\maps X_n \to Y_0$ such that $f_0 ∘ u_{n, ∞} ≈_ε v_{0, ∞} ∘ f ≈_{π - ε} v_{0, ∞} ∘ g$.
		Hence, $f_0$ is a witness for both $T(f)$ and $T(g)$.
	\end{proof}
\end{observation}

\begin{remark} \label{thm:degen_types}
	The formalism of types allows to comfortably express several known results.
	For example,
	\begin{enumerate}
		\item a circle-like continuum $X$ is also arc-like if and only if $T(X) = 0$,
		\item a circle-like continuum $X$ can be embedded into the plane if and only if $T(X) ∈ \set{0, 1}$.
	\end{enumerate}
	The fact that $T(X) = 0$ implies arc-like is proved for example in \cite[Theorem~2]{Ingram67}.
	We can also view this directly by observing that a continuous surjection $f\maps \SS \to \SS$ has degree $0$ if and only if it factorizes through $\II$.
	This is easily seen by considering a lifting $\tilde{f}\maps ℝ \to ℝ$, which is periodic if and only if $\deg(f) = 0$.
	Similarly, if $\tuple{X_*, u_*}$ is an $\S$-sequence with limit $\tuple{X, u_{*, ∞}}$, and $\tuple{Y_*, v_*}$ is an $\I$-sequence with limit $\tuple{X, v_{*, ∞}}$, by (F1) and (F2) of $\tuple{\CPolS, \MCptS}$, for every $m ∈ ω$ there are continuous surjections $g\maps Y_n \to X_m$ and $h\maps X_{m'} \to Y_n$ such that $g ∘ h ≈_π u_{m, m'}$, so $\deg(u_{m, m'}) = \deg(f ∘ g) = 0$.
	
	Claim~(ii) follows from the results of Bing~\cite{Bing62}, who however uses the formalism of refining covers instead of inverse limits.
	Every arc-like continuum is embeddable into the plane by \cite[Theorem~5]{Bing62}, see also \cite[Theorem~12.20]{Nadler92}.
	Every circle-like continuum of type $1$ is embeddable into the plane by \cite[Theorem~4]{Bing62}.
	It follows from \cite[Theorem~3]{Bing62} that a circle-like continuum of non-zero type in the plane has type $1$.
	See also \cite[4.1]{Krasinkiewicz76}.
\end{remark}

Now we can characterize $σ\S_P$ in terms of $T$.

\begin{proposition} \label{thm:sigmaSP_by_T}
	Let $P ⊆ Π$.
	\begin{enumerate}
		\item A circle-like continuum $X$ is a $σ\S_P$-object if and only if $T(X) ≤ [P^∞]$.
		\item A continuous surjection $f\maps X \to Y$ from a $σ\S$-object to a $σ\S_P$-object is a $σ\S_P$-map if and only if $T(f)$ is the multiplication by $t ≤ P^∞$, i.e. $T(f)$ is allowed to increase only coordinates from $P$.
	\end{enumerate}
	Together, $σ\S_P = T\preim{\T_P}$ where $\T_P ⊆ \T$ is the subcategory consisting of types $≤ [P^∞]$ and multiplications by numbers $≤ P^∞$.
	
	\begin{proof}
		(i) is clear since $T(X) ≤ S_P$ if and only if $X$ is a limit of some $\S$-sequence $\tuple{X_*, u_*}$ such that $\absdeg{u_{n_0, ∞}} ≤ P^∞$ for some $n_0 ∈ ω$.
		
		For (ii) let $\tuple{X_*, u_*}$ and $\tuple{Y_*, v_*}$ be $\S$-sequences with $σ\S$-limits $\tuple{X, u_{*, ∞}}$ and $\tuple{Y, v_{*, ∞}}$, and let $n ∈ ω$.
		Since $Y$ is an $σ\S_P$-object, we may assume that $v_*$ is an $\S_P$-sequence, and so $v_{n, ∞}$ is a $σ\S_P$-map and $\absdeg{v_{n, ∞}} ≤ P^∞$.
		For every $ε ∈ (0, π)$ there is an $\S$-map $f_0\maps X_m \to Y_n$ such that $f_0 ∘ u_{m, ∞} ≈_ε v_{n, ∞} ∘ f$.
		Let $t ∈ \SN$ be such $T(f)$ is the multiplication by $t$.
		We have $t ⋅ \absdeg{v_{n_0, ∞}} = \absdeg{f_0 ∘ u_{m_0, ∞}}$.
		
		Now if $f ∈ σ\S_P$, we may have chosen $u_*$ that is an $\S_P$-sequence and $f_0$ that is an $\S_P$-map, so $t ≤ \absdeg{f_0 ∘ u_{m, ∞}} ≤ P^∞$.
		On the other hand, if $t ≤ P^∞$, then $\absdeg{f_0 ∘ u_{m, ∞}} = t ⋅ \absdeg{v_{n, ∞}} ≤ P^∞$, and so $f_0$ is an $\S_P$-map and $u_*$ is an $\S_P$-sequence after $m$.
		Hence, $v_{n, ∞} ∘ f ≈_ε f_0 ∘ u_{m, ∞} ∈ σ\S_P$.
		Since $ε > 0$ was arbitrarily small, $v_{n, ∞} ∘ f ∈ σ\S_P$ by the definition of $σ$-closure.
		Similarly, since $n$ was arbitrary, $f ∈ σ\S_P$.
	\end{proof}
\end{proposition}

Let us summarize properties of $σ\S_P$-maps in the following corollary.

\begin{corollary}
	Let $P ⊆ Π$.
	\begin{enumerate}
		\item Every $σ\S_P(X, Y) ⊆ σ\S(X, Y)$ is clopen. More precisely, for every $σ\S$-object $Y$ there is $δ > 0$ such that for every $σ\S$-object $X$, $σ\S_P(X, Y)$ is $δ$-separated from its complement.
		\item We have $g ∘ f ∈ σ\S_P$ if and only if $g ∈ σ\S_P$ and $f ∈ σ\S_P$ for every pair of $σ\S$-maps $f\maps X \to Y$ and $g\maps Y \to Z$ between $σ\S$-objects.
		\item A $σ\S$-map $f\maps X \to Y$ between $σ\S_P$-objects is a $σ\S_P$-map if and only if $g ∘ f$ is a $σ\S_P$-map for some $σ\S_P$-map $g\maps Y \to \SS$.
		\item For every pair of $σ\S$-maps $f, g\maps X \to \SS$ such that $f ≈_π g$ we have $f ∈ σ\S_P$ if and only if $g ∈ σ\S_P$.
		\item For every $σ\S$-map $f\maps X \to Y$ between $σ\S_P$-objects, every pair of $\S_P$-sequences $\tuple{X_*, u_*}$, $\tuple{Y_*, v_*}$ with limits $\tuple{X, u_{*, ∞}}$, $\tuple{Y, v_{*, ∞}}$, and every $\S$-map $f_0\maps X_{m_0} \to Y_{n_0}$ such that $f_0 ∘ u_{m_0, ∞} ≈_π v_{n_0, ∞} ∘ f$ we have $f ∈ σ\S_P$ if and only if $f_0 ∈ \S_P$.
		\item For every $σ\S_P$-object $X$ there is $δ > 0$ such that every $δ$-monic $σ\S$-map $f\maps X \to \SS$ is a $σ\S_P$-map.
	\end{enumerate}
	
	\begin{proof}
		We will use Proposition~\ref{thm:sigmaSP_by_T} without noting.
		Claims~(i) and (iv) follow from the continuity of $T$ (Observation~\ref{thm:T_MU-continuous}).
		Claim~(ii) follows since if $T(g)$ and $T(f)$ are the multiplications by $s$ and $t$, we have $s ⋅ t ≤ P^∞$ if and only if both $s, t ≤ P^∞$.
		Claim~(iii) follows from (ii).
		Claim~(v) follows since $T(f)$ maps $\absdeg{v_{n_0, ∞}} ∈ T(Y)$ to $\absdeg{f_0 ∘ u_{m_0, ∞}}$ and $\absdeg{u_{m_0, ∞}} ≤ P^∞$.
		To prove (vi) let $\tuple{X_*, u_*}$ be a $\S_P$-sequence with a $σ\S_P$-limit $\tuple{X, u_{*, ∞}}$.
		By Proposition~\ref{thm:polyhedral_factorization} there is $δ > 0$ such that for every $δ$-monic $σ\S$-map $f\maps X \to \SS$ there is an $\S$-map $g$ such that $u_{0, ∞} ≈_π g ∘ f$.
		Since $u_{0, ∞} ∈ σ\S_P$, we have $g ∘ f ∈ σ\S_P$ by (iv) and $f ∈ σ\S_P$ by (ii).
	\end{proof}
\end{corollary}

So to show that a $σ\S$-map $f\maps X \to Y$ between $σ\S_P$-objects is a $σ\S_P$-map, we may either use (v) directly, or use (iii) to reduce the problem to the case $Y = \SS$.
A $σ\S$-map $f\maps X \to \SS$ is a $σ\S_P$-map if it is $π$-close to a map of the form $u_{n_0, ∞}$ for a sequence $u_*$ with limit $X$ such that $\absdeg{u_{n_0, ∞}} ≤ P^∞$ or $π$-close to a map of the form $h ∘ g$ where $g\maps X \to \SS$ is a $δ$-monic $σ\S$-map for sufficiently small $δ > 0$ and $h\maps \SS \to \SS$ is an $\S_P$-map.

This description is somewhat similar to the definition of maps of \emph{positive rank} in the discrete setting of Irwin~\cite[Definition~4.5]{Irwin07}.
See also the related definition of \emph{positive rank} of a continuous surjection $f\maps X \to Y$ between circle-like continua~\cite[Definition~4.8]{Irwin07}.

\subsection{Pseudo-solenoids as Fraïssé limits} \label{sec:pseudo-solenoid}

For every $\T$-object $S$ there is a unique (up to a homeomorphism) hereditarily indecomposable circle-like continuum $X$ of type $S$.
We shall call it the \emph{$S$-adic pseudo-solenoid}.
For $P ⊆ Π$ the $[P^∞]$-adic pseudo-solenoid is called just \emph{$P$-adic pseudo-solenoid}, and for $P = \set{p}$ just \emph{$p$-adic pseudo-solenoid}.
The $[1]$-adic pseudo-solenoid is also called the \emph{pseudo-circle}.
(But note that Rogers~\cite{Rogers70} calls every pseudo-solenoid of non-zero type a pseudo-circle.)
The $Π$-adic pseudo-solenoid is also called the \emph{universal pseudo-solenoid}.

For $S = 0$, the $S$-adic pseudo-solenoid is the pseudo-arc since by Remark~\ref{thm:degen_types} it is also arc-like, and Bing's theorem~\cite{Bing51} and Theorem~\ref{thm:pseudo-arc} apply.
The pseudo-circle was constructed by Bing~\cite{Bing51} and its uniqueness was proved by Fearnley~\cite[Theorem~6.3]{Fearnley70}.
Later, Fearnley~\cite[Theorem~3.2]{Fearnley72} extended the previous results by showing that two pseudo-solenoids are homeomorphic if and only if they have the same type.
As Irwin~\cite[page~9]{Irwin07} notes, Fearnley's definition of equivalent fundamental sequences needs to be fixed by allowing to cut their finite initial segments.
The same classification based on sequences of primes up to an equivalence applies to \emph{solenoids} as well as proved earlier by McCord~\cite[page~198]{McCord65}.

Note that the existence of the $S$-adic pseudo-solenoid for $S ∈ \SN/{\sim}$ follows from Theorem~\ref{thm:crooked_limit} since it is enough to build a crooked sequence of type $S$.
There exists an $ε$-crooked $\S$-map of degree $1$ for every $ε > 0$ by a construction similar to Construction~\ref{con:crooked}.
By composing with the map $z \mapsto z^n$ (or a tent map for $n = 0$) we get an $ε$-crooked $\S$-map of degree $n$.
Then we proceed as in Construction~\ref{thm:crooked_sequence_exists} to obtain a crooked sequence.

We finally show the following.

\begin{theorem} \label{thm:fraisse_pseudo-solenoid}
	The Fraïssé limit $ℙ_P$ of $\S_P$ in $σ\S_P$ is the $P$-adic pseudo-solenoid for every $P ⊆ Π$.
	
	\begin{proof}
		Since $ℙ_P$ is a $σ\S_P$-object, we have $T(ℙ_P) ≤ [P^∞]$ by Proposition~\ref{thm:sigmaSP_by_T}.
		But since every $\S_P$-sequence $\tuple{X_*, u_*}$ with limit $\tuple{ℙ_P, u_{*, ∞}}$ is Fraïssé, for every $p ∈ P$ and $n ∈ ω$ there are $\S_P$-maps $f\maps X_n \from \SS$ and $g\maps \SS \from X_m$ such that $\deg(f) = p$ and $f ∘ g ≈_π u_{n, m}$.
		Hence, $p$ divides $u_{n, m}$ and $T(ℙ_P) = [P^∞]$.
		Similarly, for every $n ∈ ω$ and $ε > 0$ there is an $ε$-crooked $\S_P$-map $f\maps X_n \from \SS$ and an $\S_P$-map $g\maps \SS \from X_m$ such that $f ∘ g ≈_ε u_{n, m}$, and so $u_{n, m}$ is $3ε$-crooked.
		Hence $\tuple{X_*, u_*}$ is a crooked sequence, $ℙ_P$ is a hereditarily indecomposable circle-like continuum of type $[P^∞]$, and the classification applies.
	\end{proof}
\end{theorem}

\begin{remark}
	From the previous theorem it follows that the pseudo-circle is homogeneous in $\tuple{\S_∅, σ\S_∅}$, i.e. for all continuous surjections $f, g\maps ℙ_∅ \to \SS$ of degree $1$ and for every $ε > 0$ there is a homeomorphism $h\maps ℙ_∅ \to ℙ_∅$ such that $f ≈_ε g ∘ h$.
	This is in contrast with the results of Boroński and Smith~\cite[Theorems~1.7 and 1.9]{BS18} showing that the pseudo-circle is not homogeneous in $\tuple{\S, σ\S}$ in a strong sense.
	The reason are the degrees.
	Also existence of a homogeneous object in $\tuple{\S, σ\S}$ would imply the amalgamation property of $\S$, contradicting Proposition~\ref{thm:circle_no_amalgamation}.
	
	Similarly, it follows that for every $P ⊆ Π$, every $σ\S_P$-map $f\maps ℙ_P \to ℙ_P$ is a near-homeomorphism, and so $T(f) = \id_{[P^∞]}$ (in fact this is true for every $σ\S_P$-map $f\maps X \to Y$ with $T(X) = T(Y) = [P^∞]$ by Proposition~\ref{thm:sigmaSP_by_T}), but for $d ∈ Π \setminus P$ there is a $σ\S$-map $g\maps ℙ_P \to ℙ_P$ that is $d$-covering~\cite{BS14}, and so $T(g)$ is the multiplication by $d$, and $g$ is not a near-homeomorphism.
\end{remark}

As with the pseudo-arc and Bing's theorem (Remark~\ref{thm:Bing_reproved}), it is possible to view Fearnley's classification of pseudo-solenoids as coming from uniqueness of a generic object.
In place of the crookedness factorization theorem (\ref{thm:crooked_factorization}) we use the following version of a theorem by Kawamura~\cite[Theorem~7]{Kawamura89}.

\begin{theorem} \label{thm:circular_crookedness_factorization}
	For an $\S_Π$-sequence $\tuple{X_*, f_*}$, the following conditions are equivalent.
	\begin{enumerate}
		\item $\tuple{X_*, f_*}$ is a crooked sequence.
		\item For every $ε > 0$ and $\S_Π$-map $g\maps X_n \from \SS$ such that $\deg(g)$ divides $\deg(f_{n, n'})$ for some $n' ≥ n$ (equivalently, $\absdeg{g} ≤ \absdeg{f_{n, ∞}}$) there is an $\S_Π$-map $h\maps \SS \from X_{n''}$ for some $n'' ≥ n'$ such that $g ∘ h ≈_ε f_{n, n'}$.
			In other words, $\tuple{X_*, f_*}$ is absorbing in $\S_Π$, but only for maps of relatively bounded degree.
	\end{enumerate}
	
	\begin{proof}
		As we have already seen, in order for the sequence to be crooked, it is enough to absorb all $\S$-maps of degree $1$.
		The other implication follows from \cite[Theorem~7]{Kawamura89}, where we replace the limit being hereditarily indecomposable by the sequence being crooked (via Theorem~\ref{thm:crooked_limit}) and we replace the condition AEOP (which corresponds to projectivity) by its weaker version corresponding to the absorption property.
	\end{proof}
\end{theorem}

\begin{corollary} \label{thm:unique_pseudosolenoid}
	Any two crooked $\S$-sequences $\tuple{X_*, f_*}$, $\tuple{Y_*, g_*}$ of the same type $S ∈ \SN/{\sim}$ have homeomorphic limits, and hence there is a unique $S$-adic pseudo-solenoid.
	
	\begin{proof}
		For $S = 0$, the limit is arc-like and hence the pseudo-arc.
		Otherwise, we can cut initial parts of the sequences and assume that $\absdeg{f_{0, ∞}} = \absdeg{g_{0, ∞}} ≠ 0$.
		With the use of the previous theorem we construct a back and forth sequence as in the proof of Proposition~\ref{thm:dominating_sequence}.
		Let $0 < ε_0 ≤ π$, let $n_0 = m_0 = 0$, and let $h_0\maps Y_{n_0} \from X_{m_0}$ be the identity.
		Suppose we have defined $0 < ε_k ≤ π$ and $h_{2k}\maps Y_{n_k} \from X_{m_k}$ such that $\absdeg{h_{2k} ∘ f_{m_k, ∞}} = \absdeg{g_{n_k, ∞}}$.
		Since $\absdeg{h_{2k}} ≤ \absdeg{g_{n_k, ∞}}$, by the previous theorem there is $n_{k + 1} > n_k$ and an $\S$-map $h_{2k + 1}\maps X_{m_k} \from Y_{n_{k + 1}}$ such that $h_{2k} ∘ h_{2k + 1} ≈_{ε_k/2} g_{n_k, n_{k + 1}}$.
		Since $ε_k ≤ π$, we have $\deg(h_{2k} ∘ h_{2k + 1}) = \deg(g_{n_k, n_{k + 1}})$, and so $\absdeg{h_{2k} ∘ h_{2k + 1} ∘ g_{n_{k + 1}, ∞}} = \absdeg{g_{n_k, ∞}} = \absdeg{h_{2k} ∘ f_{m_k, ∞}}$ and $\absdeg{h_{2k + 1} ∘ g_{n_{k + 1}, ∞}} = \absdeg{f_{m_k, ∞}}$.
		There is $0 < δ_k ≤ π$ such that $f_{n_j, n_k}$ is $\tuple{δ_j/2^{k - j}, δ_k}$-continuous for every $j < k$ and $h_{2k}$ is $\tuple{ε_k/2, δ_k}$-continuous.
		This finishes a half-step of the induction.
		The other half-step is analogous: there is an $\S$-map $h_{2k + 2}\maps Y_{n_{k + 1}} \from X_{m_{k + 1}}$ such that $h_{2k + 1} ∘ h_{2k + 2} ≈_{δ_k/2} f_{m_k, m_{k + 1}}$ and so $\absdeg{h_{2k + 2} ∘ f_{m_{k + 1}, ∞}} = \absdeg{g_{n_{k + 1}, ∞}}$, and there is a suitable $ε_{k + 1} > 0$.
		In the end we use Corollary~\ref{thm:back_and_forth}.
	\end{proof}
\end{corollary}

The $S$-adic pseudo-solenoid for $S = [0]$ or $[P^∞]$ for $P ⊆ Π$ is a generic object over $\S$ or $\S_P$, respectively.
For $S = [0]$ this is the pseudo-arc; for $S = [P^∞]$ the $P$-adic pseudo-solenoid is even a Fraïssé limit.
For the other types $[s] ∈ \SN/{\sim}$ (where $0 < s(p) < ∞$ for infinitely many $p ∈ Π$) there is no corresponding subcategory of $\S$.
However, a certain modification of the Banach–Mazur game still allows us to view every $S$-adic pseudo-solenoid as a generic object.

\begin{definition} \label{def:modified_Banach--Mazur}
	For a type $S ∈ \SN/{\sim}$ we define the abstract \emph{Banach–Mazur game in $\S$ below $S$}.
	$\BM(\tuple{\S, S})$ denotes the play scheme $\BM(\S)$ with the following extra rule.
	Every play $f_*$ of the game is supposed to be a sequence of type $≤ S$.
	If it is not the case, the player responsible the failure loses.
	More precisely, for $S = [0]$ the extra rule is vacuous, while for a positive type $S$ we take $s ∈ S$ and for every $p ∈ Π$ we look at the maximal $n ∈ ω$ (if it exists) such that $\absdeg{f_{0, n}}(p) ≤ s(p)$.
	We blame Eve or Odd, respectively, for the failure at $p$ depending on whether $n$ is even or odd.
	The play sequence $f_*$ fails to have type $≤ S$ if and only if there have been failures at infinitely many primes $p$.
	In that case, if Eve is responsible for infinitely many of them, she loses. Otherwise, Odd loses.
	Note that the responsibility for a failure at particular $p$ depends on $s ∈ S$, but the total responsibility does not.
	This specifies the game scheme.
	A fully specified game $\BM(\tuple{\S, S}, \G)$ also adds the goal class $\G$ for Odd as in Definition~\ref{def:BM}, which determines who wins in the case that the type of the play sequence stays below $S$.
\end{definition}

\begin{observation}
	The game $\BM(\tuple{\S, [0]})$ is just $\BM(\S)$, while for $P ⊆ Π$ the games $\BM(\tuple{\S, [P^∞]})$ and $\BM(\S_P)$ are equivalent in the sense that there is a translation preserving winning strategies for both players both ways.
	The translation is as follows.
	We build a single $\S$-sequence interpreted both as a play in $\BM(\tuple{\S, [P^∞]})$ and $\BM(\S_P)$.
	As long as the moves are in $\S_P$ there is no difference, but when a player playing $\BM(\tuple{\S, [P^∞]})$ plays a map not in $\S_P$, this cannot be interpreted as a move in $\BM(\S_P)$.
	In this case we “reset” the corresponding play of $\BM(\S_P)$ – we pretend that the play of $\BM(\S_P)$ is just about to start, i.e. we interpret only the tail of the sequence $f_*$ as the play (in the case that the next player is Odd, we pretend Eve started with $\id_\SS$).
	If this reset occurs only finitely many times, we have a tail of the sequence corresponding to a play in both games and the full sequence corresponding to a play in $\BM(\tuple{\S, [P^∞]})$, all three with the same outcome.
	If the reset occurs infinitely many times, there is no corresponding play of $\BM(\S_P)$ and the type of $f_*$ is not below $[P^∞]$ so one of the players is to blame for the failure.
	But neither a winning strategy in $\BM(\tuple{\S, [P^∞]})$ nor any strategy in $\BM(\S_P)$ can be responsible for such failure, and translated winning strategies remain winning in the other game.
\end{observation}

\begin{proposition} \label{thm:pseudo-solenoids_generic}
	For every type $S ∈ \SN/{\sim}$, Odd has a winning strategy in $\BM(\tuple{\S, S})$ for the $S$-adic pseudo-solenoid.
	
	\begin{proof}
		By Corollary~\ref{thm:unique_pseudosolenoid} it is enough for Odd to have a winning strategy in $\BM(\tuple{\S, S}, \G)$ where $\G$ is the set of all crooked $\S$-sequences of type $S$.
		But it is easy for Odd to force the crookedness and the right type (unless Eve is to blame for the failure of the type condition).
		The case $S = 0$ corresponds to $\BM(\S)$, so suppose $S ≠ 0$.
		Let $s ∈ S$ and let $\tuple{p_k}_{k ∈ ω}$ be an enumeration of all primes such that each appears infinitely many times.
		At every step $2k + 1 = n ∈ ω$ Odd plays an $ε_n$-crooked $\S$-map of degree $1$ for suitable $ε_n > 0$ as in Proposition~\ref{thm:crooked_sequence_generic}, composed with the map $z \mapsto z^{p_k}$ if $\absdeg{f_{0, n}}(p_k) < s(p_k)$.
	\end{proof}
\end{proposition}

In the final paragraphs let us summarize known results regarding existence of continuous surjections between circle-like continua and in particular pseudo-solenoids, using types and Fraïssé theory.

\begin{proposition}
	Let $f\maps X \to Y$ be a continuous surjection between circle-like continua and let $P, Q ⊆ Π$.
	\begin{enumerate}
		\item If $T(f) ≠ 0$, then $T(X) ≥ T(Y)$, and $T(f)$ is the multiplication by $t$ with $[t] ≤ T(X)$.
		\item There is a notion of a \emph{self-entwined} circle-like continuum, introduced by Rogers~\cite{Rogers70_entwined}.
			If $Y$ is self-entwined, then $T(f) ≠ 0$.
			If $T(Y) ∉ \set{0, 1}$, then $Y$ is self-entwined.
			The pseudo-circle $ℙ_∅$ is self-entwined.
		\item There is a continuous surjection $f\maps ℙ_P \to ℙ_Q$ if and only if $P ⊇ Q$.
			In this case, $T(f) ≠ 0$.
		\item There is a continuous surjection $f\maps ℙ_P \to ℙ$ (necessarily with $T(f) = 0$), but no continuous surjection $ℙ \to ℙ_P$ for every $P ⊆ Π$.
	\end{enumerate}
	
	\begin{proof}
		Claim~(i) holds since for $s ∈ T(Y)$ we have $T(f)(s) ∈ T(X)$ and $T(f)(s) ≥ s$.
		Claim~(ii) follows from the results by Rogers~\cite[Theorem~3, 4, 5]{Rogers70_entwined}, and from the observation that Roger's \emph{limit degree zero} corresponds to our $T(f) = 0$.
		
		If $f\maps ℙ_P \to ℙ_Q$ is a continuous surjection, then $T(f) ≠ 0$ by (ii), and $[P^∞] = T(ℙ_P) ≥ T(ℙ_Q) = [Q^∞]$ by (i).
		On the other hand, if $P ⊇ Q$, then $ℙ_Q$ is a $σ\S_P$-object and $ℙ_P$ is cofinal in $σ\S_P$.
		There is no continuous surjection $f\maps ℙ \to ℙ_P$ since by (ii) we would have $T(f) ≠ 0$, which is impossible since $T(ℙ) = 0$.
		
		As in \cite[Theorem~12]{Rogers70} we may consider two copies of an $\I$-sequence $\tuple{X_*, f_*}$ with limit $ℙ$ such that $f_n(i) = i$ for $i ∈ \set{0, 1}$ and every $n ∈ ω$, and glue them at the end-points $0$, $1$.
		We obtain an $\S$-sequence $\tuple{Y_*, g_*}$ of maps of degree $1$ with a limit $Y_∞$, and a sequence of continuous surjections $φ_n\maps Y_n \to X_n$, $n ∈ ω$, commuting with $f_*$ and $g_*$ such that every $φ_n$ is the quotient $\SS \to \II$ gluing two halves of the circle together.
		In the and we obtain a continuous surjection $φ_∞\maps Y_∞ \to ℙ$, where $T(Y_∞) = 1$, so $Y_∞$ is a continuous image of every $ℙ_P$.
	\end{proof}
\end{proposition}

\begin{corollary} \label{thm:universal_pseudo-solenoid}
	The universal pseudo-solenoid $ℙ_Π$ admits a continuous surjection onto every circle-like or arc-like continuum.
	Every continuous surjection $ℙ_Π \to ℙ_Π$ is a near-homeomorphism.
	Also, for every pair of continuous surjections $f, g\maps ℙ_Π \to Y$ onto a non-planar circle-like continuum $Y$ and every $ε > 0$ there is a homeomorphism $h\maps ℙ_Π \to ℙ_Π$ such that $f ≈_ε g ∘ h$.
	
	\begin{proof}
		We combine the fact that $ℙ_Π$ is the Fraïssé limit of $\S_Π$ in $σ\S_Π$ with the previous proposition.
		$ℙ_Π$ continuously maps onto every circle-like continuum of positive type, but by (iv) also onto the pseudo-arc, which maps onto every arc-like continuum.
		Every continuous surjection $ℙ_Π \to ℙ_Π$ has non-zero type by (iii), and so is a near-homeomorphism by Observation~\ref{thm:near_automorphism}.
		Similarly, for the continuous surjection $f, g\maps ℙ_Π \to Y$, if $Y$ is non-planar, i.e. $T(Y) ∉ \set{0, 1}$ (Remark~\ref{thm:degen_types}), we know by (ii) that $f$ and $g$ are $σ\S_Π$-maps, and we may use the homogeneity.
	\end{proof}
\end{corollary}

Note that (using our language) the homogeneity in $\tuple{\S_Π, σ\S_Π}$ and the cofinality in $σ\S_Π$ of the universal pseudo-solenoid was already proved by Irwin~\cite[Theorem~4.20 and 4.22]{Irwin07}.

	\appendix
	
\section{Appendix: More on \texorpdfstring{$σ$}{σ}-closure and \texorpdfstring{$σ$}{σ}-consistency}
\label{sec:appendix}

In the appendix we give a general construction of the $σ$-closure (Definition~\ref{def:sigma_closure}) in detail, we relate it to some other natural definitions (versions of which appear in literature), and show that if our category of small objects is $σ$-consistent (Definition~\ref{def:sigma_consistent}), all the definitions agree.
Furthermore, we give a concrete example of failure of Fraïssé theory when we do not have $σ$-consistency.

\medskip

	Let us fix MU-categories $\K ⊆ \L$.
	The $σ$-closure $σ\K$ is supposed to be the “relevant part” of $\L$ containing $\K$ such that $\tuple{\K, σ\K}$ has the best chance to be a free completion.
	Condition~(L2) says that every $\L$-object should be a limit of a $\K$-sequence, which may easily not be the case – some $\L$-objects may be simply “too far” from $\K$.
	A natural thing to do is to define $σ\K$ to be the category of all $\L$-limits of $\K$-sequences.
	This is clear with respect to the objects of $σ\K$, but not so with respect to the morphisms.
	Taking $σ\K$ to be a full subcategory of $\L$ may not be appropriate.
	Another issue is that the notion of a limit is not absolute in the following sense.
	
	Let us consider a category $\L'$ such that $\K ⊆ \L' ⊆ \L$.
	The ideal situation is that $\L'$-limits of $\K$-sequences are exactly their $\L$-limits.
	This really consists of three conditions:
	\begin{enumerate}[label={(\alph*)}]
		\item every $\L$-limit cone for a $\K$-sequence lies in $\L'$,
		\item every $\L$-limit cone for a $\K$-sequence $f_*$ that lies in $\L'$ is also an $\L'$-limit cone for $f_*$,
		\item every $\L'$-limit cone for a $\K$-sequence $f_*$ is also an $\L$-limit cone for $f_*$.
	\end{enumerate}
	Condition~(a) is easy to achieve – just add all the $\L$-limit cones to $\L'$.
	Condition~(b) says that the inclusion functor $\L' ⊆ \L$ reflects limits of $\K$-sequences, and will be important later – we shall say that $\L'$ is \emph{$σ$-reflecting} in $\L$ for $\K$ if it satisfies (b).
	Condition~(c), saying that the inclusion functor $\L' ⊆ \L$ preserves limits of $\K$-sequences, may not be achievable – a $\K$-sequence with a limit in $\L'$ may simply not have a limit in $\L$ (which is fixed).
	But if $\tuple{\K, \L}$ satisfies (L1) and we have (a) and (b), then we have also (c) – an $\L$-limit cone $f_{*, ∞}$ for a $\K$-sequence $f_*$ exists by (L1), lies in $\L'$ by (a), and is an $\L'$-limit of $f_*$ by (b), so the fixed $\L'$-limit cone for $f_*$ is $\L'$-isomorphic (and so $\L$-isomorphic) to $f_{*, ∞}$.

\begin{construction} \label{con:sigma_closure}
	To clearly describe the construction of $σ\K$ let us introduce several closure-like operators $\clx$ on families of maps $\F ⊆ \L$.
	The operators satisfy $\A ⊆ \clx(\A) ⊆ \clx(\B)$ for all families $\A ⊆ \B$.
	\begin{itemize}
		\item $\clcat(\F)$ is the subcategory of $\L$ generated by $\F$, i.e. we close the family $\F$ under identities and composition.
		Some of the other closures do not necessarily produce a subcategory, but we definitely want $σ\K$ to be a subcategory.
		This is why we explicitly consider this closure.
		
		\item $\clloc(\F)$ or simply $\clo{\F}$ denotes the local closure (see Definition~\ref{def:sigma_closure}), i.e. $\clo{\F} ∩ \L(X, Y)$ is the closure of $\F ∩ \L(X, Y)$ in the $∞$-metric space $\L(X, Y)$ for all $\L$-objects $X, Y$.
		In order to satisfy (L1), $σ\K$ has to be locally complete and so necessarily locally closed in $\L$.
		Since composition is continuous in any MU-category, we have $\clcat(\clloc(\F)) ⊆ \clloc(\clcat(\F))$, and so if $\F$ is a subcategory, so is its local closure.
		
		\item $\cllim(\F)$ consists of $\F$ together with all $\L$-maps $f_{n, ∞}$, $n ∈ ω$, where $f_{*, ∞}$ is an $\L$-limit of a $(\K ∩ \F)$-sequence $f_*$ (meaning $f_{n', n''} ∈ \K ∩ \F$ for every $n' ≤ n'' ∈ ω$).
		That means, $\L'$ satisfies (a) if and only if $\cllim(\L') ⊆ \L'$ if and only if $\cllim(\K) ⊆ \L'$.
		The last part is true since for every family $\F ⊇ \K$ we have $\cllim(\F) = \F ∪ \cllim(\K)$.
		
		\item $\clabs(\F)$ consists of $\F$ together with all $\L$-maps $h$ such that there is a $(\K ∩ \F)$-sequence $f_*$ and an $\L$-limit $f_{*, ∞}$ of $f_*$ lying in $\F$ such that $f_{n, ∞} ∘ h ∈ \F$ for every $n ∈ ω$.
		This closure addresses condition~(b) – it may happen that an $\L$-limit cone $\tuple{X_∞, f_{*, ∞}}$ for a $\K$-sequence $f_*$ lies in $\L'$, but is not an $\L'$-limit, which is witnessed by an $\L'$-cone $\tuple{Y, γ_*}$ for $f_*$ such that the unique factorizing $\L$-map $γ_∞\maps Y \to X_∞$ is not in $\L'$.
		The closure $\clabs(\F)$ collects all limit factorizing maps relevant to $\F$.
		So $\L'$ is $σ$-reflecting in $\L$ for $\K$ if and only if $\clabs(\L') ⊆ \L'$.
	\end{itemize}
	
	It seems natural to define $σ\K$ simply as $\A_0 := \clloc(\clcat(\cllim(\K)))$, i.e. the smallest locally closed subcategory of $\L$ containing $\K$ and all $\L$-limits of $\K$-sequences.
	But $\A_0$ may not be $σ$-reflecting in $\L$ for $\K$, i.e. even though we have added all $\L$-limits of $\K$-sequences, they may not be limits from the point of view of $\A_0$.
	We should really define $σ\K$ as the smallest locally closed and $σ$-reflecting for $\K$ subcategory of $\L$ containing $\K$ and all $\L$-limits of $\K$-sequences.
	Is this well-defined?
	If the category $\L$ is small, then yes – we can simply take the intersection of all such subcategories, $\L$ being one of them.
	But if $\L$ is a large category, we should be more careful.
	
	Note that it is not enough in general to consider just $\clabs(\A_0)$.
	The closure $\clabs$ may not be idempotent since adding the limit factorizing maps for the relevant cones may introduce new relevant cones.
	Also, $\clabs(\A_0)$ may not be locally closed or even a subcategory.
	We perform the following closing-off construction.
	For every ordinal $α$ let us put $\A_{α + 1} := \clloc(\clcat(\clabs(\A_α)))$ and for every limit ordinal $α$ let $\A_α := ⋃_{β < α} \A_β$.
	Note that all families $\A_α$ are subcategories of $\L$ with the same class of objects, and that $\A_1$ contains all $\L$-isomorphisms between its objects.
	We show that the increasing sequence stabilize after at most $ω_1$ many steps, and so $σ\K := \A_{ω_1}$ is the smallest category between $\K$ and $\L$ closed under all four closures $\clx$.
	The reason is that for every $f ∈ \clx(\A_{ω_1})$ we have $f ∈ \clx(\F)$ for a countable family $\F ⊆ \A_{ω_1}$.
	Hence, $\F ⊆ \A_α$ for some $α < ω_1$, and so $f ∈ \A_{α + 1} ⊆ \A_{ω_1}$.
	
	Note that formally there may be a problem defining a transfinite sequence of proper classes, but we may instead define a single class relation $R(f, α)$ by “$f ∈ \A_0$” for $α = 0$, by “$∃β < α, R(f, β)$” for $α$ limit, and by “there exists a countable family $\F$ such that $f ∈ \clloc(\clcat(\clabs(\F)))$ and $∀g ∈ \F, R(g, β)$” if $α = β + 1$.
	
	This concludes the construction of the $σ$-closure $σ\K$.
	If the supercategory $\L$ is not clear from the context, we write $σ_\L \K$.
	Note that since $σ\K$ is $σ$-reflecting in $\L$ for $\K$ and closed with respect to the closures $\clx$, the corresponding closures $\clx$ in $σ\K$ are equal to the original ones, and we have $σ_{σ\K} \K = σ\K$.
	Also note that $σ\K$-isomorphisms are exactly $\L$-isomorphisms between $σ\K$-objects.
\end{construction}

\begin{corollary} \label{thm:absolute}
	If $\K ⊆ \L$ are MU-categories satisfying (L1), then the limits of $\K$-sequences are the same from the point of view of $\L$ and $σ\K$, i.e. we have conditions (a), (b), (c) for $\L' = σ\K$.
	Moreover, we have $σ_{\L'} \K' = σ_\L \K'$ for every $\K' ⊆ \K$.
	
	\begin{proof}
		The $σ$-closure was constructed so that it satisfies (a) and (b).
		Since we also assume (L1), we have also (c) as shown above.
		The last part follows from the fact that conditions (a), (b), (c) are clearly true also for every $\K' ⊆ \K$.
	\end{proof}
\end{corollary}

We shall discuss several other possible definitions of the $σ$-closure which nevertheless work only if our category of small objects is $σ$-consistent.
First, we define several notions formalizing the concept of an $\L$-map being approximated from $\K$.
We already used these notions implicitly in Proposition~\ref{thm:cone_transfer}.

\begin{definition}
	Let $\K$ be an MU-category and let $\tuple{X_*, f_*}$ and $\tuple{Y_*, g_*}$ be $\K$-sequences.
	\begin{itemize}
		\item A \emph{pre-transformation} $φ_*\maps f_* \to g_*$ is any family $\tuple{φ_n\maps X_{φ(n)} \to Y_n}_{n ∈ ω}$ of $\K$-maps such that $\tuple{φ(n)}_{n ∈ ω}$ is an increasing sequence in $ω$.
		\item A \emph{transformation} is a pre-transformation $φ_*\maps f_* \to g_*$ such that for every $n ∈ ω$ and $ε > 0$ there is $n_0 ≥ n$ such that for every $n'' ≥ n' ≥ n_0$ we have $g_{n, n'} ∘ φ_{n'} ∘ f_{φ(n'), φ(n'')} ≈_ε g_{n, n''} ∘ φ_{n''}$.
		\item A pre-transformation $φ_*\maps f_* \to g_*$ is \emph{normalized} by an epsilon sequence $\tuple{ε_n}_{n ∈ ω}$ for $g_*$ if for every $n ∈ ω$ we have $φ_n ∘ f_{φ(n), φ(n + 1)} ≈_{ε_n} g_{n, n + 1} ∘ φ_{n + 1}$.
	\end{itemize}
	Let $\L ⊇ \K$ be an MU-category and let $\tuple{X_∞, f_{*, ∞}}$ and $\tuple{Y_∞, g_{*, ∞}}$ be $\L$-limits of $f_*$ and $g_*$, respectively.
	\begin{itemize}
		\item An $\L$-map $φ_∞\maps X_∞ \to Y_∞$ is the \emph{limit} of a pre-transformation $φ_*\maps f_* \to g_*$ if for every $n ∈ ω$ and $ε > 0$ there is $n_0 ≥ n$ such that for every $n' ≥ n_0$ we have $g_{n, n'} ∘ φ_{n'} ∘ f_{φ(n'), ∞} ≈_ε g_{n, ∞} ∘ φ_∞$, i.e. if $g_{n, ∞} ∘ φ_∞ = \lim_{n' ≥ n}(g_{n, n'} ∘ φ_{n'} ∘ f_{φ(n'), ∞})$ for every $n$.
	\end{itemize}
\end{definition}

\begin{remark}
	Proposition~\ref{thm:cone_transfer} shows that every normalized pre-transformation is a transformation and that every transformation has a limit as long as $\L$ is locally complete and the sequences $f_*$ and $g_*$ have limits.
	Also note that the limit of a pre-transformation is unique.
\end{remark}

\begin{proposition} \label{thm:transformations}
	Let $\K ⊆ \L$ be MU-categories and let $\tuple{X_*, f_*}$ and $\tuple{Y_*, g_*}$ be $\K$-sequences with $\L$-limits $\tuple{X_∞, f_{*, ∞}}$ and $\tuple{Y_∞, g_{*, ∞}}$, respectively.
	\begin{enumerate}
		\item An $\L$-map $h\maps X_∞ \to Y_∞$ is the limit of some pre-transformation $φ_*\maps f_* \to g_*$ if and only if for every $n ∈ ω$ and $ε > 0$ there exists $m ∈ ω$ and a $\K$-map $h'\maps X_m \to Y_n$ such that $h' ∘ f_{m, ∞} ≈_ε g_{n, ∞} ∘ h$.
		\item If $\tuple{\K, \L}$ satisfies (F2), then for every pre-transformation $φ_*\maps f_* \to g_*$ with a limit $φ_∞$ there is an increasing sequence $\tuple{ψ(n)}_{n ∈ ω}$ of natural numbers $ψ(n) ≥ φ(n)$ such that $ψ_* = \tuple{φ_n ∘ f_{φ(n), ψ(n)}}_{n ∈ ω}$ is a transformation.
		\item For every transformation $φ_*\maps f_* \to g_*$ and an epsilon sequence $\tuple{ε_n}_{n ∈ ω}$ for $g_*$ there is an increasing sequence $\tuple{α(n)}_{n ∈ ω}$ of natural numbers $α(n) ≥ n$ such that the transformation $ψ_* = \tuple{g_{n, α(n)} ∘ φ_{α(n)}}_{n ∈ ω}$ is normalized by $\tuple{ε_n}_{n ∈ ω}$.
	\end{enumerate}
	
	\begin{proof}
		To simplify the following formulas, let us denote the composition $g_{n, n'} ∘ φ_{n'} ∘ f_{φ(n'), m}$ for $n ≤ n'$ and $φ(n') ≤ m ≤ ∞$ simply by $φ_{n, n'}{}^m$.
		
		Claim~(i):
		If $h\maps X_∞ \to Y_∞$ is the limit of a pre-transformation $φ_*\maps f_* \to g_*$ and $n ∈ ω$ and $ε > 0$ are given, there is $n_0 ≥ n$ such that $φ_{n, n_0}{}^∞ ≈_ε g_{n, ∞} ∘ h$, so it is enough to put $h' = φ_{n, n_0}{}^{φ(n_0)}$.
		
		For the other implication we fix an epsilon sequence $\tuple{ε_n}_{n ∈ ω}$ for $g_*$ and for every $n ∈ ω$ we pick a $\K$-map $φ_n\maps X_{φ(n)} \to Y_n$ such that $φ_n ∘ f_{φ(n), ∞} ≈_{ε_n} g_{n, ∞} ∘ h$ and such that the sequence $\tuple{φ(n)}_{n ∈ ω}$ is increasing.
		Then for every $n ≤ n' ≤ n''$ we have 
		\[
			φ_{n, n'}{}^∞ ≈_{ε_n/2^{n' - n}} g_{n, ∞} ∘ h ≈_{ε_n/2^{n'' - n}} φ_{n, n''}{}^∞,
		\]
		which is enough.
		
		Claim~(iii):
		By the assumption, for every $n ∈ ω$ there is $α(n) ≥ n$ such that for every $n' ≥ α(n)$ we have $φ_{n, α(n)}{}^{φ(n')} ≈_{ε_n} φ_{n, n'}{}^{φ(n')}$.
		We pick the numbers $α(n)$ inductively so that the sequence $\tuple{α(n)}_{n ∈ ω}$ is increasing, and we put $ψ_n = g_{n, α(n)} ∘ φ_{α(n)}$.
		Then for every $n ∈ ω$ we have
		\[
			ψ_n ∘ f_{ψ(n), ψ(n + 1)} = φ_{n, α(n)}{}^{φ(α(n + 1))} ≈_{ε_n} φ_{n, α(n + 1)}{}^{φ(α(n + 1))} = g_{n, n + 1} ∘ ψ_{n + 1}.
		\]
		
		Claim~(ii):
		First we fix an epsilon sequence $\tuple{ε_n}_{n ∈ ω}$ for $g_*$ and a sequence $\tuple{δ_n}_{n ∈ ω}$ such that $δ_n$ is an (F2) witness for $Y_n$ and $ε_n$.
		Similarly to the proof of (iii) we define an increasing sequence $\tuple{α(n)}_{n ∈ ω}$ such that $φ_{n, α(n)}{}^∞ ≈_{δ_n/2} g_{n, ∞} ∘ φ_∞ ≈_{δ_n/2} φ_{n, n'}{}^∞$ for every $n' ≥ α(n)$.
		Then for every $n$ we pick $β(n) ≥ φ(α(n + 1))$ such that for every $α(n) ≤ n' ≤ α(n + 1)$ we have $φ_{n, α(n)}{}^{β(n)} ≈_{ε_n} φ_{n, n'}{}^{β(n)}$.
		Here we use (F2) and $δ_n$ for finitely many $\K$-maps.
		Finally, we fix any increasing sequence $\tuple{ψ(n)}_{n ∈ ω}$ such that $ψ(n') ≥ β(n)$ whenever $α(n) ≤ n' ≤ α(n + 1)$.
		
		Then, given $n ∈ ω$ and $ε > 0$, for every $n ≤ n_0 ≤ α(n_0) ≤ α(n_1) ≤ n' ≤ α(n_1 + 1)$ we have
		\[
			φ_{n, α(n_0)}{}^{ψ(n')} ≈_{2ε_n/2^{n_0 - n}} φ_{n, α(n_1)}{}^{ψ(n')} ≈_{ε_n/2^{n_1 - n}} φ_{n, n'}{}^{ψ(n')},
		\]
		so choosing $n_0 ≥ n$ large enough works.
		The first estimate follows as in the proof of Proposition~\ref{thm:cone_transfer}~(i), and the second one follows from our choice of $β(n)$.
	\end{proof}
\end{proposition}

\begin{remark} \label{rem:sigma_consistency_back_forth}
	The previous proposition gives a converse to Proposition~\ref{thm:cone_transfer} in the sense that under (F1), every $\L$-map between limits of $\K$-sequences is the limit of a pre-transformation from $\K$, and if we additionally have (F2), the pre-transformation can be normalized.
	Similarly, one can show that $σ$-consistency is equivalent to every isomorphism being the limit of a back and forth sequence, which under (F2) can be normalized as in Corollary~\ref{thm:back_and_forth}.
\end{remark}

\begin{definition}
	Let $\K ⊆ \L$ be MU-categories.
	By $σ^∃\K$ and $σ^∀\K$ we denote the families of all $\L$-maps $h\maps X \to Y$ between $σ\K$-objects such that for \emph{some} or \emph{all} (respectively) $\K$-sequences $f_*$ and $g_*$ with $\L$-limits $X$ and $Y$ there is a pre-transformation $φ_*\maps f_* \to g_*$ with limit $h$.
\end{definition}

\begin{remark} \label{rem:concrete_limits}
	In the context of projective Fraïssé theory, Solecki~\cite[Appendix~A.2]{Solecki25} defines the $σ$-closure (there denoted by $\C^ω$) essentially as the family of all limits of pre-transformations without distinguishing between $σ^∀\K$ and $σ^∃\K$.
	This is done by working with concrete representatives of limits – the spaces of threads $X_∞ ⊆ ∏_{n ∈ ω} X_n$ with the projections.
	Formally this means that $\C^ω$ may not contain all isomorphisms between its objects.
\end{remark}

\begin{notation}
	Let $\L(\K, \K)$ and $\L(σ\K, \K)$ denote the families of all $\L$-maps between $\K$-objects and from $σ\K$-objects to $\K$-objects, respectively.
	In addition, recall the families $\A_0$ and $\A_1$ defined in Construction~\ref{con:sigma_closure}.
	Furthermore, for every category $\C$ let $\id(\C)$ denote the family of all identity morphisms.
\end{notation}

\begin{lemma} \label{thm:AE_sigma}
	In general, we have the following.
	\begin{enumerate}
		\item $σ^∀\K ∩ \L(\K, \K) ⊆ \clo{\K}$,
		\item $σ^∀\K ∩ \L(σ\K, \K) ⊆ \clo{\cllim(\K)} ⊆ \A_0$,
		\item $σ^∀\K$ is closed under composition and is locally closed,
		\item $\bigl(\clo{\K} ∪ \cllim(\K) ∪ \id(σ\K) ∪ \clabs(σ^∀\K)\bigr) ⊆ σ^∃\K ⊆ \clabs(\A_0) ⊆ \A_1 ⊆ σ\K$.
	\end{enumerate}
	
	\begin{proof}
		All the claims are straightforward.
		We will just sketch the proof.
		We will use the names for $\K$-sequences $\tuple{X_*, u_*}$, $\tuple{Y_*, v_*}$, $\tuple{Z_*, w_*}$ with $\L$-limits $\tuple{X, u_{*, ∞}}$, $\tuple{Y, v_{*, ∞}}$, $\tuple{Z, w_{*, ∞}}$ as needed,
		and we will use Proposition~\ref{thm:transformations}~(i) to characterize $σ^∀\K$-maps.
		
		To show (i) and (ii) we use the fact that if $X$ or $Y$ is a $\K$-object, then we can take the constant identity sequences $u_*$ or $v_*$, respectively.
		For a $σ^∀\K$-map $f\maps X \to Y$, $n ∈ ω$, and $ε > 0$ there is a $\K$-map $f'\maps X_m \to Y_n$ such that $v_{n, ∞} ∘ f ≈_ε f' ∘ u_{m, ∞}$.
		In (i) we obtain $f ≈_ε f' ∈ \K$ and so $f ∈ \clo{\K}$.
		In (ii) we obtain $f ≈_ε f' ∘ u_{m, ∞} ∈ \K ∘ \cllim(\K) ⊆ \cllim(\K)$ and so $f ∈ \clo{\cllim(\K)}$.
		
		For the first part of (iii) suppose that $f\maps X \to Y$ and $g\maps Y \to Z$ are $σ^∀\K$-maps, $n ∈ ω$, and $ε > 0$.
		There is a $\K$-map $g'\maps Y_m \to Z_n$ such that $g' ∘ v_{m, ∞} ≈_ε w_{n, ∞} ∘ g$ and $δ > 0$ such that $g'$ is $\tuple{ε, δ}$-continuous.
		Then there is a $\K$-map $f'\maps X_k \to Y_m$ such that $f' ∘ u_{k, ∞} ≈_δ v_{m, ∞} ∘ f$.
		Hence, $g' ∘ f' ∘ u_{k, ∞} ≈_ε g' ∘ v_{m, ∞} ∘ f ≈_ε w_{n, ∞} ∘ g ∘ f$.
		
		For the second part of (iii) suppose that $f\maps X \to Y$ is a $\clo{σ^∀\K}$-map, $n ∈ ω$, and $ε > 0$.
		There is $δ > 0$ such that $v_{n, ∞}$ is $\tuple{ε, δ}$-continuous, a $σ^∀\K$-map $g\maps X \to Y$ with $g ≈_δ f$, and a $\K$-map $g'\maps X_m \to Y_n$ with $g' ∘ u_{m, ∞} ≈_ε v_{n, ∞} ∘ g ≈_ε v_{n, ∞} ∘ f$.
		
		The inclusion $\id(σ\K) ⊆ σ^∃\K$ is clear as we can take $u_* = v_*$.
		The inclusion $\clo{\K} ⊆ σ^∃\K$ again follows from the fact that we can take the constant identity sequences for $u_*$ and $v_*$.
		Similarly we obtain $\cllim(\K) ⊆ σ^∃\K$: for $f\maps X \to Y$ being equal to $u_{m, ∞} ∈ \cllim(\K)$, we have $Y = X_m$, and we take the constant identity sequence for $v_*$.
		For any $n ∈ ω$ and $ε > 0$ we put $f' = u_{n, \max(m, n)}$.
		
		To obtain $\clabs(σ^∀\K) ⊆ σ^∃\K$ we first note that clearly $σ^∀\K ⊆ σ^∃\K$.
		Then we consider an $\L$-map $f\maps X \to Y$ such that $v_{n, ∞} ∘ f ∈ σ^∀\K$ for every $n$.
		Let $n ∈ ω$ and $ε > 0$ be given.
		Since $v_{n, ∞} ∘ f ∈ σ^∀\K$ and since we can take the constant identity sequence for $Y_n$, there is a $\K$-map $f'\maps X_m \to Y_n$ such that $f' ∘ u_{m, ∞} ≈_ε v_{n, ∞} ∘ f$, which witnesses that $f ∈ σ^∃\K$.
		
		To obtain $σ^∃\K ⊆ \clabs(\A_0)$ suppose that $f\maps X \to Y$ is a $σ^∃\K$-map as witnessed by the sequences $u_*$ and $v_*$.
		For every $n ∈ ω$ and $ε > 0$ there is a $\K$-map $f'\maps X_m \to Y_n$ with $v_{n, ∞} ∘ f ≈_ε f' ∘ u_{m, ∞} ∈ \cllim(\K)$.
		Hence, $v_{n, ∞} ∘ f ∈ \A_0$ for every $n$, and so $f ∈ \clabs(\A_0)$.
	\end{proof}
\end{lemma}

Note that it is not guaranteed in general that $σ^∀\K$ contains identities or that $σ^∃\K$ is closed under composition or is locally closed (see Example~\ref{ex:non_sigma_consistent}).
The key relevant property is $σ$-consistency of $\K$.

\begin{proposition} \label{thm:AE_sigma_consistent}
	For MU-categories $\K ⊆ \L$ we have that $\K$ is $σ$-consistent if and only if $σ^∀\K = σ^∃\K$, and in that case we have
	\begin{enumerate}
		\item $σ\K ∩ \L(\K, \K) = \clo{\K}$, i.e. $\K$ is locally dense in $σ\K$,
		\item $σ\K ∩ \L(σ\K, \K) = \clo{\cllim(\K)} = \A_0 ∩ \L(σ\K, \K)$,
		\item $σ\K = σ^∀\K = σ^∃\K = \clabs(\A_0) = \A_1$.
	\end{enumerate}
	In particular, all considered definitions of $σ\K$ agree, and the iterative construction of $σ\K$ stops after the first step.
	
	\begin{proof}
		Note that by the characterization from Proposition~\ref{thm:transformations}~(i), $\K ⊆ \L$ is $σ$-consistent if and only if $\id(σ\K) ⊆ σ^∀\K$, and that by Lemma~\ref{thm:AE_sigma} we have $\id(σ\K) ⊆ σ^∃\K$.
		Hence, if $σ^∀\K = σ^∃\K$, then $\K$ is $σ$-consistent.
		
		On the other hand, suppose that $\K$ is $σ$-consistent and that $f\maps X \to Y$ is a $σ^∃\K$-map.
		This is witnessed by some $\K$-sequences $\tuple{X_*, u_*}$ and $\tuple{Y_*, v_*}$ with limits $\tuple{X, u_{*, ∞}}$ and $\tuple{Y, v_{*, ∞}}$.
		Let $\tuple{X'_*, u'_*}$ and $\tuple{Y'_*, v'_*}$ be some other $\K$-sequences with limits $\tuple{X, u'_{*, ∞}}$ and $\tuple{Y, v'_{*, ∞}}$ and let $n ∈ ω$ and $ε > 0$.
		Since $\K$ is $σ$-consistent, there is a $\K$-map $f'\maps Y_m \to Y'_n$ with $v'_{n, ∞} ≈_ε f' ∘ v_{m, ∞}$, and there is $δ > 0$ such that $f'$ is $\tuple{ε, δ}$-continuous.
		Then there is a $\K$-map $f''\maps X_l \to Y_m$ with $v_{m, ∞} ∘ f ≈_δ f'' ∘ u_{l, ∞}$, and there is $γ > 0$ such that $f''$ is $\tuple{δ, γ}$-continuous.
		Finally, there is a $\K$-map $f'''\maps X'_k \to X_l$ with $u_{l, ∞} ≈_γ f''' ∘ u_{k, ∞}$.
		Altogether,
		\[
			v'_{n, ∞} ≈_ε f' ∘ v_{m, ∞} ≈_ε f' ∘ f'' ∘ u_{l, ∞} ≈_ε f' ∘ f'' ∘ f''' ∘ u'_{k, ∞},
		\]
		which witnesses that $f ∈ σ^∀\K$.
		
		Next we show that if $σ^∀\K = σ^∃\K$, then it is also equal to $σ\K$.
		This will give (iii).
		By Lemma~\ref{thm:AE_sigma}, we have $\K ⊆ σ^∃\K ⊆ σ\K$ so it is enough to show that $σ^∀\K = σ^∃\K$ is closed under all the relevant closures.
		Again by Lemma~\ref{thm:AE_sigma}, $\id(σ\K) ∪ \cllim(\K) ⊆ σ^∃\K$, while $σ^∀\K$ is closed under composition and is locally closed.
		Finally, $\clabs(σ^∀\K) ⊆ σ^∃\K$.
		
		Claims (i) and (ii) follow from Lemma~\ref{thm:AE_sigma}~(i) and (ii) since we have $σ\K = σ^∀\K$ and since the inclusions $\clo{\K} ⊆ \clo{\cllim(\K)} ⊆ \A_0 ⊆ σ\K$ are clear.
	\end{proof}
\end{proposition}

\begin{corollary} \label{thm:F1}
	Let $\K ⊆ \L$ be MU-categories.
	\begin{enumerate}
		\item If $\tuple{\K, σ\K}$ satisfies (F1), then $\K ⊆ \L$ is $σ$-consistent.
			The other implication holds as well if $\tuple{\K, \L}$ satisfies (L1).
		\item $\tuple{\K, \L}$ satisfies (F1) if and only if $\K ⊆ \L$ is $σ$-consistent and $σ\K ⊆ \L$ is full.
	\end{enumerate}
	
	\begin{proof}
		If $\tuple{\K, σ\K}$ or $\tuple{\K, \L}$ satisfies (F1), then $\K$ is $σ$-consistent as observed in Observation~\ref{thm:sigma_consistent_necessary}.
		On the other hand, if $\tuple{\K, \L}$ satisfies (L1) and $\K$ is $σ$-consistent, then by Corollary~\ref{thm:absolute} every $σ\K$-limit $\tuple{X_∞, f_{*, ∞}}$ of a $\K$-sequence is also an $\L$-limit,
		and by Proposition~\ref{thm:AE_sigma_consistent} we have $σ\K = σ^∀\K$, and so every $σ\K$-map $h\maps X_∞ \to Y$ to a $\K$-object admits $ε$-factorizations as we can use the constant identity sequence for $Y$.
		
		In general we have $σ^∀\K ⊆ σ\K ⊆ \L(σ\K, σ\K)$.
		But $\K$ is $σ$-consistent if and only if $σ^∀\K = σ\K$ by Proposition~\ref{thm:AE_sigma_consistent}, $σ\K = \L(σ\K, σ\K)$ if and only if $σ\K ⊆ \L$ is full, and we can observe that $σ^∀\K = \L(σ\K, σ\L)$ if and only if $\tuple{\K, \L}$ satisfies (F1).
	\end{proof}
\end{corollary}

We have seen that $σ^∀\K$ is a category only if $\K$ is $σ$-consistent.
There is another construction, based on the factorization property, such that we always obtain a category.
It is a two-step construction that was used as the definition of the $σ$-closure in the projective Fraïssé theory by Panagiotopoulos and Solecki~\cite{PS22} (there the $σ$-closure is denoted by $\C^ω$ and is used for the category $\C$ of all finite connected graphs and connected epimorphisms, and formally only particular representatives of the limits are considered, so the definition is equivalent to the definition mentioned in Remark~\ref{rem:concrete_limits}).
The definition is also related to the notion of \emph{admissible extension} of a category of finite topological first-order structures by Irwin~\cite[Definition~2.7]{Irwin07}.

Again, we will demonstrate that with $σ$-consistency the definitions agree, while without $σ$-consistency we do not have a well-behaving notion of a Fraïssé limit (see Example~\ref{ex:non_sigma_consistent}).

\begin{definition}
Let $\K ⊆ \L$ be MU-categories.
	\begin{itemize}
		\item Let $\B_0$ be the family of all maps $h ∈ \L(σ\K, \K)$ for which (F1) holds,
			i.e. such that for every $\K$-sequence $f_*$ with an $\L$-limit $\tuple{\dom(h), f_{*, ∞}}$ and every $ε > 0$ there is $g ∈ \K$ such that $h ≈_ε g ∘ f_{n, ∞}$.
		\item Let $\B_1$ be the family of all $\L$-maps $h$ between $σ\K$-objects such that $\B_0 ∘ h ⊆ \B_0$.
	\end{itemize}
\end{definition}

\begin{lemma} \label{thm:B}
	In general we have the following.
	\begin{enumerate}
		\item $\B_0 ∩ \L(\K, \K) ⊆ \clo{\K}$,
		\item $σ^∀\K ∩ \L(σ\K, \K) ⊆ \B_0 ⊆ \clo{\cllim(\K)} ⊆ \A_0$,
		\item $\B_0 ∘ \B_0 ⊆ \B_0 = \clo{\B_0}$,
		\item $\B_1$ is a locally closed subcategory of $\L$ containing $\B_0$ and $σ^∀\K$.
	\end{enumerate}
	
	\begin{proof}
		For an $\L(σ\K, \K)$-map $h\maps X \to Y$, the only difference between $h ∈ \B_0$ and $h ∈ σ^∀\K$ is that in the former only the constant identity sequence is considered as an admissible representation of $Y$.
		Hence, $σ^∀\K ∩ \L(σ\K, \K) ⊆ \B_0$, and the rest of claims (i) and (ii) follows as in the proof of Lemma~\ref{thm:AE_sigma}.
		
		The fact that $\B_0$ is locally closed is clear from the definition.
		Next, we have 
		\[
			\B_0 ∘ \B_0 = (\B_0 ∩ \L(\K, \K)) ∘ \B_0 ⊆ \clo{\K} ∘ \B_0 ⊆ \clo{\K ∘ \B_0} ⊆ \clo{\B_0} = \B_0.
		\]
		We use (i), the continuity of composition, the clear fact that $\K ∘ \B_0 ⊆ \B_0$, and local closedness of $\B_0$.
		This concludes the proof of (iii).
		
		Claim~(iv):
		The fact that $\B_1$ contains identities, is closed under composition, and is a locally closed is easy to check from the definition.
		For the local closedness we use $\B_0 ∘ \clo{\B_1} ⊆ \clo{\B_0 ∘ \B_1}$.
		The inclusion $\B_0 ⊆ \B_1$ is by the definition equivalent to $\B_0 ∘ \B_0 ⊆ \B_0$.
		The inclusion $σ^∀\K ⊆ \B_1$, which is equivalent to $\B_0 ∘ σ^∀\K ⊆ \B_0$, follows analogously to Lemma~\ref{thm:AE_sigma}~(iii).
	\end{proof}
\end{lemma}

Note that $\B_0$ may not contain $\K$ or even identities of $\K$-objects.
However, $\cllim(\K) ⊆ \B_0$ is equivalent to $\K$ being $σ$-consistent.

\begin{proposition} \label{thm:B_ids}
	If $\id(\K) ⊆ \B_0$, then $\B_1$ is a locally closed subcategory of $\L$ containing $\K$ such that
	\begin{enumerate}
		\item $\B_1 ∩ \L(\K, \K) = \B_0 ∩ \L(\K, \K) = \clo{\K}$, so $\K$ is locally dense in $\B_1$,
		\item $\B_1 ∩ \L(σ\K, \K) = \B_0$.
	\end{enumerate}
	If we have even $\cllim(\K) ∩ \L(\K, \K) ⊆ \B_0$ (in particular, if $\K$ contains all $\L$-isomorphisms between its objects and if every $\K$-sequence with an $\L$-limit object in $\K$ eventually consists of isomorphisms), then additionally $\B_0 ⊆ σ^∀\K$, and so we have
	\begin{enumerate}[resume]
		\item $\clo{\K} ⊆ \B_0 ⊆ σ^∀\K ⊆ \B_1$,
		\item $\B_1 ∩ \L(σ\K, \K) = σ^∀\K ∩ \L(σ\K, \K) = \B_0$.
	\end{enumerate}
	
	\begin{proof}
		If $\id(\K) ⊆ \B_0$, then $\K ⊆ \K ∘ \B_0 ⊆ \B_0$.
		By Lemma~\ref{thm:B}, we have $\clo{\K} ⊆ \B_0$ since $\B_0$ is locally closed, and $\B_0 ∩ \L(\K, \K) = \clo{\K}$.
		Next we have 
		\[
			\B_1 ∩ \L(σ\K, \K) = \id(\K) ∘ \B_1 ⊆ \B_0 ∘ \B_1 ⊆ \B_0,
		\]
		and so $\B_1 ∩ \L(σ\K, \K) ⊆ \B_0$.
		The rest of (i) and (ii) follows easily.
		
		For the next part note that if a $\K$-sequence $f_*$ whose $\L$-limit is a $\K$-object $K$ eventually consists of isomorphisms, we have that some $f_{n, ∞}$ is an isomorphism, and if it is further in $\K$, we have $f_{n', ∞} = f_{n', n} ∘ f_{n, ∞} ∈ \K ⊆ \B_0$ for every $n' ≤ n$.
		So we indeed have $\cllim(\K) ∩ \L(\K, \K) ⊆ \B_0$ under the particular extra assumptions.
		
		Finally we show $\B_0 ⊆ σ^∀\K ⊆ \B_1$ – the rest of (iii) follows from (i) and Lemma~\ref{thm:B}, and (iv) follows directly from (ii) and (iii).
		For a $\B_0$-map $h\maps X_∞ \to Y_∞$ with $\tuple{X_∞, f_{*, ∞}}$ and $\tuple{Y_∞, g_{*, ∞}}$ being $\L$-limits of $\K$-sequences $f_*$ and $g_*$, and for $n ∈ ω$ and $ε > 0$ we have $g_{n, ∞} ∈ \cllim(\K) ∩ \L(\K, \K) ⊆ \B_0$, and so $g_{n, ∞} ∘ h ∈ \B_0$ and there is a $\K$-map $h'\maps X_m \to Y_n$ with $h' ∘ f_{m, ∞} ≈_ε g_{n, ∞} ∘ h$, which by Proposition~\ref{thm:transformations}~(i) witnesses that $h ∈ σ^∀\K$.
	\end{proof}
\end{proposition}

\begin{corollary} \label{thm:B_sigma}
	If $\K ⊆ \L$ is $σ$-consistent, then $\B_0 = \clo{\cllim(\K)}$ and $\B_1 = σ\K$.
	
	\begin{proof}
		From the $σ$-consistency we have $σ\K = σ^∀\K = σ^∃\K$ by Proposition~\ref{thm:AE_sigma_consistent}, and we have $\cllim(\K) ⊆ \B_0$, and so $\B_0 ⊆ σ\K ⊆ \B_1$ by Proposition~\ref{thm:B_ids}.
		
		We show that $\B_1 ⊆ \clabs(\B_0)$, and so $\B_1 ⊆ σ\K$.
		For any $\B_1$-map $h\maps X \to Y$ we can pick a $\K$-sequence $g_*$ with $\L$-limit $\tuple{Y, g_{*, ∞}}$.
		Since $\cllim(\K) ⊆ \B_0$, for every $n ∈ ω$ we have $g_{n, ∞} ∈ \B_0$ and so $g_{n, ∞} ∘ h ∈ \B_0$.
		Hence, $h ∈ \clabs(\B_0)$ and we are done.
		
		Finally we have $\B_0 = \B_1 ∩ \L(σ\K, \K) = σ\K ∩ \L(σ\K, \K) = \clo{\cllim(\K)}$.
	\end{proof}
\end{corollary}

Without $σ$-consistency we may obtain various pathologies.
We give examples even of ordinary categories (i.e. having the discrete MU-structure).

\begin{example}
	Let $\K = \set{f_{i, j}: i ≤ j \in \omega}$ be a category consisting of a single inverse sequence $\tuple{K_*, f_*}$.
	Let $\L ⊇ \K$ be the category that adds a cone $\tuple{K_0, g_*}$ for $\tuple{K_*, f_*}$, i.e. $g_0 = \id_{K_0}$ and $g_i = f_{i, j} ∘ g_j$ for every $i ≤ j$, and $\L = \K \cup \set{e_{i, j} = g_i ∘ f_{0, j}: i, j ∈ ω}$.
	
	Then $\tuple{K_0, g_*}$ becomes the limit of $\tuple{K_*, f_*}$, and we have $\id(\K) ⊆ \B_0$, but not $\cllim(\K) ∩ \L(\K, \K) ⊆ \B_0$.
	Moreover, we have
	\[
		σ^∀\K ⊊ \B_0 = \B_1 = \K ⊊ \cllim(\K) = σ^∃\K ⊊ \clcat(σ^∃\K) = σ\K = \L,
	\]
	so the extra assumption in Proposition~\ref{thm:B_ids} is needed.
	
		It is easy to check that we have
		\[
			e_{i, j} ∘ e_{j, k} = f_{i, j} ∘ e_{j, k} = e_{i, j} ∘ f_{j, k} = e_{i, k}.
		\]
		for every $i, j, k ∈ ω$ (with $i ≤ j$ and $j ≤ k$ for the equalities involving $f_{i, j}$ and $f_{j, k}$, respectively).
		Moreover, we have 
		\[
			\L(K_j, K_i) = \begin{cases}
				\set{f_{0, j} = e_{0, j}} & \text{if } i = 0, \\
				\set{f_{i, j} ≠ e_{i, j}} & \text{if } 0 < i ≤ j, \\
				\set{e_{i, j}} & \text{if } i > j.
			\end{cases}
		\]
		It follows that $\L$-cones for $f_*$ are exactly $\tuple{K_n, e_{*, n}} = \tuple{K_0, g_*} ∘ f_{0, n}$ for $n ∈ ω$.
		Since $f_{0, n}$ is the unique $\L$-map $K_n \to K_0$, we have that $\tuple{K_0, g_*}$ is the limit of $f_*$.
		
		To show $\id(\K) ⊆ \B_0$, then only nontrivial case to check is whether $\id_{K_0}$ factorizes through $g_*$, but for every $n$ we have $f_{0, n} ∘ g_n = g_0 = \id_{K_0}$, so we are done.
		On the other hand, $g_n ∈ \cllim(\K) ∩ \L(\K, \K) \setminus \B_0$ for $n > 0$ since it does not factorize through the constant sequence $\tuple{\id_{K_n}}_*$ by a $\K$-map.
		
		Since $\L(\K, \K) = \L$, by Proposition~\ref{thm:B_ids}~(i), Lemma~\ref{thm:AE_sigma}, Lemma~\ref{thm:B}, and by the previous paragraph we have $σ^∀\K ⊆ \B_0 = \B_1 = \K ⊊ \cllim(\K) ⊆ σ^∃\K$.
		Since $\cllim(\K) ⊇ \K ∪ \set{g_n: n ∈ ω}$, we have $\cllim(\K) ⊊ \clcat(\cllim(\K)) = \L$.
		It remains to show that $\K ⊈ σ^∀\K$ and that $σ^∃\K ⊆ \cllim(\K)$.
		
		To compute $σ^∀\K$ and $σ^∃\K$ we analyze sequences and transformations in $\K$.
		Since $\K$-sequence is eventually an identity, or is a subsequence of $f_*$, it is enough to consider the trivial sequences $\tuple{\id_{K_n}}_*$, $n ∈ ω$, and $f_*$ with their unique limit cones $\tuple{\id_{K_n}}_*$ and $g_*$, respectively.
		There is no transformation $\tuple{\id_{K_n}}_* \to f_*$ as for $i > n$ there is no $\K$-map $K_n \to K_i$, but $f_{0, n}$ is a $\K$-map between the limits, so $f_{0, n} ∈ \K \setminus σ^∀\K$.
		Clearly, transformations $\tuple{\id_{K_n}}_* \to \tuple{\id_{K_m}}_*$ produces only $\K$-maps as limits; transformations $f_* \to f_*$ produce only $\id_{K_0}$ as this is the unique $\L$-map $K_0 \to K_0$, and the unique $f_* \to \tuple{\id_{K_n}}_*$ produces $g_n$ as the limit.
		Altogether, $σ^∃\K = \K ∪ \set{g_n: n ∈ ω} = \cllim(\K)$.
\end{example}

\cdef \odd {_\text{odd}}
\cdef \even {_\text{even}}

\begin{example} \label{ex:non_sigma_consistent}
	We give an example of a classical Fraïssé class and its discrete free completion $\tuple{\K, \L}$ (see Remark~\ref{rm:classical_fraisse}) and of a Fraïssé category $\F ⊆ \K$ such that $\F$ is not $σ$-consistent and such that the limit of a Fraïssé sequence is not homogeneous and does not have the extension property.
	Furthermore, $σ^∃\F$ is not closed under composition and so is not equal to $σ\F$.
	In fact, we have the following sequence of inclusions of closures:
	\begin{gather*}
		σ^∀\F = \B_0 = \F ⊊ \cllim[\F](\F) ⊊ \cllim[\F](\F) ∪ \id(σ\F) = \A_0, \\
			\A_0 ⊊ \clabs[\F](\A_0) = σ^∃\F ⊊ \clcat(σ^∃\F) = \A_1 = σ\F ⊉ \B_1.
	\end{gather*}
	
	We let $\tuple{\K, \L}$ be the classes of all finite and countable linear orders, respectively, with all embeddings.
	(We are now in the injective setting so all morphisms have orientation opposite to the convention used in MU-categories.)
	$\K$ is one of the most classical examples of a Fraïssé class and its limit is $ℚ$.
	We let $\E ⊆ \K$ be the wide subcategory of all end-extensions, i.e., an embedding $e\maps K \to L$ of finite linear orders is in $\E$ if and only if $e\im{K}$ is an initial segment of $L$.
	It is easy to see that $\tuple{\E, σ\E}$ is a free completion with $σ\E$ adding one new isomorphism type – $ω$ – and that $σ\E$-maps are exactly end-extensions.
	A skeleton of $\E$ (a full subcategory with one representative of each isomorphism type) is $\set{e_i^j: i ≤ j < ω}$ where every $e_i^j$ denotes the inclusion $i ⊆ j$; a skeleton of $σ\E$ is then $\set{e_i^j: i ≤ j ≤ ω}$.
	Clearly, $\E$ is Fraïssé with $ω$ being the limit.
	
	We define a wide subcategory $\F ⊆ \K$ as a certain modification of $\E$.
	Let $\E\odd$, $\E\even$, $\F\odd$, $\F\even$ denote the full subcategories of $\E$ and $\F$ consisting of finite linear orders of odd and even lengths, respectively.
	We put $\F\odd = \E\odd$ and $\F\even = \E\even$, but we allow no $\F$-map from an $\F\even$-object to an $\F\odd$-object, and for the only allowed $\F$-map $f\maps K \to L$ from an $\F\odd$-object to an $\F\even$-object there is exactly one point $y ∈ L$ with $y < f\im{K}$ and $f\im{K}$ is an interval, i.e. $f$ is an “end-extension shifted by one”.
	Describing the skeleton, $\F$ consists of two sequences $\set{e_{2i}^{2j}: i ≤ j < ω}$ and $\set{e_{2i + 1}^{2j + 1}: i ≤ j < ω}$ of inclusions and of a “transformation” $\set{f_{2i + 1}^{2j}: i < j < ω}$ from the odd sequence to the even sequence consisting of the shifted end-extensions.
	
	Now we prove the stated properties of the example.
	As $\F$ is essentially a countable directed poset, it is a Fraïssé category.
	The even sequence is Fraïssé, but the odd sequence is not as the outgoing maps $f_{2i + 1}^{2j}$ are not absorbed.
	Nevertheless, the limit of any non-trivial $\F$-sequence is $ω$, which is rigid, but is not homogeneous and does not even have the extension property – the maps $e_{2i + 1}^ω$ cannot be extended along any $f_{2i + 1}^{2j}$.
	
	Let us more carefully describe how (the skeleton of) $σ\F$ is generated.
	The limit closure $\cllim[\F](\F)$ adds the inclusions $e_i^ω$ for $i ≤ ω$ as well as the maps $f_{2i + 1}^ω = e_{2j}^ω ∘ f_{2i + 1}^{2j}$ for $i < j < ω$ (note that $f_{2i + 1}^ω$ is a limit cone map for the sequence $\tuple{f_{2i + 1}^{2i + 2}, e_{2i + 2}^{2i + 4}, e_{2i + 4}^{2i + 6},…}$).
	Since $\cllim[\F](\F)$ is already closed under composition and since the local closure is trivial in the discrete setting, we have $\A_0 = \cllim[\F](\F) ∪ \id(σ\F)$.
	Essentially, $\A_0$ adds a maximum to the skeleton of $\F$, which is a poset.
	
	Let $s\maps ω \to ω$ be the successor map.
	We have $s ∘ e^ω_{2i + 1} = e^ω_{2i + 2} ∘ f^{2i + 2}_{2i + 1} ∈ \A_0$ for every $i < ω$.
	Hence, $\clabs[\F](\A_0)$ adds the map $s$.
	On the other hand, this is the only map added: for a map $g\maps ω \to ω$ to be added, the restriction $g ∘ e^ω_i$ to the initial segment would be prescribed for infinitely many numbers $i < ω$ and would have to be equal to $e^ω_i$ or $f^ω_i$, forcing $g$ to be $\id_ω$ or $s$, respectively.
	Since clearly $s ∈ σ^∃\F$, we have $\clabs[\F](\A_0) = σ^∃\F$ as $\A_0 = \cllim[\F](\F) ∪ \id(σ\F) ⊆ σ^∃\F ⊆ \clabs[\F](\A_0)$ by Lemma~\ref{thm:AE_sigma}.
	
	We see that $σ^∃\F$ is not closed under composition, so $σ^∃\F ⊊ \clcat(σ^∃\F)$, which is equal to $\clcat(\clabs[\F](\A_0)) = \A_1$.
	In fact, we need to add the compositions $s^k$ and $s^k ∘ e^ω_i$ for $i, k < ω$, i.e. the skeleton of $\A_1$ with objects $\set{0, 1, …, ω}$ is $\set{e_i^j: i ≤ j < ω} ∪ \set{s^k ∘ e^ω_i: k < ω, i ≤ ω}$.
	Moreover, $\clabs[\F](\A_1) = \A_1$ and so $σ\F = \A_1$.
	To see that $\clabs[\F](\A_1)$ does not add any new maps, note that if for some $g\maps ω \to ω$ we have $g ∘ e^ω_i ∈ \A_1$ for every $i ∈ I$ and an infinite set $I ⊆ ω$, then we have $g ∘ e^ω_i = s^{k_i} ∘ e^ω_i$ for every $i ∈ I$.
	But for $i < j ∈ I$ we have $s^{k_j} ∘ e^ω_i = s^{k_j} ∘ e^ω_j ∘ e^j_i = g ∘ e^ω_j ∘ e^j_i = g ∘ e^ω_i = s^{k_i} ∘ e^ω_i$, and so $k_j = k_i =: k$ and $g = s^k$.
	
	Finally we show that $σ^∀\F = \B_0 = \F$ and so $\B_1$ contains all self-embeddings of $ω$ as there is no $\B_0$-map to $ω$.
	Since clearly every $\F$-sequence with $\L$-limit object in $\F$ eventually consists of isomorphisms, by Proposition~\ref{thm:B_ids} we have $\F = \B_0 ∩ \L(\F, \F)$ and $\B_0 = σ^∀\F ∩ \L(\F, σ\F)$.
	Every $\B_0$-map $g\maps n \to ω$ is of the form $e^ω_{2i + 1} ∘ g'$ with $g' ∈ \F$ since $\tuple{ω, e^ω_{2* + 1}}$ is a limit cone.
	Hence, $g' = e^{2i + 1}_n$, $n$ is odd, and $g = e^ω_n$.
	At the same time, $g$ is of the form $e^ω_{2j} ∘ g''$ for some $g'' ∈ \F$ since $\tuple{ω, e^ω_{2*}}$ is a limit cone.
	But since $n$ is odd, we have $g'' = f^{2j}_n$ and $g = f^ω_n ≠ e^ω_n$, so there is no such map $g$ to start with, i.e. we have $\B_0 ⊆ \L(\F, \F)$.
	Similarly, for a $σ^∀\F$-map $h\maps ω \to ω$ we have $h ∘ e^ω_0 ∈ \B_0$, and so no such map exists and $σ^∀\F ⊆ \L(\F, σ\F)$.
	Altogether, we have $σ^∀\F = \B_0 = \F$.
\end{example}

	\subsection*{Acknowledgements}
	We would like to thank the anonymous referee for suggestions that improved the overall presentation.
	We would also like to thank Kateřina Fuková for help with proofreading the new version of the paper.
	
	This version of the article has been accepted for publication, after peer review, but is not the Version of Record and does not reflect post-acceptance improvements, or any corrections.
	The Version of Record is available online at: \url{https://doi.org/10.1007/s00029-025-01119-5}.

	\bibliographystyle{mysiam}
	\bibliography{references}

\end{document}